\documentclass[11pt]{amsart}
\usepackage{amsthm}
\newtheorem{theorem}{Theorem}[section]
\newtheorem{lemma}[theorem]{Lemma}
\newtheorem{proposition}[theorem]{Proposition}
\theoremstyle{definition}
\newtheorem{definition}[theorem]{Definition}

\theoremstyle{remark}
\newtheorem{remark}[theorem]{Remark}
\newtheorem{corollary}[theorem]{Corollary}
\newtheorem{example}[theorem]{Example}
\counterwithin*{section}{part}

\usepackage{amssymb,mathrsfs}
\usepackage{empheq}
\usepackage{stmaryrd}
\usepackage{enumerate}
\usepackage[shortlabels]{enumitem}
\usepackage{calc}
\usepackage{url}

\usepackage{mathabx}

\usepackage{xcolor}

\usepackage{mathtools}
\DeclarePairedDelimiter{\ceil}{\lceil}{\rceil}

\newcommand{\seq}[1]{\{{#1}\}}

\def\NN{\mathbb{N}}

\def\QQ{\mathbb{Q}}
\def\RR{\mathbb{R}}

\newcommand{\WEPAomega}{\textsf{WE}\mbox{-}\textsf{PA}^\omega}

\newcommand{\model}[1]{\llbracket {#1} \rrbracket }
\newcommand{\QFAC}{\textsf{QF}\mbox{-}\textsf{AC}}
\newcommand{\QFER}{\textsf{QF}\mbox{-}\textsf{ER}}
\newcommand{\DC}{\textsf{DC}}

\def\sigUB{\Sigma^0_1\mbox{-}\mathsf{UB}^{X}}

\usepackage[left=2cm,%
                right=2cm,%
                top=2.5cm,%
                bottom=3.5cm,%
                headheight=12pt,%
                a4paper]{geometry}%

\begin{document}

\title[Logical Metatheorems for Abstract Spaces in Positive
Bounded Logic II]{Logical Metatheorems for Abstract Spaces axiomatized in Positive
Bounded Logic II: Metric spaces and the model-theoretic uniformity principle}

\author[Ulrich Kohlenbach, Morenikeji Neri and Jin Wei]{Ulrich Kohlenbach${}^{\MakeLowercase a}$, Morenikeji Neri${}^{\MakeLowercase a}$ and Jin Wei${}^{\MakeLowercase a}$}
\date{\today}
\maketitle
\vspace*{-5mm}
\begin{center}
{\scriptsize 
${}^a$ Department of Mathematics, Technische Universit\"at Darmstadt,\\
Schlossgartenstra\ss{}e 7, 64289 Darmstadt, Germany,\\ 
E-mails:  kohlenbach@mathematik.tu-darmstadt.de, neri@mathematik.tu-darmstadt.de, wei@mathematik.tu-darmstadt.de}
\end{center}

\maketitle
\begin{abstract}
We extend the proof-theoretic treatment of uniform bound extraction from normed structures axiomatized in positive bounded logic [\emph{Advances in Mathematics}, 290:503--551, 2016] (as developed for the model theory of Banach spaces) to the more general setting of abstract metric structures, including discrete structures viewed as classical first-order models. In particular, we establish uniform bound extraction theorems for our generalized framework for $\forall\exists$-sentences whose matrix is the negation of (an embedding of) a formula in positive bounded logic, whose proofs use saturation. In this way, we provide a formal explanation for the successes in the extraction of uniform bounds from nonstandard proofs given in [\emph{Advances in Mathematics}, 343:567--623, 2019], which had informally followed the perspective of the monotone functional interpretation. As an application of the formal framework we develop, we provide novel explicit bounds for a structural theorem for stable subsets of groups given in [\emph{Mathematical Proceedings of the Cambridge Philosophical Society}, 168(2):405--413, 2020].
\end{abstract}
\section{Introduction}
During the last decades a novel applied
form of proof theory has emerged which, often under the label of `proof mining',
utilizes proof-theoretic transformations in the spirit of G\"odel's functional
(`Dialectica') interpretation to extract effective uniform bounds from proofs
in many areas of analysis (\cite{Kohlenbach2008,kohlenbach:19:nonlinear:icm}): nonlinear analysis, fixed point theory, ergodic
theory, nonsmooth optimization, PDE theory and, more recently (see e.g. \cite{NeriOlivaPischke2026,neri-pischke:24:formal:pub}), probability
theory. While this program started around 1990 (\cite{Kohlenbach(90)}), a main new development began
with the introduction of abstract metric and normed spaces as a kind of
atoms (`base types') to the formal systems used which allowed one to obtain
bounds which were uniform in a way which in some cases was new even qualitatively (\cite{kohlenbach2005some,GeK2008}). Already in \cite{kohlenbach2005some}, the
first author pointed to some similarities between this proof-theoretic
approach to uniform bounds and the role of `uniform $L$-structures' and their
ultrapowers in the context of the model theory of Banach spaces as developed
in \cite{henson2002ultraproducts}. For instance structures similar to those
featured in the model-theoretic approach are also central in the
proof-theoretic metatheorems from \cite{kohlenbach2005some,GeK2008}
and the assumption of the boundedness (on bounded sets) for all
function symbols with some common modulus
is a particular case of a uniformly majorized class of objects with
majorizability being a key property in the proof-theoretic context.
The assumption of equicontinuity of all the function symbols made in
\cite{henson2002ultraproducts} has the proof-theoretic consequence that
the full extensionality property of these functions becomes derivable.
This is convenient but not mandatory in the proof-theoretic framework, where one frequently deals with situations where the functions in question are
discontinuous (which then allows only for some restricted use of
extensionality which, however, usually is all what is used in given proofs).
A striking
difference between the model-theoretic versus the proof-theoretic approach
to uniform bounds is that in the former framework the uniformity is
established by extending a given structure to a large ultrapower, whereas
in the proof-theoretic framework it is a consequence of the fact that a given
proof can only make uniform queries about the objects from
the metric structure by majorizable functionals (as all terms which occur
in the proof analysis are majorizable), i.e. the reasoning takes place
in some small term model.\\

The theme of elaborating on connections between the two paradigms was taken up more
systematically in \cite{GuK2016}. That paper established that all normed
structures which are axiomatizable in positive bounded logic admit proof-theoretic uniform bound extraction metatheorems (as in \cite{kohlenbach2005some,GeK2008}). E.g. this was used to establish such metatheorems for
abstract $C(K)$- and $L^p$-spaces as well as for $BL^pL^q$-Banach lattices.
Furthermore,
the proof-theoretically tame nonstandard uniform bounded principle introduced
in \cite{Kohlenbach(06)} could be seen to a certain extent as a constructive substitute
for the (noneffective) use of ultramethods made in the model-theoretic
context as well as for the use of ultrapowers of Banach spaces made
in functional analysis.
\\[1mm]

While already \cite{kohlenbach2005some,GeK2008} not only treat normed
structures but also arbitrary metric spaces (as well as several intermediate
classes of geodesic spaces), the specific connection to model theory as
investigated in \cite{GuK2016} only considered normed structures as only
those were covered in \cite{henson2002ultraproducts}. \\[1mm]

On the logical side,
our current paper
\begin{itemize}
\item
extends \cite{GuK2016} to treat the case of abstract metric spaces
axiomatized in positive bounded logic
\cite{avigad2013ultraproducts,DI2017}. In particular, this subsumes
first-order logic by viewing classical structures as discrete metric spaces,
thereby giving access to theories from algebra as well as \emph{hybrid}
structures with both discrete and continuous parts, such as measure
constructions on groups;
\item
generalizes the uniform boundedness principles studied in \cite{GuK2016} to
a proof-theoretic treatment of saturation;
\item
establishes uniform bound extraction theorems for our generalized framework
for $\forall\exists$-sentences whose matrix is the negation of (an embedding
of) a formula in positive bounded logic rather than only a
$\exists n^{\NN}$-formula.
\end{itemize}
While the results from \cite{GuK2016} carry over more or less directly
to metric structures provided that the metric spaces have enough
intermediate points (which e.g. is the case for the hyperbolic spaces
treated already in \cite{kohlenbach2005some,GeK2008}), to include also
discrete metric spaces requires some important changes, e.g., similar
to \cite{DI2017}, one has to replace closed balls by open balls in the
definition of approximate satisfiability for universal quantifiers, thereby
restricting the class of positive bounded formulas. In our framework we
have to replace the extensional retraction function used in \cite{GuK2016}
to restrict quantification to bounded balls by an intensional one (which,
moreover, avoids the need of an extensionality assumption made in
 \cite{GuK2016}). \\[1mm]
 
With respect to applications to mathematics we use our extended framework to

\begin{itemize}
\item
  generalize a structural result for stable subsets of finite groups
  which was obtained using ultra-methods in \cite{conant2020group}
  and extract
  explicit effective bounds for this theorem;
\item
  interpret the proof-theoretic extraction of a bound for the
  Avigad-Dean-Rute-Tao meta-\\ stable dominated convergence theorem
  from \cite{Avigad-Dean-Rute:Dominated:12} as an instance of our
  new logical metatheorem thereby complementing results in
  \cite{neri-pischke:24:formal:pub}.
\end{itemize}

\subsection{Organisation of the paper and summary of main results}
\label{subsec:organisation}

We now provide further details and context for the main contributions of this paper. 
\subsubsection{A proof-theoretic treatment of the model-theoretic approach to uniform bounds}

In \cite{avigad2013ultraproducts}, it is shown that model-theoretic and ultraproduct-based techniques can be used to establish qualitative uniformity results corresponding to the quantitative ones extracted proof-theoretically. At the core of the model-theoretic approach is the following \emph{uniformity principle}.

\begin{theorem}[Uniformity principle]
Let $L$ be a first-order signature, let $\mathscr{M}$ be a family of $L$-structures, for each $n \in \NN$ let $\varphi_{n}$ be an $L$-formula. If, for every $L$-structure $\mathcal{M}$ that is an ultraproduct of structures in $\mathscr{M}$ with respect to a non-principal ultrafilter,
\[
  \exists n \in \NN\; (\mathcal{M} \vDash \varphi_{n}),
\]
then 
\[
  \exists N \in \NN\; \forall \mathcal{M} \in \mathscr{M}\; \exists n \leq N\; (\mathcal{M} \vDash \varphi_{n}).
\]
\end{theorem}
An immediate consequence of this theorem is one in which the first order formulas you consider take parameters, which are not necessarily from the language $L$, but can be from an arbitrary set.

\begin{theorem}[Uniformity principle with parameters]
Let $L$ be a first-order signature, let $\mathscr{M}$ be a family of $L$-structures, let $X$ be an arbitrary set, and for each $x \in X$ and $n \in \NN$ let $\varphi_{x,n}$ be an $L$-formula. If, for every
$L$-structure $\mathcal{M}$ that is an ultraproduct of structures in $\mathscr{M}$ with respect to a non-principal ultrafilter,
\[
  \forall x \in X\; \exists n \in \NN\; (\mathcal{M} \vDash \varphi_{x,n}),
\]
then
\[
  \forall x \in X\; \exists N \in \NN\; \forall \mathcal{M} \in \mathscr{M}\; \exists n \leq N\; (\mathcal{M} \vDash \varphi_{x,n}).
\]
\end{theorem}

In other words, if one can show that every ultraproduct of structures in a given family satisfies a $\forall\exists$ statement, then one obtains a bound on the existential witness that is uniform across the entire family. An analogous result holds in the context of positive bounded logic and metric structures; for simplicity, we have stated only the first-order (discrete) case above.

Although the uniformity principle is strikingly general, its typical application to establish uniform bounds in concrete mathematics has a much more specific character. In practice:
\begin{enumerate}
  \item The parameter set $X$ is concrete, and it is sensible to talk about computable functions taking values in $\NN$ on $X$. Typically $X$ is $\NN$ or $\NN^\NN$.
  \item The assignment $(x,n) \mapsto \varphi_{x,n}$ is not arbitrary and
        can be formalised in some (higher-order) formal system.
  \item The structures in $\mathscr{M}$ are all models of a fixed theory
        $\mathcal{T}$.
  \item The only property of ultraproducts used to establish the premise of
        the uniformity principle, if any, is countable saturation.
\end{enumerate}
Thus, in typical applications, the uniformity principle is invoked after \emph{proving} a $\forall\exists$ theorem from the axioms of some fixed theory together with countable saturation. Given that both $X$ and the assignment $(x,n)\mapsto\varphi_{x,n}$ are concrete, as described in conditions (1)--(4), in typical applications, two natural questions arise:
\begin{enumerate}
  \item[(i)]  Can one extract an explicit function $f\colon X\to\NN$
              realising the bound $N$?
  \item[(ii)] Can one bound the complexity of such a function?
\end{enumerate}
The proof of the uniformity principle is a straightforward application of \L o\'s's theorem, but its non-constructive character makes it unclear how questions (i) and (ii) can be addressed for typical results obtained through it. Simmons and Towsner \cite{simmons2019proof} observed, however, that in many applications in algebra, one could bypass the uniformity principle entirely and instead directly analyse the proofs of the $\forall\exists$ theorems that appeal to it, following the proof mining perspective via the \emph{monotone functional interpretation} developed by the first author \cite{Kohlenbach2008}. This approach enabled the authors of \cite{simmons2019proof} to obtain explicit quantitative results for theorems in polynomial and differential polynomial ring theory (from \cite{van1984bounds} and \cite{harrison2012nonstandard}) that had previously been established through the uniformity principle.

We show that, although the model-theoretic approach is very general, its typical application, as described in conditions (1)--(4) above, can be treated proof-theoretically, via our formal approach of treating abstract metric spaces axiomatised in positive bounded logic and corresponding metatheorem, and in particular allows one to address questions (i) and (ii). Our proof-theoretic treatment of the uniformity principle is set in the context of metric structures axiomatised by positive bounded logic. For simplicity, we describe the first-order (discrete) structures: For a first-order signature $L$ and an $L$-theory $\mathcal{T}$, we construct a formal system $\mathcal{A}^\omega[L,\mathcal{T}]$ (extending a weakly extensional variant of Peano arithmetic in all finite types) and through an appropriate embedding of first-order formulas (and, more generally, positive bounded formulas) into these higher-type system, $\mathcal{A}^\omega[L,\mathcal{T}]$ contains the axioms of $\mathcal{T}$; furthermore, many instances of saturation can be formalised within this system via a schema we call $\mathrm{Sat}$.

These developments culminate in our main result: a proof-theoretic analogue of the uniformity principle. Informally, it asserts the following. Let $\Theta(x,n)$ be a suitable formalisation of a parameterised formula, corresponding to $\varphi_{x,n}$ above; in line with condition~(2), we assume
the assignment $(x,n)\mapsto\varphi_{x,n}$ is concrete enough to be captured within the formal system via $\Theta(x,n)$. In line with condition~(1), we take $X = \NN$ for this exposition (in the full theorem, $x$ ranges over a class of well-behaved higher types). Suppose that
\[
  \mathcal{A}^\omega[L,\mathcal{T}] + \mathrm{Sat}
  \;\vdash\;
  \forall x\;\exists n\;\Theta(x,n).
\]
Then, from the proof, one can extract a total computable function $\Phi\colon\NN\to\NN$ (computable by bar recursion) such that the following holds in every non-trivial model of $\mathcal{T}$: for all $x, x^{*}\in\NN$, if $x^{*} \geq x$, then
\[
  \exists n \leq \Phi(x^{*})\;\Theta(x,n).
\]
Section~\ref{sec:model} introduces the model-theoretic background. We begin in Section~\ref{subsec:metric} with an introduction to positive bounded logic and a proof of the uniformity principle. The remainder of Section~\ref{sec:model} is devoted to two applications of the uniformity
principle that are amenable to our formal perspective and bound extraction theorem, which we present in Section~\ref{sec:pt}.

\subsubsection*{First application: stable subsets of groups}

The first application, developed in Section~\ref{subsec:model:APP1}, comes
from group theory. Let $(G,\cdot,e)$ be a group and $k \ge 1$. A subset
$A \subseteq G$ is said to have the \emph{$k$-order property} if there exist
elements $a_1,\ldots,a_k$ and $b_1,\ldots,b_k$ in $G$ such that
\[
  \forall i,j \le k \,(a_i \cdot b_j \in A \leftrightarrow i \le j).
\]
$A$ is \emph{$k$-stable} if it does not have the $k$-order
property, and \emph{stable} if it is $k$-stable for some $k \ge 1$.

In \cite{conant2020group}, Conant established the following structural result
for stable subsets of finite groups, generalising earlier work of Terry and
Wolf \cite{terry2019stable}. Unlike Terry--Wolf, Conant's proof proceeds via
an ultraproduct argument and yields no effective bounds.%
\footnote{Terry and Wolf were able to give explicit bounds for their version
of Theorem~\ref{thrm:1}.}

\begin{theorem}[cf.\ Theorem~1.3 of \cite{conant2020group}]
\label{thrm:1}
For $k \ge 1$ and $\varepsilon > 0$, there exists $n = n(k,\varepsilon)$
with the following property. If $G$ is a finite group and $A \subseteq G$ is
$k$-stable, then there exist a subgroup $H \le G$ of index at most $n$ and a
subset $Y \subseteq G$ that is a union of cosets of $H$, such that
$|A \mathbin{\triangle} Y| \le \varepsilon|H|$.
\end{theorem}

A similar noneffective argument in \cite{conant2020group} yields the following
strengthening.

\begin{theorem}[cf.\ Remark~3.2 of \cite{conant2020group}]
\label{thrm:meta:main}
For $k \ge 1$ and $f:\NN \to \NN$, there exists $n = n(k,f)$ with the
following property. If $G$ is a finite group and $A \subseteq G$ is $k$-stable,
then there exist a subgroup $H \le G$ of index at most $n$ and a subset
$Y \subseteq G$ that is a union of cosets of $H$, such that
$|A \mathbin{\triangle} Y| \le 2^{-f(n)}|G|$.
\end{theorem}

Applying the uniformity principle for positive bounded logic, we obtain the
following strengthening of Theorem~\ref{thrm:meta:main}, which passes from
finite groups to groups equipped with a bi-invariant finitely additive
probability measure.

\begin{theorem}
\label{thrm:meta:main:strong}
For $k \ge 1$ and $f:\NN \times \NN \to \NN$, there exist $n = n(k,f)$ and
$M = M(k,f)$ with the following property. Let $G$ be a group, $A \subseteq G$
a $k$-stable subset, $\mathcal{F}$ a bi-invariant Boolean algebra on $G$
containing $A$, and $\mu$ a bi-invariant finitely additive probability measure
on $\mathcal{F}$. Then there exist $m \le M$ and $H \in \mathcal{F}$ such that
\[
  \mathrm{Stab}_m^{<}(A) = H = \mathrm{Stab}_m(A),\mbox{ } H \le G,\mbox{ } [G:H] \le n,
\]
and there exists a subset $Y \subseteq G$ that is a union of cosets of $H$,
such that $\mu(A \mathbin{\triangle} Y) \le 2^{-f(n,m)}$, where
\[
  \mathrm{Stab}_m^{<}(A) := \{ g \in G \mid \mu(Ag \mathbin{\triangle} A) < 2^{-m} \} \mbox{ and }
  \mathrm{Stab}_m(A) := \{ g \in G \mid \mu(Ag \mathbin{\triangle} A) \le 2^{-m} \}.
\]
\end{theorem}

Our proof of Theorem~\ref{thrm:meta:main:strong} is inspired by the
quantitative finitary arguments of Conant~\cite{conant2021quantitative}, who
obtained explicit bounds for Theorem~\ref{thrm:1}. Following the formal
perspective of Section~\ref{sec:pt}, we extract explicit bounds for
Theorem~\ref{thrm:meta:main:strong} in Section~\ref{sec:quant:main}. We
stress that, although the formal perspective of Section~\ref{sec:pt} was
instrumental in guiding the quantitative arguments, the proofs and lemmas in
Section~\ref{sec:quant:main} make no reference to proof theory or model
theory, and that section can be read as a purely mathematical result
independently of Section~\ref{sec:pt}.

\subsubsection*{Second application: the metastable dominated convergence theorem}

The second application, developed in Section~\ref{subsec:model:APP2}, comes
from probability theory. Recall the Lebesgue dominated convergence theorem
for doubly indexed sequences.

\begin{theorem}
\label{thrm:DCT}
Let $(\Omega,\mathcal{F},\mu)$ be a probability space and let
$f_{n,m}\colon\Omega\to[0,1]$ be measurable for each $n,m \in \NN$. If
$\lim_{n,m\to\infty}f_{n,m}(\omega) = 0$ for almost every $\omega\in\Omega$,
then $\lim_{n,m\to\infty}\int_\Omega f_{n,m}\,d\mu = 0$.
\end{theorem}

In \cite{tao:08:ergodic}, Tao studied quantitative forms of this result. The
most natural quantitative interpretation of convergence for a doubly indexed
sequence $\{x_{n,m}\}$ is a \emph{rate of convergence}, namely a function
$\phi:\QQ^+\to\NN$ satisfying
\[
  \forall \varepsilon \in \QQ^+\,\forall i,j \ge \phi(\varepsilon)\,(|x_{i,j}| \le \varepsilon).
\]
For many convergence theorems in analysis, however, no uniform rate of
convergence exists. A well-known example is the monotone convergence theorem:
monotone sequences in $[0,1]$ can converge arbitrarily slowly, and this
obstruction persists even for computable sequences of rationals
\cite{specker:49:sequence}. In such situations, one can often obtain uniform bounds on the equivalent
\emph{metastable} reformulation:
\begin{equation}
\label{eqn:metastable}
  \forall \varepsilon \in \QQ^+\,\forall g:\NN\to\NN\,\exists n\,
  \forall i,j\in[n;\,n+g(n)]\,(|x_{i,j}|\le\varepsilon).
\end{equation}
A \emph{rate of metastability} is a functional $\Phi$ bounding $n$
in~\eqref{eqn:metastable}:
\[
  \forall \varepsilon \in \QQ^+\,\forall g:\NN\to\NN\,\exists n \le \Phi(\varepsilon,g)\,
  \forall i,j\in[n;\,n+g(n)]\,(|x_{i,j}|\le\varepsilon).
\]
The term \emph{metastability} originates with Tao~\cite{tao:07:softanalysis}.
Rates of metastability were computed earlier by the first
author~\cite{kohlenbach:meta:first,kohlenbach2004bounds} using tools from
proof theory, prior to Tao coining the term (see the survey
\cite{kohlenbach:19:nonlinear:icm} for a historical account of metastability
and the proof-mining programme as of 2018). From a logical standpoint,
metastability~\eqref{eqn:metastable} is the Herbrand normal form of
convergence, and a rate of metastability provides a solution to
Kreisel's no-counterexample interpretation of the convergence
statement~\cite{kreisel:51:proofinterpretation:part1,
kreisel:52:proofinterpretation:part2}. Logical metatheorems of the kind
presented in this paper guarantee the extractibility of uniform rates of
metastability for large classes of proofs.

Tao showed that rates of metastability for the convergence of the integrals
in the dominated convergence theorem exist and depend only on a suitable
quantitative reading of the almost sure convergence of the integrands, namely
a functional $\Psi$ satisfying
\[
  \forall \varepsilon \in \QQ^+\,\forall g:\NN\to\NN\,
  \Bigl(\mu\bigl(\{\omega\in\Omega\mid \exists n \le \Psi(\varepsilon,g)\,
  \forall i,j\in[n;\,n+g(n)]\,(f_{i,j}(\omega)\le\varepsilon)\}\bigr)=1\Bigr).
\]
Avigad, Dean, and Rute~\cite{Avigad-Dean-Rute:Dominated:12} subsequently
extended Tao's result in two directions: they weakened the quantitative
assumption on almost sure convergence, and they gave an explicit construction
of the functional transforming the quantitative premise into a rate of
metastability for the integrals.

\begin{theorem}[Avigad--Dean--Rute--Tao metastable dominated convergence theorem]
\label{thrm:avigad:dominated}
For every functional $\Psi\colon\QQ^+\times\QQ^+\times\NN^\NN\to\NN$ there
exists a functional $\Phi\colon\QQ^+\times\NN^\NN\to\NN$ such that the
following holds. Let $(\Omega,\mathcal{F},\mu)$ be a probability space and
$f_{n,m}\colon\Omega\to[0,1]$ measurable for each $n,m\in\NN$. If
\[
  \forall \varepsilon,\lambda \in \QQ^+\,\forall g:\NN\to\NN\,
  \Bigl(\mu\bigl(\{\omega\in\Omega\mid \exists n \le \Psi(\varepsilon,\lambda,g)\,
  \forall i,j\in[n;\,n+g(n)]\,(f_{i,j}(\omega)\le\varepsilon)\}\bigr)
  > 1-\lambda\Bigr),
\]
then $\Phi$ is a rate of metastability for $\bigl\{\int_\Omega
f_{n,m}\,d\mu\bigr\}$, i.e.\
\[
  \forall \varepsilon \in \QQ^+\,\forall g:\NN\to\NN\,
  \exists n \le \Phi(\varepsilon,g)\,\forall i,j\in[n;\,n+g(n)]\,
  \Bigl(\Bigl|\int_\Omega f_{i,j}\,d\mu\Bigr|\le\varepsilon\Bigr).
\]
\end{theorem}

\begin{remark}
A functional $\Psi\colon\QQ^+\times\QQ^+\times\NN^\NN\to\NN$ satisfying the
premise above is called a \emph{rate of pointwise metastable convergence};
possessing such a rate is equivalent to almost sure convergence
(cf.\ Theorem~3.2 of \cite{neri-powell:pp:martingale}).
\end{remark}

In \cite{DI2017}, essentially an application of the uniformity principle for positive bounded logic yields a  weaker version of
Theorem~\ref{thrm:avigad:dominated}, under the stronger assumption that the
rate of metastability is uniform across the entire sample space, i.e.\ that
there exists $\Phi$ satisfying
\[
  \forall \varepsilon \in \QQ^+\,\forall g:\NN\to\NN\,\forall\omega\in\Omega\,
  \exists n \le \Phi(\varepsilon,g)\,\forall i,j\in[n;\,n+g(n)]\,
  (f_{i,j}(\omega)\le\varepsilon).
\]
We show in Section~\ref{subsec:model:APP2} that the full Theorem~\ref{thrm:avigad:dominated} follows from the uniformity principle for
positive bounded logic. The main challenge is to formalise the assumption of pointwise metastable convergence in positive bounded logic. To do this, we crucially utilise the \emph{outer measure} approach from the proof-theoretic work of the second author, Oliva, and Pischke \cite{NeriOlivaPischke2026}.

We finish by noting that the extraction of explicit bounds for
Theorem \ref{thrm:avigad:dominated} was carried out in
\cite{Avigad-Dean-Rute:Dominated:12} via proof-theoretically motivated
methods. A formal explanation for the success of that extraction was given
by the second author and Pischke in \cite{neri-pischke:24:formal:pub}, and instantiating the metatheorem of Section~\ref{sec:pt} in this context recovers this result.

\section{Model theory for metric structures}
\label{sec:model}

\subsection{Henson structures and positive bounded logic}
\label{subsec:metric}
Following \cite{DI2017} closely, we start by recalling the relevant parts of the theory of Henson structures.

\begin{definition}[cf.\ Definition 6.1 and Definition 6.2 of \cite{DI2017}]
A \emph{Henson signature} is a many-sorted signature consisting of the followings:
\begin{itemize}
    \item A family $(s|s\in \mathbf{S})$ of sorts, together with a distinguished element $s_{\mathbb{R}}\in \mathbf{S}$. 
    \item A family of function symbols $f$ including:
        \begin{itemize}
        \item Binary functional symbol $d_s$ with domain $s\times s$ and range $s_{\mathbb{R}}$ for each sort $s\in \mathbf{S}$.
        \item Constant symbol $a_s$ of sort $s$ for each $s\in \mathbf{S}$.
        \item Constant symbol $r$ of sort $s_{\mathbb{R}}$ for each rational number $r$.
         \item Constant symbols $+_\RR, -_\RR, \cdot_\RR, \lor_\RR, \land_\RR: s_{\RR} \times s_{\RR} \to s_{\RR}$ and $|\cdot|:s_{\RR} \to s_{\RR}$ representing addition, subtraction, multiplication, maximum, minimum, and absolute value, respectively.
        \end{itemize}
\end{itemize}

Fixing a Henson signature $L$, a \emph{Henson structure} $\mathcal{M}$ is a $L$-structure such that 
\begin{itemize}
    \item $s^{\mathcal{M}}$ is a pointed metric space $M^{s}$ where $d_s^{\mathcal{M}}$ is the metric and $a_s^{\mathcal{M}}$ is the anchor point of $M^{s}$. 
    \item $s_{\mathbb{R}}^{\mathcal{M}}$ is the set $\mathbb{R}$ of real numbers, equipped with the standard metric and $0$ as the anchor point. Each $r^{\mathcal{M}}$ is the underlying rational number $r$  (hence we omit the difference between constant symbol $r$ and the underlying rational $r$). 
     
    \item $f^{\mathcal{M}}$ is a function $M^{s_1}\times\cdots \times M^{s_n}\to M^{s_0}$ that is  locally uniformly continuous and bounded with respect to the metrics. 
\end{itemize}

\end{definition}

For a Henson signature $L$, the $L$-terms are defined in the same way as in first-order logic. We now introduce the syntax of Henson structures.

\begin{definition}[cf.\ Definition 6.13 and Definition 6.15 of \cite{DI2017}]
A \emph{positive bounded $L$-formula} is constructed by the following rules:
\begin{itemize}
    \item If $t$ is a term of sort $s_{\mathbb{R}}$ (real-valued term) and $r$ is a rational constant, then 
    $$t\leq r\text{\quad and \quad} r\leq t$$
    are positive bounded $L$-formulas.
    \item If $\varphi$ and $\psi$ are positive bounded $L$-formulas, then 
    $$(\varphi\wedge \psi)\text{\quad and \quad} (\varphi \vee \psi)$$
    are positive bounded $L$-formulas. 
    \item If $\varphi$ is a positive bounded $L$-formulas, $r$ is a rational number, and $x$ is a variable, then 
    $$(\forall_r x\varphi)\text{\quad and \quad} (\exists_r x\varphi)$$
    are positive bounded $L$-formulas. 
\end{itemize}
As in the first order case, we call a positive bounded $L$- formula without free variables a \emph{sentence} and a $L$-theory is a collection of sentences.
\end{definition}

\begin{definition}[cf.\ Definition 6.16 and Definition 6.18 of \cite{DI2017}]
Let $\mathcal{M}$ be a $L$-Henson structure and $\varphi(x_1,\ldots,x_n)$ be a positive bounded formulas with free variables $x_i$ of sort $s_i$. For arbitrary $a_1,\ldots, a_n$ such that $a_i$ is an element of $M^{s_i}$, we define the \emph{discrete satisfaction relation} $\mathcal{M}\models \varphi[a_1,\ldots,a_n]$ inductively\footnote{Note that \emph{officially} the  elements $a_i$ of  $M^{s_i}$ are not plugged into the formula $\varphi[a_1,...,a_n]$, but names
$\bar{a}_1,...\bar{a_n}$ for them.}:
\begin{itemize}
    \item If $\varphi$ is of the form $t\leq r$ (or $r\leq t$), then $\mathcal{M}\models \varphi[a_1,\ldots,a_n]$ if and only if 
    \begin{eqnarray*}
    t^{\mathcal{M}}[a_1,\ldots,a_n]\leq r \\
   \text{(or } r\leq t^{\mathcal{M}}[a_1,\ldots,a_n]).
    \end{eqnarray*}
    \item If $\varphi$ is of the form $\psi_1\wedge \psi_2$ (or $\psi_1\vee \psi_2$), then $\mathcal{M}\models \varphi[a_1,\ldots,a_n]$ if and only if
   \begin{eqnarray*}
     && \mathcal{M}\models \psi_1[a_1,\ldots,a_n] \text{\quad and \quad} \mathcal{M}\models\psi_2[a_1,\ldots,a_n]\\
    \text{(or} && \mathcal{M}\models \psi_1[a_1,\ldots,a_n] \text{\quad or \quad} \mathcal{M}\models \psi_2[a_1,\ldots,a_n]).
   \end{eqnarray*} 
    \item If $\varphi$ is of the form $\forall_rx\psi$ with $x$ of sort $s$, then $\mathcal{M}\models \varphi[a_1,\ldots,a_n]$ if and only if $$\mathcal{M}\models \psi[a,a_1,\ldots,a_n] \text{ for all } a\in  B_{M^{s}}(r)$$    
    where $ B_{M^{s}}(r)$ denotes the open ball around the anchor point in $M^{s}$ of radius $r$.

    \item If $\varphi$ is of the form $\exists_rx\psi$ with $x$ of sort $s$, then $\mathcal{M}\models \varphi[a_1,\ldots,a_n]$ if and only if $$\mathcal{M}\models \psi[a,a_1,\ldots,a_n] \text{ for some } a\in B_{M^{s}}[r]$$  
     where $ B_{M^{s}}[r]$ denotes the closed ball around the anchor point in $M^{s}$ of radius $r$.
\end{itemize}

Given a positive bounded formula $\varphi$, we define the \emph{approximation of $\varphi$}, written as $\varphi'\Leftarrow_{\mathcal{A}}\varphi$, inductively on the complexity of $\varphi$:
\begin{eqnarray*}
\varphi  &\qquad& \text{Approximations of } \varphi\\
r\leq t &\qquad& r'\leq t \text{ where } r'<r\\
t\leq r &\qquad& t\leq r' \text{ where } r'>r\\
\varphi_1\wedge \varphi_2 &\qquad& \varphi_1'\wedge\varphi_2'  \text{ where $\varphi_i'$ approximates $\varphi_i$} \\
\varphi_1\vee \varphi_2 &\qquad& \varphi_1'\vee\varphi_2'  \text{ where $\varphi_i'$ approximates $\varphi_i$} \\
\forall_r x \varphi &\qquad& \forall_{r'} x \varphi' \text{ where } r'<r \text{ and $\varphi'$ approximates $\varphi$}\\
\exists_r x \varphi &\qquad& \exists_{r'} x \varphi' \text{ where } r'>r \text{ and $\varphi'$ approximates $\varphi$}.
\end{eqnarray*}
We say \emph{$\mathcal{M}$ approximately satisfies $\varphi[a_1,\ldots,a_n]$}, denoted as $\mathcal{M}\models_{\mathcal{A}} \varphi[a_1,\ldots,a_n]$, if $\mathcal{M}\models \varphi'[a_1,\ldots,a_n]$ for every approximation $\varphi'$ of $\varphi$. Given $\varphi(x_1,\ldots, x_n)$ a set of positive bounded formulas with free variable $x_1,\ldots, x_n$ and $\mathcal{M}$ a Henson structure,  we say $\varphi(x_1,\ldots, x_n)$ is \emph{discretely satisfiable in $\mathcal{M}$} if there exists $a_1,\ldots,a_n$ in $\mathcal{M}$ such that $\mathcal{M}\models \varphi[a_1,\ldots, a_n]$ and $\varphi(x_1,\ldots, x_n)$ is \emph{approximately satisfiable in $\mathcal{M}$} if $\mathcal{M}\models_{\mathcal{A}} \varphi[a_1,\ldots, a_n]$.
\end{definition}
\begin{remark}
\label{rem:met:vs:normed}
We note the crucial difference in the semantics of bounded quantification between the normed and metric cases. Unlike the metric case, the semantics of both bounded existential and universal quantification in the normed setting \cite{henson2002ultraproducts} is taken over closed balls. Although this difference restricts the set of formulas that can be expressed in positive bounded logic in the metric setting (in the normed space setting, one can express bounded universal quantification over open balls in positive bounded logic, but in the metric setting, one cannot express universal quantification over closed balls), this restriction is a requirement to develop a sensible model theory on metric structures. However, as noted in \cite{henson2002ultraproducts}, both choices of semantics work in the normed setting.   
\end{remark}

It is clear that discrete  satisfaction implies approximate  satisfaction.

\begin{definition}[cf.\ Definition 6.22 of \cite{DI2017}]
Given $\Gamma$ a set of positive bounded formulas, we write $\Gamma^+$ to denote the set of all approximations of formulas in $\Gamma$. Given $\Gamma$ a collection of positive bounded formulas, we call $\Gamma$ \emph{a positive bounded $L$-theory} if $T$ consists of only positive bounded sentences, namely formulas positive bounded with no free variables. Given an $L$-structure $\mathcal{M}$ and an $L$-theory $T$, we say that $\mathcal{M}$ is \emph{a (approximate) model of $T$}, denoted as $\mathcal{M}\models_{\mathcal{A}} T$, if $\mathcal{M}\models_{\mathcal{A}} \varphi$ for every $\varphi$ in $T$. Given an $L$-structure $\mathcal{M}$, we write $\mathrm{Th}_{\mathcal{A}}(\mathcal{M})$ for the set of positive bounded sentences that are approximately satisfied by $\mathcal{M}$. Given a Henson $L$-theory $T$, we use $\mathrm{Mod}_{\mathcal{A}}(T)$ to denote the class of all approximate models of $T$. For Henson $L$-structures $\mathcal{M}$ and $\mathcal{N}$ , we say $\mathcal{M}$ and $\mathcal{N}$ are \emph{(approximately) elementarily equivalent}, written $\mathcal{M}\equiv_{\mathcal{A}}\mathcal{N}$, if $\mathrm{Th}_{\mathcal{A}}(\mathcal{M})=\mathrm{Th}_{\mathcal{A}}(\mathcal{N})$. In the discrete context, we also have the corresponding notions $\mathcal{M}\models \Gamma$, $\mathrm{Th}(\mathcal{M})$, $\mathrm{Mod}(T)$, and $\equiv$ by replacing approximate satisfaction in the definitions with discrete satisfaction. In the remainder of this article, we will assume the approximate versions unless otherwise specified. When needed, we will emphasize the discrete versions explicitly.
\end{definition}

\begin{remark}
Positive bounded logic is usually referred to as the positive and bounded fragment of first-order logic by viewing Henson structures as first-order structures and bounded quantifications as 
$$\forall_rx\varphi:=\forall x(d(x,a)< r\to \varphi)\text{ and } \exists_rx\varphi:=\exists x(d(x,a)\leq  r\wedge \varphi)$$

In this paper, we instead adopt another convention by interpreting first-order logic as the discrete part of positive bounded logic. Namely, first-order structures are interpreted as Henson structures endowed with the discrete metric. Formulas in first-order logic are treated as formulas in positive bounded logic by viewing 
\begin{itemize}
    \item $\forall x\varphi:=\forall_2x\varphi$ and $\exists x\varphi:=\exists_2x\varphi$.
    \item First-order predicates are identified with their indicator function with value $\leq \frac{1}{2}$. In particular, $x=y:=d(x,y)\leq \frac{1}{2}$.
    \item $\neg \varphi$ is defined using De Morgan's rules and passed down to the predicate level. The negation of a predicate is the positive bounded formula expressing the indicator function $\geq \frac{1}{2}$. In particular, $\neg(x=y):=d(x,y)\geq \frac{1}{2}$.
\end{itemize}
Our model-theoretic and proof-theoretic treatments of positive bounded logic in this paper thus extend naturally to first-order logic for discrete structures by the above interpretation. 
\end{remark}

\begin{definition}[cf.\ Definition 6.6 and Definition 6.22 of \cite{DI2017}]
Given a Henson signature $L$, a collection $\mathscr{C}$ of Henson $L$-structures is called a \emph{uniform class} if for every function symbol $f: s_1\times \cdots \times s_n\to s_0$ in $L$ and every $r>0$, the following two conditions hold:
\begin{enumerate}
    \item (Equiboundedness on bounded sets) There exists $\Omega_{f,r}\in [0,\infty)$ such that for every $\mathcal{M}$ in $\mathscr{C}$, 
    $$x\in B_{\mathcal{M}^{(s_1)}}(r)\times \cdots \times B_{\mathcal{M}^{(s_n)}}(r)\to f^{\mathcal{M}}\in  B_{\mathcal{M}^{(s_0)}}(\Omega_{f,r}) .$$ 
    In other words, $\Omega_{f,r}$ is a bound uniform over $\mathscr{C}$ for $f$ on domains bounded by $r$. 
    \item (Equicontinuity on bounded sets) There exists $\Delta_{f,r}:(0,\infty)\to (0,\infty)$ such that for every $\mathcal{M}$ in $\mathscr{C}$ and $x,y$ in $B_{\mathcal{M}^{(s_1)}}(r)\times \cdots \times B_{\mathcal{M}^{(s_n)}}(r)$, 
    $$d_{s_1\times\cdots \times s_n}^{\mathcal{M}}(x,y)<\Delta_{f,r}(\epsilon)\to d_{s_0}^{\mathcal{M}}(f^{\mathcal{M}}(x),f^{\mathcal{M}}(y))<\epsilon.$$
    $\Delta_{f,r}$ is a modulus of uniform continuity, which is uniform over $\mathscr{C}$, for $f$ on domains bounded by $r$. 
\end{enumerate}
We say an $L$-theory $T$ is \emph{uniform} if $\mathrm{Mod}_{\mathcal{A}}(T)$ is a uniform class.
\end{definition}
We have the compactness theorem for positive bounded formulas.
\begin{theorem}[cf.\ Theorem 6.31 of \cite{DI2017}]
\label{thrm:compactness}
    Let $L$ be a Henson signature, $\mathscr{C}$ be a uniform class of  Henson $L$-structures and let $T$ be an $L$-theory. If every finite subset of $T^+$ has a discrete model in $\mathscr{C}$, then $T$ has a model in $\mathscr{C}$.
\end{theorem}
An important application of the compactness theorem is the following uniformity principle: 
\begin{theorem}
\label{thrm:main:model:unif}
  Let $L$ be a Henson signature, $T$ a uniform $L$-theory, $X$ an arbitrary set, and $\varphi_{x,n}$ a positive bounded $L$-sentence for all $x \in X$ and $n \in \NN$. Suppose for all $\mathcal{M}\models_{\mathcal{A}}T$, we have that: 
  \[
  \mbox{for all } x \in X \mbox{ there exists } n \in \NN \mbox{ such that }
    \mathcal{M}\nvDash_{\mathcal{A}} \varphi_{x,n}.
  \]
Then for all $x \in X$ there exists $N \in \NN$ such that for all $\mathcal{M}\models_{\mathcal{A}}T$ there exists $n \le N$ such that  
\[
\mathcal{M} \nvDash_{\mathcal{A}} \varphi_{x,n}.
\]
\end{theorem}

\begin{proof}
       Suppose, for contradiction, that the conclusion of the theorem was false. Then we have some $x \in X$ such that for each $N \in \NN$ there exist $\mathcal{M}_N\models_{\mathcal{A}}T$ such that for all $n \le N$ we have 
    \[
 \mathcal{M}_N \models_{\mathcal{A}} \varphi_{x,n}.
    \]
Consider the theory $\Gamma:=\{\varphi_{x,n}| n \in \NN\}$. We show that $\Gamma$ has a model in $\mathrm{Mod}_{\mathcal{A}}(T)$ which will contradict the premise of the theorem and we are done. By the compactness theorem (Theorem \ref{thrm:compactness}), it suffices to show that $\{\varphi'_{x,n_1},\ldots,\varphi'_{x,n_k}\}$ has a discrete model in $\mathrm{Mod}_{\mathcal{A}}(T)$ for each $k,n_1,\ldots,n_k \in \NN$, and $\varphi'_{x,n_i}\Leftarrow_{\mathcal{A}}\varphi_{x,n_i}$. Take $N:= \max_{i\le k} n_i$ and we consider $\mathcal{M}_N$, which is a discrete model in $\mathrm{Mod}_{\mathcal{A}}(T)$. For all $n\le N$, we have that $\mathcal{M}_N\models_{\mathcal{A}}\varphi_{x,n}$ and the result follows, since for each $n_i$ and $\varphi'_{x,n_i}$, we have 
       \[
 \mathcal{M}_N \models \varphi'_{x,n_i}.
    \]
\end{proof}
Theorem \ref{thrm:main:model:unif} was already established in \cite{DI2017} in the context of metastable convergence and called the \emph{Uniform Metastability Principle} by the authors (cf.\ Proposition 2.4 of \cite{DI2017}).

We actually have a strengthening of Theorem \ref{thrm:main:model:unif}, where the premise of the theorem doesn't have to hold for all models of $T$, but a special subset of models.

\begin{definition}
      Let $L$ be a Henson signature. A class of Henson $L$-structures $\mathscr{C}$ is called \emph{relatively consistent} if for every uniform $L$-theory $T$, we have
      \[
      \mathrm{Mod}_{\mathcal{A}}(T)\neq \emptyset \to  \mathrm{Mod}_{\mathcal{A}}(T)\cap \mathscr{C} \neq \emptyset.
      \]
\end{definition}
We have the following strengthening of Theorem \ref{thrm:main:model:unif}.
\begin{theorem}
\label{thrm:main:model:unif:gen}
  Let $L$ be a Henson signature, $T$ a uniform $L$-theory, $\mathscr{C}$ a relatively consistent class of Henson $L$-structures, $X$ an arbitrary set, and $\varphi_{x,n}$ a positive bounded $L$-sentence for all $x \in X$ and $n \in \NN$. Suppose for all $\mathcal{M} \models_{\mathcal{A}} T$ in $\mathscr{C}$, we have that: 
  \[
  \mbox{for all } x \in X \mbox{ there exists } n \in \NN \mbox{ such that }
    \mathcal{M}\nvDash_{\mathcal{A}} \varphi_{x,n}.
  \]
Then for all $x \in X$ there exists $N \in \NN$ such that for all $\mathcal{M}\models_{\mathcal{A}} T$ there exists $n \le N$ such that  
\[
\mathcal{M} \nvDash_{\mathcal{A}} \varphi_{x,n}.
\]
\end{theorem}

\begin{proof}
We argue by contradiction. Following the exact proof of Theorem \ref{thrm:main:model:unif}, we have that for some $x$, $\{\varphi_{x,n}| n \in \NN\}$ has a model in $\mathrm{Mod}_{\mathcal{A}}(T)$. So, the fact that $\mathscr{C}$ is relatively consistent implies there is a structure in $\mathscr{C}$ that is a model of $T\cup \{\varphi_{x,n}| n \in \NN\}$, which contradicts the assumption of the theorem.
\end{proof}
The most important relatively consistent class is that of saturated models.

\begin{definition}[cf. Definition 6.45 of \cite{DI2017} and the comment before Proposition 6.46 in \cite{DI2017}]
\label{def:sat}
Given $\mathcal{M}$ a Henson $L$-structure and $\kappa \le |\mathcal{M}|$ an infinite cardinal, we say $\mathcal{M}$ is \emph{$\kappa$-saturated} if whenever $C$ is a subset of the universe of $\mathcal{M}$ of cardinality $<\kappa$ and $\Gamma(x_1,\ldots,x_n)$ is a set of positive bounded $L(C)$-formulas\footnote{We write $L(C)$ for the Henson signature which extends $L$ with fresh constant symbols for each element in $C$. We interpret the new constants coming from $C$, semantically, as the corresponding element in $C$.}, if there exists $r$ such that for any finite $\Delta\subseteq \Gamma$ 
$$(\mathcal{M},C)\models_{\mathcal{A}}\exists_r x_1\ldots \exists_r x_n\, \bigwedge_{\varphi \in \Delta} \varphi(x_1,\ldots,x_n)$$
then there exists $c_1,\ldots,c_n$ in the universe of $\mathcal{M}$ such that 
$$(\mathcal{M},C)\models_{\mathcal{A}}\varphi(c_1,\ldots,n_n)$$
for all $\varphi \in \Gamma$. We say $\mathcal{M}$ is \emph{countably saturated} if $\mathcal{M}$ is $\aleph_1$-saturated.
\end{definition}

\begin{theorem}[cf. Theorem 6.47 of \cite{DI2017}]
\label{thrm:approx:model}
    Let $L$ be a Henson signature and $\mathcal{M}$ a countably saturated Henson $L$-structure. For all positive bounded $L$-formula $\varphi(x_1, . . . , x_n)$ and elements $a_1,\ldots, a_n$  of suitable sorts, we have:
    \[
    \mathcal{M}\models_{\mathcal{A}} \varphi[a_1,\ldots, a_n] \leftrightarrow \mathcal{M}\models \varphi[a_1,\ldots, a_n].
    \]
\end{theorem}
\begin{remark}
    If one takes the semantics for bounded quantification to be as that of the normed case, cf.\ Remark \ref{rem:met:vs:normed}, then Theorem \ref{thrm:approx:model} fails, since if we consider $\NN$ with the discrete metric $d(x,y):= |x-y|$, then it is clear that 
    \[
    \forall n \le 1\, (d(x,0) =0)
    \]
    is false, but for all $\varepsilon >0$ and $r<1$
    \[
    \forall n \le r\, (d(x,0) \le \varepsilon) 
    \]
    is true. In the normed space case, one must crucially use the uniform continuity of the function symbols, as well as the fact that norm spaces have `intermediate points', which is not the case for general metric spaces. We note here, however, that for metric spaces, with `intermediate points', such as geodesic spaces, one can argue as in the normed space case and obtain Theorem \ref{thrm:approx:model} for the more liberal semantics of bounded universal quantification.
\end{remark}
\begin{theorem}[Existence of Saturated Models]
\label{thrm:exist:sat}
    Let $L$ be a Henson signature and $\kappa$ be an infinite cardinal. The class of all $\kappa$-saturated Henson $L$-structures is relatively consistent.
\end{theorem}
\begin{proof}
    The result is immediate from Proposition 6.48 of \cite{DI2017}.
\end{proof}
We now have the following uniformity principle from which many uses of the ultraproduct to show the existence of uniform bounds follow:
\begin{corollary}
    \label{cor:main:model:unif:sat}
  Let $L$ be a Henson signature, $T$ a uniform $L$-theory, $\kappa$ an infinite cardinal, $X$ an arbitrary set, and $\varphi_{x,n}$ a positive bounded $L$-sentence for all $x \in X$ and $n \in \NN$. Suppose for all $\kappa$-saturated models $\mathcal{M} \in \mathrm{Mod}(T)$, we have that: 
  \[
  \mbox{for all } x \in X \mbox{ there exists } n \in \NN \mbox{ such that }
    \mathcal{M}\nvDash\varphi_{x,n}.
  \]
Then for all $x \in X$ there exists $N \in \NN$ such that for all $\mathcal{M}\in \mathrm{Mod}(T)$ there exists $n \le N$ such that  
\[
\mathcal{M} \nvDash_{\mathcal{A}} \varphi_{x,n}.
\]
\end{corollary}
\begin{proof}
    The result then follows from Theorems \ref{thrm:exist:sat} and \ref{thrm:main:model:unif:gen}, with the weakening of the premise from $\nvDash_{\mathcal{A}}$ to $\nvDash$ is justified by Theorem \ref{thrm:approx:model}. 
\end{proof}

\subsection{Loeb structures}
\label{subsec:ultra}
Since both our applications, Theorems \ref{thrm:meta:main:strong} and \ref{thrm:avigad:dominated}, concern probability measures, we need to introduce so-called pre-Loeb structures \cite{DI2017}. These are Henson structures containing finitely additive probability measures.
\begin{definition}[cf. Definition 3.1 of \cite{DI2017}]
    The \emph{signature for Loeb structures}, denoted by $\mathcal{L}$, is a \emph{Henson signature} with sorts
    \[
    \mathcal{S} := \{s_\RR, \Omega, \mathcal{F}\},
    \]
    where $\Omega$ is interpreted as a set and $\mathcal{F}$ as a Boolean algebra of subsets of $\Omega$. Since $\mathcal{L}$ is a Henson signature, it comes equipped with the following function symbols:
    \begin{itemize}
        \item $d_{s_\RR}: s_\RR\times s_\RR\to s_\RR$, $d_{\Omega}: \Omega \times \Omega \to s_\RR$, and $d_{\mathcal{F}}: \mathcal{F} \times \mathcal{F} \to \RR$ representing metrics on the respective sorts.
        \item $\omega_0: \Omega$, and $\emptyset: \mathcal{F}$ representing anchor points on the respective spaces.
   
    \end{itemize}

    In addition, we introduce the following symbols:
    \begin{itemize}
        \item $\in: \Omega \times \mathcal{F} \to s_\RR$ representing the element relation.
        \item $\cup: \mathcal{F} \times \mathcal{F} \to \mathcal{F}$ representing union.
        \item $\cap: \mathcal{F} \times \mathcal{F} \to \mathcal{F}$ representing intersection.
        \item $(\cdot)^c: \mathcal{F} \to \mathcal{F}$ representing complementation.
        \item $\emptyset: \mathcal{F}$ representing the empty set.
        \item $\mu: \mathcal{F} \to s_\RR$ representing the probability measure.
    \end{itemize}
    We write $\model{\omega \in A}$ for $\in(\omega,A)$.
    
       A \emph{pre}-Loeb (probability) structure is a Henson structure $\mathcal{M}$ with signature $\mathcal{L}$ satisfying, for all $A, B \in \mathcal{F}^{\mathcal{M}}$ and $\omega \in \Omega^{\mathcal{M}}$:
    \begin{itemize}
        \item $d_{\mathcal{F}}$ and $d_{\Omega}$ are discrete metrics;
        \item $\model{\omega \in^{\mathcal{M}} A} \in \{0,1\}$;
        \item $d_{\mathcal{F}}(A,B) = \sup_{\omega \in^{\mathcal{M}} \Omega} \big| \model{\omega \in^{\mathcal{M}} A} - \model{\omega \in^{\mathcal{M}} B} \big|$;
        \item $\model{\omega \in^{\mathcal{M}} \emptyset} = 0$;
        \item $\model{\omega \in^{\mathcal{M}} \Omega} = 1$;
        \item $\model{\omega \in^{\mathcal{M}} (A \cup B)} = \model{\omega \in^{\mathcal{M}} A} \lor \model{\omega \in^{\mathcal{M}} B}$;
        \item $\model{\omega \in^{\mathcal{M}} (A \cap B)} = \model{\omega \in^{\mathcal{M}} A} \land \model{\omega \in^{\mathcal{M}} B}$;
        \item $\model{\omega \in^{\mathcal{M}} A} + \model{\omega \in^{\mathcal{M}} A^c} = 1$;
        \item $\mu(\hat{\Omega}^\mathcal{M}) = 1$;
        \item $0 \le \mu^{\mathcal{M}}(A) \le 1$;
        \item $\mu^{\mathcal{M}}(A \cup B) + \mu^{\mathcal{M}}(A \cap B) = \mu^{\mathcal{M}}(A) + \mu^{\mathcal{M}}(B)$.
    \end{itemize}
  Here and throughout, we write $\hat{\Omega}:\equiv \emptyset^{c}$.
\end{definition}
It is clear that the axioms pre-Loeb structures can be expressed by positive bounded  $\mathcal{L}$-formulas and the class of pre-Loeb structures is axiomatizble by a uniform $\mathcal{L}$-theory. In general, a Loeb structure is not itself a \emph{real} finitely additive probability space, but it can be naturally associated with one.
\begin{definition}[cf. Definition 3.3 of \cite{DI2017}]
    Let $\mathcal{M}$ be a pre-Loeb probability structure.  
    The \emph{external probability space} on $\Omega^{\mathcal{M}}$ is defined as follows:
    \begin{itemize}
        \item For each $A \in \mathcal{F}^{\mathcal{M}}$, define $[A]^{\mathcal{M}} := \{ \omega \in \Omega^{\mathcal{M}} : \model{\omega \in^{\mathcal{M}} A} = 1 \} \subseteq \mathcal{P}(\Omega^{\mathcal{M}})$.
        \item Let $[\mathcal{F}]^{\mathcal{M}} := \{ [A]^{\mathcal{M}} : A \in \mathcal{F}^{\mathcal{M}} \}$.
            \medskip
        \item Define $[\mu]^{\mathcal{M}} : [\mathcal{F}]^{\mathcal{M}} \to \RR$ by $[\mu]^{\mathcal{M}}([A]^{\mathcal{M}}) := \mu^{\mathcal{M}}(A)$.
    \end{itemize}
    The mapping $[\cdot]^{\mathcal{M}}: \mathcal{F}^{\mathcal{M}} \to [\mathcal{F}]^{\mathcal{M}}$ is an isomorphism of Boolean algebras, so $[\mu]^{\mathcal{M}}$ is well defined.  
    The triple $(\Omega^{\mathcal{M}}, [\mathcal{F}]^{\mathcal{M}}, [\mu]^{\mathcal{M}})$ is called the \emph{external probability space} of $\mathcal{M}$, and it is indeed a finitely additive probability space.
\end{definition}

When $\mathcal{M}$ is fixed, we write $[\mu], [\mathcal{F}], [A]$ for $[\mu]^{\mathcal{M}}, [\mathcal{F}]^{\mathcal{M}}, [A]^{\mathcal{M}}$,  respectively.

\begin{theorem}
\label{thrm:extensionLoeb}
    Let $\mathcal{M}$ be an $\aleph_1$-saturated pre-Loeb structure.  
    The external probability space $([\Omega], [\mathcal{F}], [\mu])$ of $\mathcal{M}$ extends uniquely to a $\sigma$-additive probability space $(\Omega, \mathcal{F}_L, [\mu]_L)$ with $\mathcal{F}_L:= \sigma([\mathcal{F}])$ the $\sigma$-algebra generated by $[\mathcal{F}]$.
\end{theorem}
\begin{proof}
    This is immediate from  Proposition 3.4 of \cite{DI2017}.
\end{proof} 

\subsection{Application I: A model-theoretic proof of Theorem  \ref{thrm:meta:main:strong}}
\label{subsec:model:APP1}
Theorem  \ref{thrm:meta:main:strong} will follow from Corollary \ref{cor:main:model:unif:sat}. We first demonstrate that the context of the result can be captured in positive bounded logic. To do this, we first add a group structure to Loeb structures.

We introduce the following Henson signature $\mathcal{L}_{G}$, which will extend $\mathcal{L}$ by the following:
    \begin{itemize}
        \item A binary function symbol $\cdot: \Omega \times \Omega \to \Omega$
        \item A binary function symbol $\cdot:  \mathcal{F}\times \Omega \to \mathcal{F}$
        \item A binary function symbol $\cdot: \Omega \times \mathcal{F} \to \mathcal{F}$
        \item A unary constant symbol $(\cdot)^{-1}: \Omega \to \omega$
        \item A constant $e : \Omega$
        \item A constant $A: \mathcal{F}$
    \end{itemize}
       Now define $\mathscr{C}$ to be the uniform class of $\mathcal{L}_{G}$-structures, $\mathcal{M}$, which are pre-Loeb structures, satisfy that $(\Omega^\mathcal{M},\cdot^\mathcal{M},e^\mathcal{M})$ forms a group as well as
       \begin{itemize}
           \item $\forall a_1^, \ldots, a_k, b_1,\ldots,b_k \in \Omega^{\mathcal{M}}\, \bigvee_{i,j \le k}(\neg (\model{(a_i\cdot b_j) \in A} =1 \leftrightarrow i\le j))$
           \item $\forall H \in \mathcal{F}^{\mathcal{M}}\, \forall g \in \Omega^{\mathcal{M}}\,(\mu(g\cdot H) = \mu(H)=\mu(H\cdot g))$
             \item $\forall H \in \mathcal{F}^{\mathcal{M}}\, \forall g,q \in \Omega^{\mathcal{M}}\,(\model{q \in H\cdot g} =1 \leftrightarrow \exists h\in \Omega^{\mathcal{M}}\,(\model{h \in H} =1 \land q = h\cdot g))$
           \item $\forall H \in \mathcal{F}^{\mathcal{M}}\, \forall g,q \in \Omega^{\mathcal{M}}\,(\model{q \in g\cdot H} =1 \leftrightarrow \exists h\in \Omega^{\mathcal{M}}\,(\model{h \in H} =1 \land q =g\cdot h))$
       \end{itemize}
The above axioms can be expressed in positive bounded logic. In particular, we can express the final axiom as
\begin{equation*}
    \begin{aligned}
        & \forall_2 H\, \forall_2 g\,\forall_2 q((\model{q \in g\cdot H} =0 \lor (\exists_1 h\,(\model{h \in H} =1 \land q =g\cdot h))\\
        &\land ((\forall_2 h\,(\model{h \in H} =0 \lor h \neq g\cdot p))\lor\model{h \in g\cdot H} =1)) 
    \end{aligned}
\end{equation*}
where $ h \neq g\cdot p$ is an abbreviation for $d_\Omega(h,g\cdot p)\ge 1$. Furthermore, it is clear that $\mathscr{C}$ is axiomatizable by a uniform $\mathcal{L}_{G}$-theory, call this theory $T$. Finally, note that if $\mathcal{M} \in \mathscr{C}$ then for $H \in \mathcal{F}^{\mathcal{M}}$ and $g \in \Omega^\mathcal{M}$, one can easily verify that 
    \[
    [g\cdot H] = g \cdot[H] \mbox{ and } [H\cdot g] = [H]\cdot g.
    \]
Theorem \ref{thrm:meta:main:strong} will follow from Corollary \ref{cor:main:model:unif:sat} and the following result:
\begin{theorem}
\label{thrm:main:nonquant:strong}
Suppose $k \ge 1$, $G$ is a group and $A$ is a $k$-stable subset of $G$. Let $\mathcal{F}$ be a bi-invariant Boolean algebra containing $A$. Let $\mu$ be a bi-invariant finitely additive probability measure on $\mathcal{F}$. Then there exists $m$ such that $\mathrm{Stab}_m(A) \in \mathcal{F}$, $\mathrm{Stab}_m(A) \le G$, $\mathrm{Stab}_m(A)$ has finite index, and $\mathrm{Stab}(A) = \mathrm{Stab}_m(A)$. In addition,  there exists a subset $Y \subseteq G$, which is a union of cosets of $\mathrm{Stab}_m(A)$, such that $\mu(A \triangle Y ) =0$.
\end{theorem}

Assuming Theorem \ref{thrm:main:nonquant:strong}, our first step is to state the non-uniform version of Theorem \ref{thrm:meta:main:strong} as a positive bounded $\mathcal{L}_{G}$ formula. For each $f: \NN \times \NN \to \NN$ and $n,m \in \NN$ we define the  positive bounded $\mathcal{L}_{G}$ formula $\varphi_{f,n,m}$ via
\begin{equation*}
    \begin{aligned}
            \varphi_{f,n,m}:\equiv& \forall_2 H\,\forall_2 Y((\mbox{$H$ is not a subgroup})\lor(\mbox{$H$ does not have index at most $n$})\lor\\
            &(\mbox{not }(\mathrm{Stab}_m^{<}(A)= H  =\mathrm{Stab}_m(A)))\lor (\mbox{$Y$ is not a union of cosets of $H$})\lor\\
            &(\mu(A \triangle Y)\ge 2^{-f(n,m)}))
    \end{aligned}
\end{equation*}
where
\begin{equation*}
    \begin{aligned}
&\mbox{$H$ is not a subgroup}:\equiv \exists_1 g\, \exists_1 h\,((\model{g \in H} = 1 \land \model{h \in H} = 1) \land \model{g\cdot h^{-1} \in H} = 0 ) \lor H = \emptyset\\
& \mbox{$H$ does not have index at most $n$}:\equiv \forall_2 g_1\ldots\forall_2 g_n\,\left( d_{\mathcal{F}}\left(\bigcup_{i\le n}g_i\cdot H, \hat{\Omega}\right)=1  \right)
    \end{aligned}
\end{equation*}
\begin{equation*}
    \begin{aligned}
\mbox{not }(\mathrm{Stab}_m^{<}(A)= H  =\mathrm{Stab}_m(A)):\equiv& (\exists_1 h\,((\mu( A\cdot h \triangle A)\le 2^{-m})\land(\model{h \in H} = 0) ))\lor \\
&(\exists_1 h\,((\model{h \in H} = 1)\land(\mu(A\cdot h \triangle A)\ge 2^{-m}) ))
    \end{aligned}
\end{equation*}
and
\begin{equation*}
\mbox{$Y$ is not a union of cosets of $H$}:\equiv \forall_2 g_1\ldots\forall_2 g_n\,\exists_1 \omega\,\left( \model{\omega \in Y\triangle \bigcup_{i\le n}g_i\cdot H}=1 \right). 
\end{equation*}
In the above we write $P\triangle Q:\equiv (P\cap Q^c) \cup (Q\cap P^c)$ and $\bigcup_{i\le n}$ is `meta' notation for the application of $\cup$ $n$ times. We now apply Corollary \ref{cor:main:model:unif:sat} with $X= \NN^{\NN \times \NN}$. Take $\mathcal{M} \in \mathscr{C}=\mathrm{Mod}(T)$. Then $(\Omega^{\mathcal{M}},\cdot^{\mathcal{M}},e^\mathcal{M})$ is a group, the external probability measure $[\mu]$ will be bi-invariant, the external algebra $[\mathcal{F}]$ will be bi-invariant and contain $[A]$ and $[A]$ will be $k$-stable. Thus, Theorem \ref{thrm:main:nonquant:strong} implies there exists $m \in \NN$,$[H]\in [\mathcal{F}]$ which will correspond to some $H \in \mathcal{F}^{\mathcal{M}}$, and  $[Y] \in [\mathcal{F}]$ corresponding to some $Y \in \mathcal{F}^{\mathcal{M}}$ such that $H$ is a finite index subgroup with $\mathrm{Stab}^<_m([A])=[H]=\mathrm{Stab}_m([A])$, $[Y]$ a union of cosets of $[H]$ and $[\mu]([A]\triangle[Y]) =0$. This will imply that for all $f:\NN \times \NN \to \NN$ there exists $n,m \in \NN$ such that
\begin{equation*}
    \begin{aligned}
            &\exists [H] \in [\mathcal{F}]\,\exists [Y]\in [\mathcal{F}]\,((\mbox{$[H]$ is a subgroup})\land(\mbox{$[H]$ has index at most $n$})\land\\
            &(\mathrm{Stab}^<_m([A])=[H]=\mathrm{Stab}_m([A]))\land (\mbox{$[Y]$ is a union of cosets of $[H]$})\land\\
            &([\mu]([A] \triangle [Y])< 2^{-f(n,m)}))
    \end{aligned}
\end{equation*}
since for a given $f$ we can just take  $m, [H]$ and $[Y]$ from Theorem \ref{thrm:main:nonquant:strong} and $n$ the index of $[H]$. It then follows that $\mathcal{M}\nvDash \varphi_{f,n,m}$ and so the result follows from Corollary \ref{cor:main:model:unif:sat}, noting that one can view $(m,n)$ as a single variable via an appropriate paring function. All that remains is to prove Theorem \ref{thrm:main:nonquant:strong}. 

Throughout the remainder of the section, fix $G,A,k, \mathcal{F}, \mu$ as in the statement of Theorem \ref{thrm:main:nonquant:strong}. To prove Theorem \ref{thrm:main:nonquant:strong}, we need a series of lemmas. The first result is taken from \cite{Conant:2021:stability:in:group}. To state this result, we must first introduce the definition of a subset of $G$ being generic.
\begin{definition}
    A subset of $X \subseteq G$ is said to be (left) $n$-generic, if there exists $n$ (left) translates of $X$ covering $G$, that is, there exists $g_1,\ldots g_n \in G$ such that 
    \[
    G= \bigcup_{i = 1}^n g_iX.
    \]
    $X$ is generic if it is $n$-generic for some $n\ge 1$.
\end{definition}
\begin{lemma}
\label{lem:pos:gen}
    For all stable $B \in \mathcal{F}$, $\mu(B) >0$ iff $B$ is left  generic iff $B$ is right generic.
\end{lemma}
\begin{proof}
     $\mu$ restricted to the Boolean algebra of stable sets will be the unique bi-invariant finitely additive measure by point (a) Theorem 1.1 of \cite{Conant:2021:stability:in:group} (and the remark after the statement of the theorem). Thus, the result follows from point (b) of Theorem 1.1 of \cite{Conant:2021:stability:in:group}.
\end{proof}
 We must now recall some results from $\mathrm{VC}$ theory.

\begin{definition}
A set system on a set $X$ is a set $\mathcal{S} \subseteq \mathcal{P}(X)$ of subsets of $X$. For a set $X$, a set system on $X$ $\mathcal{S} \subseteq \mathcal{P}(X)$ and $W \subseteq X$, we say that $\mathcal{S}$ shatters $W$ if $\mathcal{P}(W) =\{A \cap S : S \in \mathcal{S}\}$. The $\mathrm{VC}$-dimension of a set system $\mathcal{S}$ on a set $X$ is $\mathrm{VC}(S) = \sup\{n \in \NN : \mathcal{S} \mbox{ shatters some } W \subseteq X \mbox{ of size }n\}$. If $B \subseteq G$, write $\mathrm{VC}_l(B):= \mathrm{VC}(\{gB|g \in G\})$ and $\mathrm{VC}_r(B):= \mathrm{VC}(\{Bg|g \in G\})$.
\end{definition}

The following result linking $\mathrm{VC}$ theory to stability is immediate.

\begin{theorem}[cf.\ Fact 2.8 \cite{conant2021quantitative}]
\label{thrm:stab:vc}
    If a set $B \subseteq G$ is $k$-stable, then $\mathrm{VC}_l(B), \mathrm{VC}_r(B) < k$.
\end{theorem}

The next result we need is a generalisation of Haussler's Packing Lemma 
\cite{haussler1995sphere} to finitely additive measures given in \cite{lovasz2010regularity}:

\begin{theorem}[cf.\ Lemma 4.6 of \cite{lovasz2010regularity}]
\label{thrm:packing}
    Let $m \in \NN$ and $\mathcal{S} \subseteq \mathcal{F}$ be a set system on $G$ with  $\mathrm{VC}$-dimension $k$ and $\mathcal{H} \subseteq \mathcal{S}$ satisfying  $\mu(B \triangle C) \ge 2^{-m}$ for all distinct $C, B \in \mathcal{H}$. Then $|\mathcal{H}| \le  \omega_k(m):=(80k)^k2^{20mk}$.
\end{theorem}
We must now recall the definition of general stable relations.
\begin{definition}
    For sets $X, Y$, a relation $\phi(x,y)$ on $X \times Y$ is said to be $k$-stable if there does not exist $a_1,\ldots,a_k$ and $b_k,\ldots,b_k$ satisfying 
    \[
    \forall i,j\le k \,(\phi(a_i,b_j)  \leftrightarrow i \le j). 
    \] 
\end{definition}
\begin{lemma}[cf.\ Proposition 3.7 of \cite{conant2021quantitative}]
\label{lem:stable:symm}
    The relation $\phi(x,(y,z))$ on $G \times G^2$ defined via $x \in Ay \triangle Az$ is $n$-stable, where $n= R(R(k, k + 1), R(k, k + 1))+ 1$ and $R(x,y)$ is the usual two-color Ramsey number for graphs.
\end{lemma}

\begin{lemma}
\label{lem:bound:symm:nonquant}
    There exists $m \in \NN$ such that for all $g,h \in G$, 
    \[
    \mu(Ag \triangle Ah) >0 \to \mu(Ag \triangle Ah)>2^{-m}
    \]
\end{lemma}
\begin{proof}
Write $\phi(x,(y,z))$ for the relation $x \in Ay \triangle Az$, which we know is $s$-stable for some $s$ by Lemma \ref{lem:stable:symm}. We argue by contradiction. Suppose for each $m \in \NN$, there exists $g,h \in G$ such that $2^{-m}\ge \mu(Ag \triangle Ah) >0$. 
By induction on $1 \le n \le s$,  for $1 \le i \le n$ we will construct $g_i,h_i \in G$, such that for all $1 \le t \le n$
\[
\mu\left(\bigcap_{i=1}^t Ag_i\triangle Ah_i \cap \bigcap_{i=t+1}^n(Ag_i\triangle Ah_i)^c\right)>0.
\]
Given this, for $1\le i \le s$, we can pick 
\[
a_i \in \bigcap_{i=1}^t Ag_i\triangle Ah_i \cap \bigcap_{i=t+1}^n(Ag_i\triangle Ah_i)^c
\]
which will imply that $\phi(a_i,(g_j,h_j))$ iff $i \ge j$. Taking $x_i = a_{s-i+1}$, $y_j = g_{s-j+1}$ and $z_i = h_{s-j+1}$ implies that $\phi(x_i,(y_j,z_j))$ iff $i \le j$ contradicting the stability of $\phi$. For the case $n =1$, by our assumption (with $m=0$ say) we can take $g,h \in G$ such that $\mu(Ag\triangle Ah) >0$. Now fix $1\le n < s$ and suppose we have constructed $g_i, h_i$ with $1 \le i \le n$ satisfying the desired condition, so we can take $m \in \NN$ such that, for all $1 \le t \le n$
\[
\mu\left(\bigcap_{i=1}^t Ag_i\triangle Ah_i \cap \bigcap_{i=t+1}^n(Ag_i\triangle Ah_i)^c\right)\ge 2^{-m}.
\]
By assumption, take $g,h \in G$ satisfying $2^{-(m+1)}\ge \mu(Ag \triangle Ah) >0$. Thus, Lemma \ref{lem:pos:gen} implies that $Ag \triangle Ah$ is generic. Setting 
\[
B:= \bigcap_{i=1}^n Ag_i\triangle Ah_i
\]
which will have measure greater than $2^{-m}$, we will have, for some $q \in \NN$ and $v_1, \ldots, v_q \in G$ that 
\[
B = \bigcup_{i=1}^q (B \cap v_i(Ag \triangle Ah))
\]
since $Ag \triangle Ah$ is generic. Thus we can take some $v \in G$ such that $\mu(B \cap Agv \triangle Ahv)=\mu(B \cap (Ag \triangle Ah)v) >0$. Setting $g_{n+1} = gv$ and $h_{n+1} = hv$ yields that 
\[
\mu\left(\bigcap_{i=1}^{n+1} Ag_i\triangle Ah_i \right)>0
\]
and for $1 \le t \le n$ we have 
\begin{equation*}
    \begin{aligned}
        \mu\left(\bigcap_{i=1}^t Ag_i\triangle Ah_i \cap \bigcap_{i=t+1}^{n+1}(Ag_i\triangle Ah_i)^c\right)&\ge \mu\left(\bigcap_{i=1}^t Ag_i\triangle Ah_i \cap \bigcap_{i=t+1}^n(Ag_i\triangle Ah_i)^c\right) - \mu\left(Ag_{n+1}\triangle hg_{n+1}\right)\\
        & \ge 2^{-m} - \mu\left((Ag\triangle Ah)v\right) \ge 2^{-(m+1)}>0.
    \end{aligned}
\end{equation*}
We use left invariance of $\mu$ and the fact that $2^{-(m+1)}\ge \mu(Ag \triangle Ah)$ to obtain the last line. 
\end{proof}
The next result we need  concerns the relation $\sim$ on $\mathcal{F}$ defined via $B \sim C$ iff $\mu(B \triangle C)=0$. One can easily verify that $\sim$ is an equivalence relation and we denote the equivalence class of $B \in \mathcal{F}$ as $[B]$. From the left invariance of $\mu$, one can easily verify that for all $B, C \in \mathcal{F}$ and $g \in G$, if $B \sim C$ then $Bg \sim Cg$, this yields an action of $G$ on the equivalence classes of $\sim$, $\mathcal{F}/\sim$ via $[A]\cdot g = [A\cdot g]$. 

\begin{lemma}
\label{lem:fin:orbit:nonquant}
    The orbit of $[A]$ under the action of $G$ on $\mathcal{F}/\sim$ is finite.
\end{lemma}
\begin{proof}
If this was not the case, we could take a sequence $(g_i) \subseteq G$ satisfying that for all $i < j, \mu(Ag_i \triangle Ag_j) >0$. Thus, Lemma \ref{lem:bound:symm:nonquant} implies there exists $m \in \NN$ such that for all $i < j, \mu(Ag_i \triangle Ag_j) \ge 2^{-m}$. Setting $\mathcal{S}:= \{Ag|g \in G\}$ and $\mathcal{H}:= \{Ag_i|i \in \NN\}$, Theorem \ref{thrm:stab:vc} implies $\mathrm{VC}(\mathcal{S}) < k$ and so Theorem \ref{thrm:packing} implies that $\mathcal{H}$ is finite, which is a contradiction.
\end{proof}
The last result we shall need, concerns a criteria for when we can conclude $\mathrm{Stab}_m^\mu(A)$ is is the algebra $\mathcal{F}$.
\begin{lemma}[cf.\ \cite{conant2021quantitative}]
\label{lem:Stab:in:F}
    Let $B \in \mathcal{F}$ with $\mathrm{VC}_r(B) < \infty$ (in particular, if $B$ is stable) and $m \in \NN$. If there exists $m< n$ such that $\mathrm{Stab}_n^\mu(B) = \mathrm{Stab}_m^\mu(B)$ then $\mathrm{Stab}_m^\mu(B) \in \mathcal{F}$. 
\end{lemma}
In the case when $G$ is finite and $\mu$ is the counting measure, the above lemma follows immediately from Proposition 4.5 of \cite{conant2021quantitative}. The general case can be proven by combining the proof of Proposition 4.5 of \cite{conant2021quantitative} and Lemma 4.3 of \cite{chernikov2016definable} as described in Remark 5.2 of \cite{conant2021quantitative}.
\begin{remark}
    Note that if $m < n$ then $\mathrm{Stab}_n^\mu(B) \subseteq \mathrm{Stab}_m^\mu(B)$. Thus the condition that $\mathrm{Stab}_n^\mu(B) = \mathrm{Stab}_m^\mu(B)$ is equivalent to $\mathrm{Stab}_m^\mu(B) \subseteq \mathrm{Stab}_n^\mu(B)$ or that for all $g \in G$, 
    \[
    \mu(Bg \triangle B) \le 2^{-m} \to  \mu(Bg \triangle B) \le 2^{-n}.
    \]
\end{remark}
\begin{proof}[Proof of Theorem \ref{thrm:main:nonquant:strong}]
    The stabilizer of $[A]$ under this action satisfies $\mathrm{Stab}([A]) = \{g \in G| \mu(Ag \triangle A) =0\}$. Lemma \ref{lem:fin:orbit:nonquant} and the orbit-stabilizer theorem implies $\mathrm{Stab}([A])$ has finite index. Theorem \ref{lem:bound:symm:nonquant} implies there exists $m \in \NN$ such that $\mathrm{Stab}([A]) = \mathrm{Stab}_m^\mu(A)$ and this will be the $m$ we take. The fact that $\mathrm{Stab}_m^\mu(A) \in \mathcal{F}$ follows from Lemma \ref{lem:Stab:in:F}. 

    Now, set $H:= \mathrm{Stab}([A]) = \mathrm{Stab}_m^\mu(A)$. We argue as in the proof of Theorem 1.3 in \cite{conant2021quantitative}. We show that for each $g \in G$ either $\mu(gH \cap A) =0$ or $\mu(gH \cap A^c) =0$.
    
    Suppose not. Then setting $B:=H \cap g^{-1}A$ and $C:=H \cap g^{-1}A^c$, would imply $\mu(B), \mu(C) >0$. Thus, Lemma \ref{lem:pos:gen} implies $B$ is right generic and so since $\mu(C)>0$ there exists $h \in G$ such that $\mu(Bh \cap C)>0$. For such an $h$, we have $Bh \cap C = H \cap Hh \cap g^{-1}(Ah \cap A^c)$ which implies  $h \in H$ and $\mu(Ah \cap A^c)=\mu(g^{-1}(Ah \cap A^c))>0$. But, $h \in H=\mathrm{Stab}([A])$ implies $\mu(Ah \cap A^c) \le \mu(Ah \triangle A)=0$ a contradiction.

    Now let $C_1,\ldots C_n$ be the left cosets of $H$. Let $I$ be the set of $1 \le i \le n$ such that $\mu(C_i \cap A) >0$, so if $i \in I$ we have $\mu(C_i \cap A^c) = 0$. Set 
    \[
    Y:= \bigcup_{i\in I}C_i.
    \]
    So 
    \begin{equation*}
        \begin{aligned}
            \mu(A \triangle Y) = \mu(A \cap Y^c) + \mu(A^c \cap Y)
            =\mu\left(\bigcup_{i \notin I} C_i \cap A\right)+ \mu\left(\bigcup_{i \in I} C_i \cap A^c\right) =0
        \end{aligned}
    \end{equation*}
    and the result follows.
\end{proof}

\subsection{Application II: A model-theoretic proof of Theorem  \ref{thrm:avigad:dominated}}
\label{subsec:model:APP2}
As with Theorem  \ref{thrm:meta:main:strong}, Theorem \ref{thrm:avigad:dominated} will follow from Corollary \ref{cor:main:model:unif:sat} after demonstrating that the context of the result can be captured in positive bounded logic. To do this, we must treat integrable functions. Our treatment is different from that of \cite{DI2017} and is informed by the  characterisation of integrable functions on finitely additive probability spaces \cite{RR1983}.  

For a probability space $(\Omega, \mathcal{F},\mu)$, a function $f:\Omega \to [0,1]$ is measurable if and only if there exists a sequence of \emph{simple functions} $\seq{s_n}$ (linear combinations of indicator functions of measurable sets) satisfying 
\begin{equation}
\label{eqn:pointwise:simple:bound}
    \forall \omega \in \Omega\,(|s_n(\omega)-f(\omega)| \le 2^{-n})
\end{equation}

for each $n \in \NN$. One direction follows from the well known fact that pointwise limits of simple functions are measurable. For the other direction, for a measurable function $f$ we can take 
\[
s_n(\omega):= \sum_{k=0}^{2^n}\frac{k}{2^{n}}I_{\{\omega| k2^{-n}\le f(\omega)< (k+1)2^{-n}\}}.
\]
Furthermore, the integral of a measurable function $f:\Omega \to [0,1]$ will be the limit of the integrals of the the simple functions representing $f$ (furthermore, a function being integrable is equivalent to the convergence of the integrals of the simple functions that represent it) since (\ref{eqn:pointwise:simple:bound}) and the properties of the integral imply\footnote{We note that our approach to integrable functions in the context of model theory for metric structures differs from that in \cite{DI2017} and is motivated by the proof-theoretic treatment given by the second author and Pischke \cite{neri-pischke:24:formal:pub}.}

\[
\left|\int_{\Omega}s_n(\omega)d\mu(\omega)-\int_{\Omega}f(\omega)d\mu(\omega)\right| \le \int_{\Omega}\left|s_n(\omega)-f(\omega)\right|d\mu(\omega) \le 2^{-n}.
\]
Next, for a measurable function $f : \Omega \to [0,1]$ and $\varepsilon > 0$, we need to be able to formally refer to (bounds for) the measures of unions and intersections of sets of the form $\{\omega \mid f(\omega) \le \varepsilon\}$. A related representation problem was addressed in the proof-theoretic work of the second author, Oliva and Pischke \cite{NeriOlivaPischke2026}. We adopt their approach here, and it forms a crucial ingredient of our model-theoretic proof of Theorem \ref{thrm:avigad:dominated}. A central contribution of that work was a novel \emph{outer measure} translation, whereby one could formalise that quantified statements hold with a specified degree of certainty in a way that naturally recasts the computational content of such statements in a stochastic setting. Concretely, we define the outer measure (cf.\ Definition~4.1.3 of \cite{RR1983}) $\mu^* : \mathcal{P}(\Omega) \to [0,1]$ by
\[
  \mu^*(F) := \inf\{\mu(B) : F \subseteq B \in \mathcal{F}\}.
\]
 and, following the approach of \cite[Section 3] {NeriOlivaPischke2026}, we have that for all $F \in \mathcal{P}(\Omega)$ and $\lambda \in \RR$  
\[
\mu^*(F) \ge \lambda \leftrightarrow \forall A\in \mathcal{F} (\forall \omega \in \Omega\,(\omega \in F \to \omega \in A) \to \mu(A) \ge \lambda)
\]
equivalently 
\[
\mu^*(F) < \lambda \leftrightarrow \exists A\in \mathcal{F} (\forall \omega \in \Omega\,(\omega \in F \to \omega \in A) \land \mu(A)<  \lambda).
\]
Thus, if a set $F \in \mathcal{F}$ is represented by some quantified statement $\varphi(\omega)$, that is, $F = \{\omega \mid \varphi(\omega)\}$, the above allows one to formalise that the measure of the set represented by $\varphi(\omega)$ exceeds some $\lambda$ without having to comprehend $\varphi(\omega)$ as a measurable set $F$ \cite[Section~3]{NeriOlivaPischke2026}. Even further, in the case that $\varphi(\omega)$ does not correspond to a measurable set, the above still allows for a meaningful interpretation of $\varphi(\omega)$ holding with a certain degree of certainty. Moreover, it is shown in \cite[Section 5]{NeriOlivaPischke2026} that the outer measure translation transforms the computational content of a quantified statement (now formally expressed in a suitable logical system) into an analogous stochastic variant. Lastly, we note that $\mu^*(F) = \mu(F)$ for all $F \in \mathcal{F}$.

Now, to show Theorem \ref{thrm:avigad:dominated} follows from Corollary \ref{cor:main:model:unif:sat}. We introduce the following Henson signature $\mathcal{L}_B$ which will be an extension of $\mathcal{L}$ by the following
\begin{itemize}
      \item A unary function symbol $f_{n,m}:\Omega \to \RR$ for each $n,m \in \NN$.
    \item A unary function symbol $s^{m,n}_i:\Omega \to \RR$ for each  $n,m,i \in \NN$.
    \item A constant symbol $A^{m,n,i}_j: \mathcal{F}$ for each $n,m,i \in \NN$
     \item A constant symbol $I_{n,m}:\RR$ for each $n,m \in \NN$. 
\end{itemize}
fix a functional $\Phi:\QQ^+\times \QQ^+\times \NN^\NN\to \NN$. We define $\mathscr{C}$ to be the uniform class of $\mathcal{L}_B$ structures, $\mathcal{M}$, that are Loeb structures, and additionally satisfy (in the following we suppress the superscript $\mathcal{M}$ for the interpretation of terms to enhance readability) that for each $m,n \in \NN$, $f_{n,m}(\omega) \in [0,1]$, which can be expressed in positive bounded logic. To deal with the measurability and integrability of $f_{n,m}$, we work with simple functions as above. For each $n,m,i \in \NN$ we have $s^{m,n}_i$ are the simple functions (determined by the measurable sets $A^{m,n,i}_j$) that converge uniformly to $f_{n,m}$, that is we have:
\begin{itemize}
    \item $A^{m,n,i}_q \cap A^{m,n,i}_p = \emptyset$ for each $n,m,i,p,q \in \NN$ with $p \neq q$.
    \item $A^{m,n,i}_0 \cup A^{m,n,i}_1 \ldots \cup A^{m,n,i}_{2^{i}}= \Omega$ for each $n,m,i \in \NN$
    \item $\forall \omega \in \Omega^\mathcal{M}\,(s^{m,n}_i(\omega) = 0 \lor s^{m,n}_i(\omega)=\frac{1}{2^i} \lor \ldots \lor s^{m,n}_i(\omega) =1)$ for each $n,m,i\in \NN$
    \item $\forall \omega \in \Omega^\mathcal{M}\,(\model{\omega \in A^{m,n,i}_q} =1 \to s^{m,n}_i(\omega)=\frac{q}{2^i} )$ for each $n,m,i\in \NN$ and $q \in[0;2^i]$
\end{itemize}
It is clear that the above can be expressed in positive bounded logic and expresses the fact that $s^{m,n}_i$ are simple functions taking values $0, \frac{1}{2^i},\ldots 1$. We now need to express that these simple functions converge uniformly to $f_{n,m}$. This can be expressed simply as 
\[
\forall \omega \in \Omega^\mathcal{M}\,(|s^{m,n}_i(\omega)-f_{n,m}(\omega)| \le 2^{-i})
\]
for each $n,m,i \in \NN$. Furthermore, we have the $I_{n,m}$ are the values of the integrals of $f_{n,m}$, which will be the uniform limit of the integrals of the simple functions:
\[
 \left|I_{n,m}- \sum_{i=0}^{2^k}\frac{i}{2^{k}}\mu(A^{m,n,k}_i)\right| \le 2^{-k}
\]
for each $n,m,k\in \NN$. Again, it is clear that the above formulas can be expressed in positive bounded logic. The last property we need is that $\Phi_{\lambda,\varepsilon,g}$ is a rate of metastable pointwise convergence. We use the outer measure trick we stated above and include the following axiom:
\[
\exists A \in \mathcal{F}^\mathcal{M} \left(\forall \omega\in \Omega^\mathcal{M} \,\left(\bigwedge_ {n\leq \Phi(\lambda,\varepsilon,g)}\, \bigvee_{i,j\in [n;n+g(n)]}(f_{i,j}(\omega)> \varepsilon) \to \model{\omega \in A}=1\right) \land \mu(A) \le \lambda\right)
\]
for each $\varepsilon, \lambda \in \QQ^+$ and $g: \NN \to \NN$. We note that this axiom is actually a weakening of the fact that $\Phi_{\lambda,\varepsilon,g}$ is a rate of metastable pointwise convergence (to fully characterize this, we would require the last inequality to be strict. However this is not possible in positive bounded logic), but this shall be enough for our purposes. It is clear this can be expressed in positive bounded logic and that the class $\mathscr{C}$ is axiomatizable by a uniform $\mathcal{L}_B$-theory which we call $T$.

Now that we have the required setup, we are ready to prove Theorem \ref{thrm:avigad:dominated}. Let $\mathcal{M}$ be an $\aleph_1$-saturated model in $\mathscr{C}=\mathrm{Mod}(T)$. Then, by Theorem \ref{thrm:extensionLoeb}, $(\Omega,\mathcal{F}_L,[\mu]_L)$ is a probability space and we have have for each $n,m \in \NN$ that $\seq{s_i^{n,m}}$ will be a sequence of simple functions (determined by the measureable set $[A^{m,n,i}] \in [\mathcal{F}] $) converging pointwise to $f_{n,m}$. Thus, $f_{n,m}$ is a random variable with integral $I_{n,m}$. Furthermore, suppose $\varepsilon, \lambda \in \QQ^+$ and $g: \NN \to \NN$ are give then there exists $[A] \in [\mathcal{F}]$ such that
 \[
 \forall \omega^\mathcal{M} \in \Omega \,\left(\bigwedge_ {n\leq \Phi(\lambda,\varepsilon,g)}\, \bigvee_{i,j\in [n;n+g(n)]}(f_{i,j}(\omega)> \varepsilon) \to \omega \in [A]\right) \land [\mu]([A]) \le \lambda.
 \]
 Now suppose we have 
 \[
 [\mu]_L\left(\left\{\omega \in \Omega^\mathcal{M} :\bigwedge_ {n\leq \Phi(\lambda,\varepsilon,g)}\, \bigvee_{i,j\in [n;n+g(n)]}(f_{i,j}(\omega)> \varepsilon)\right\}\right) > \lambda.
 \]
But that would imply $[\mu]_L([A]) > \lambda$ which is a contradiction as $[\mu]_L([A]) = [\mu]([A]) \le \lambda$. Thus $\Phi(\lambda,\varepsilon,g)$ is a rate of pointwise metastability for $\seq{f_{n,m}}$ with respect to $[\mu]_L$. Now, as possessing a rate of pointwise metastability is equivalent to almost sure convergence (cf.\ Theorem 3.2 of \cite{neri-powell:pp:martingale}), we have that $\seq{f_{n,m}}$ converges $[\mu]_L$ almost surely and so the dominated convergence theorem (Theorem \ref{thrm:DCT}) implies $\seq{I_{n,m}}$ converges and thus is metastable, that is, we have for all $\varepsilon \in \QQ^+$ and $g:\NN \to \NN$ there exists $n \in \NN$
\[
\forall i,j \in [n;n+g(n)](I_{i,j}<\varepsilon)
\]
Thus, for all $\varepsilon \in \QQ^+, g:\NN \to \NN$, and $n \in \NN$, defining $\varphi_{(g,\varepsilon),n}$ by the positive bounded formula
\[
\varphi_{(g,\varepsilon),n}:\equiv \bigvee_{i,j \in [n;n+g(n)]}(I_{i,j} \ge \varepsilon).
\]
yields  for all $\varepsilon \in \QQ^+$ and $g:\NN \to \NN$ there exists $n \in \NN$ such that $\mathcal{M}\nvDash\varphi_{(g,\varepsilon),n}$ and the result follows by Corollary \ref{cor:main:model:unif:sat} with $X = \NN^\NN \times \QQ^+$.

\begin{remark}
 We conclude by remarking that, instead of appealing to Corollary~\ref{cor:main:model:unif:sat} to prove Theorem~\ref{thrm:avigad:dominated}, one could apply the uniform metastability principle of Due\~{n}ez and Iovino \cite{DI2017}. The key insight of the proof we present was to use the outer measure translation of \cite{NeriOlivaPischke2026} to appropriately express the assumptions of Theorem~\ref{thrm:avigad:dominated} as positive bounded formulas; we note that the use of the outer measure in model theory to represent stochastic properties in positive bounded logic was already anticipated in \cite[Section~1.4]{NeriOlivaPischke2026}.

\end{remark}

\section{A proof-theoretic treatment of abstract metric spaces axiomatized in positive bound logic}
\label{sec:pt}
We now demonstrate that large classes of proofs appealing to the
model-theoretic uniformity principle (Corollary~\ref{cor:main:model:unif:sat})admit uniform bounds that can be extracted directly from those proofs. Our starting point is the system $\mathcal{A}^\omega = \WEPAomega + \QFAC +
\DC$, which formalises classical analysis in all finite types via a weakly
extensional variant of Peano arithmetic in all finite types, together with
certain choice principles (see e.g.\ \cite{kohlenbach2005some}). We briefly
sketch the key features of $\mathcal{A}^\omega$ and refer the reader to
\cite{Kohlenbach2008} for further details. Concretely, $\mathcal{A}^\omega$ is built on many-sorted first-order classical
logic, with a set of sorts (or types) $T$ defined by
\[
  0 \in T^{\Omega,S}, \quad \rho, \tau \in T \rightarrow \rho(\tau) \in T.
\]
We use natural numbers to denote pure types, writing $n+1:=0(n)$. For each
$\rho \in T^{\Omega,S}$, we have variables $x^\rho, y^\rho, z^\rho, \dots$ and
quantifiers ranging over them. The only primitive relation symbol is $=_0$,
representing equality at type $0$; equality at higher types is defined as an
abbreviation via
\[
  x^{\tau(\xi)} =_{\tau(\xi)} y^{\tau(\xi)} := \forall z^\xi \left( xz =_\tau yz \right).
\]
Beyond $=_0$, the language contains constants for zero and the successor
function, as well as constants corresponding to Sch\"onfinkel's combinators
\cite{Schoenfinkel1924}, which suffice to formalise $\lambda$-abstraction. It
further includes constants for simultaneous primitive recursion in the sense of
G\"odel \cite{Goe1958} and Hilbert \cite{Hil1926}, together with their
standard defining axioms (see \cite{Kohlenbach2008} for explicit details).

The system admits only a quantifier-free extensionality rule:
\[
  \frac{A_0 \to s =_\rho t}{A_0 \to r[s/x^\rho] =_\tau r[t/x^\rho]}
  \tag{$\QFER$}
\]
where $A_0$ is a quantifier-free formula, $s$ and $t$ are terms of type $\rho$,
and $r$ is a term of type $\tau$. Crucially, the system does not include the
full axiom of extensionality; see \cite{Kohlenbach2008} for a discussion of the
limitations that full extensionality imposes on systems suitable for program
extraction. In addition, $\QFER$ is enough to prove the following rule
\[
  \frac{A_0 \to s =_\rho t}{A_0 \to (B[s/x^\rho] \to B[t/x^\rho])}
\]
where $B$ is an arbitrary formula and $s,t$ are free for $x^\rho$ in $B$.

The representation of the real numbers within $\WEPAomega$ will be central to
the present work. We sketch the construction here and refer the reader to
Chapter~4 of \cite{Kohlenbach2008} for full details. 

Natural numbers are represented as objects of type $0$. Rational numbers are
encoded as pairs of natural numbers via the canonical pairing function $j$,
defined by
\[
  j(n^0,m^0) := \begin{cases}
    \min u \leq_0 (n+m)^2+3n+m \bigl[2u =_0 (n+m)^2+3n+m\bigr]
      & \text{if existent,} \\
    0^0 & \text{otherwise.}
  \end{cases}
\]
Through terms operating on such codes, one can primitively recursively define
the standard operations $+_\mathbb{Q}$, $\cdot_\mathbb{Q}$,
$\vert\cdot\vert_\mathbb{Q}$, and so on; the relations $=_\mathbb{Q}$,
$<_\mathbb{Q}$, and so on are likewise definable by quantifier-free formulas. For $r \in \QQ$, we write $\langle r \rangle$ for the type 0 object that codes for $r$.

The reals are represented as fast-converging Cauchy sequences with a fixed
modulus of convergence. Concretely, an object $f^1$ of type $1$ is interpreted
as a sequence of rationals, and we identify it with a real number when it
satisfies
\[
  \forall n^0\; \lvert f(n) -_{\QQ} f(n+1) \rvert_\QQ \le_\QQ 2^{-n-1}.
\]
To allow implicit quantification over the reals, we introduce the operator
$\widehat{\cdot}$, which turns any $f$ of type $1$ into a fast-converging
Cauchy sequence $\widehat{f}$ via
\[
  \widehat{f}(n) := \begin{cases}
    f(n)
      & \text{if } \forall k <_0 n\,
        \bigl(\lvert f(k) -_\mathbb{Q} f(k+1) \rvert_\mathbb{Q}
        <_\mathbb{Q} 2^{-k-1}\bigr), \\
    f(k)
      & \text{for $k <_0 n$ least with }
        \lvert f(k) -_\mathbb{Q} f(k+1) \rvert_\mathbb{Q}
        \geq_\mathbb{Q} 2^{-k-1}, \text{ otherwise.}
  \end{cases}
\]
One can verify that $\widehat{\cdot}$ ensures every type-$1$ object codes a
unique real: if $f^1$ is already a fast-converging Cauchy sequence as above,
then $\forall n^0\,(f(n) =_0 \widehat{f}(n))$. Unlike the rationals, equality
on the reals is not given by a quantifier-free formula; instead, it is expressed by a $\Pi^0_1$ formula. Similarly, $<_\RR$ and $\le_\RR$ are defined by $\Sigma^0_1$ and $\Pi^0_1$
formulas, respectively. The natural numbers and rationals embed into $\RR$ via
constant sequences, and the standard operations $+_\mathbb{R}$,
$\cdot_\mathbb{R}$, $\vert\cdot\vert_\mathbb{R}$, and so on are primitively
recursively definable. Throughout, we follow the standard convention of
suppressing type subscripts and the symbol $\cdot_\mathbb{R}$ whenever the
context makes them clear. For a rational number $r$, we $(r)_\RR:\equiv \lambda k. \langle r \rangle$ i.e.\ we recognise $r$ as a real number via the constant sequence. When it is clear, we shall just write the rational number $r$ (we typically only write $(r)_\RR$ when $r$ is a natural number).

We will also require a canonical selection of a Cauchy sequence representing a given real number. Following \cite[Defintion 17.7]{Kohlenbach2008}, for non-negative real numbers $x$ we define the the function $(x)_\circ$ by
\[
(x)_\circ(n):=j(2k_0,2^{n+1}-1),\label{def:circDef}
\]
where
\[
k_0:=\max k\left[\frac{k}{2^{n+1}}\leq x\right].
\]
We will need an extension of this function $(\cdot)_\circ$ to all real numbers. So for $x<0$, we define  $(x)_\circ(n):=-_\mathbb{Q}(\vert x\vert)_\circ(n)$. We shall need the following properties of $(\cdot)_\circ$:

\begin{lemma}[Lemma 17.8 of \cite{Kohlenbach2008} and Lemma 2.1 of \cite{Pis2023} ]
\label{lem:circprop}
Let $x\in\mathbb{R}$. Then:
\begin{enumerate}
\item $(x)_\circ$ is a representation of $x$ in the sense of the above (see again e.g.\ \cite{Kohlenbach2008}), in particular $\widehat{(x)_\circ}=_1 (x)_\circ$.
\item For $s\in [0,\infty)$, if $\vert x\vert \leq s$, then $(x)_\circ\leq_1 (s)_\circ$, i.e.\ $(x)_\circ(n)\leq (s)_\circ(n)$ for all $n\in\mathbb{N}$.
\item $(x)_\circ$ is nondecreasing (as a type 1 function), i.e.\  $(x)_\circ(n)\leq (x)_\circ(n+1)$ for all $n\in\mathbb{N}$.
\end{enumerate} 
\end{lemma}

Naturally, such an association of a real number with a fast converging Cauchy sequence will be non-effective; however, we note that if $x^0$ then $(x)_\circ$ can be given by a closed term in the system. For $x\in \NN$ one can define $(x)_\circ(n):= j(x\cdot 2^{n+2},2^{n+1}-1)$, see \cite[p.426 and p.383]{Kohlenbach2008}.

The rule obtained from $\QFER$, as above, with  $=_\rho$ replaced with $=_\RR$ is not valid in the system. For a formula $B$ and $\sigma$ a type we write $\QFER^\sigma_\RR(B(x))$ if the following rule is admissible in $\mathcal{A}^\omega$
\[
  \frac{A_0 \to \forall z^\sigma\,( s(z) =_\RR t(z))}{A_0 \to (B[s/x^{1(\sigma)}] \to B[t/x^{1(\sigma)}])}
\]
for all terms $s,t$.  In this section, we consider systems that extend $\mathcal{A}^\omega$. If $B$ is a formula in such an extension, we write $\QFER^\sigma_\RR(B(x))$ if the above rule is admissible in this extension.

The system $\mathcal{A}^\omega$ also includes the \emph{quantifier-free axiom
of choice} schema in all types,
\[
  \forall \underline{x}\,\exists \underline{y}\, A_0(\underline{x},\underline{y})
  \;\to\;
  \exists \underline{Y}\,\forall \underline{x}\, A_0(\underline{x},\underline{Y}\underline{x})
  \tag{$\mathsf{QF}$-$\mathsf{AC}$}
\]
where $A_0$ is quantifier-free and the tuples $\underline{x}$, $\underline{y}$
may range over arbitrary types, as well as the principle of dependent choice
$\DC$, comprising the schemata $\DC^{\underline{\rho}}$ for all type-tuples
$\underline{\rho}$:
\[
  \forall x^0,\underline{y}^{\underline{\rho}}\,
  \exists \underline{z}^{\underline{\rho}}\, A(x,\underline{y},\underline{z})
  \;\to\;
  \exists \underline{f}^{\underline{\rho}(0)}\,
  \forall x^0\, A\bigl(x,\underline{f}(x),\underline{f}(S(x))\bigr)
  \tag{$\DC^{\underline{\rho}}$}
\]
where $\underline{f}^{\underline{\rho}(0)}$ stands for
$f_1^{\rho_1(0)},\dots,f_k^{\rho_k(0)}$ and $A$ may be arbitrary. Since $\DC$
permits arbitrary comprehension over natural numbers, full second-order
arithmetic (in the sense of reverse mathematics \cite{simpson2009subsystems})
embeds into $\mathcal{A}^\omega$, identifying subsets of $\NN$ with their
characteristic functions.

\subsection{A formal representation of model theory for metric structures}
Let $L$ be a Henson signature and $\mathcal{T}$ a uniform $L$-theory. We extend the set of types $T$ in  $\mathcal{A}^\omega$ by introducing abstract types for each of the sorts in $L$ to form the set of types $T^L$ defined by 
\[
0,X_s\in T^L,\quad \rho,\tau\in T^L\rightarrow \rho(\tau)\in T^L,
\]
for each sort $s \in \mathbf{S}\setminus\{s_\RR\}$ from $L$. We shall write $X_{s_\RR}:\equiv 0(0)\equiv 1$, which will represent the type for the standard encoding of the reals as fast-converging Cauchy sequences.

In $\mathcal{A}^\omega$, we can define the closed term $P$ of type $0(0)$ with satisfying 
\[
\forall n^0( n >_\QQ 0 \to P(n)=_0 n) \land \forall n^0( n \le_\QQ 0 \to P(n)=_0 \langle 1\rangle).
\]
Using $P$, we are able to quantifier over the set of strictly positive rationals $\QQ^*_+$. Here we use the fact that rationals are encoded by natural numbers using the Cantor pairing function (see \cite{Kohlenbach2008} for more details). In particular, whenever we speak about a rational number $r$, we strictly mean a code for this natural number $\langle r\rangle$ (which we simply write as $r$ when the context is clear).
\begin{definition}
     $\mathcal{A}^\omega[L]$ will be the system resulting from $\mathcal{A}^\omega$ over the augmented language
including the types $X_s$ for each sort $s \in \mathbf{S}\setminus\{s_\RR\}$ from $L$ (where all the respective constants and axioms now also refer to these new types, if applicable) by extending this system, for each $f:s_1 \times \ldots s_n \to s_0 $ in $L$, with the following constant:
\begin{itemize}
    \item $\Tilde{f}$ of type $X_{s_0}(X_{s_n})\ldots(X_{s_1})$.
\end{itemize}
Furthermore, we fixed closed terms of $\mathcal{A}^\omega$:
\begin{itemize}
    \item $\Delta_f$ of type $0(0)(0)$.
    \item $\Omega_f$ of type $0(0)(0)$.
\end{itemize}
\end{definition}
We would like  $\mathcal{A}^\omega[L]$ to have axioms stating that $\Delta_f$ is a modulus of equicontinuity on bounded sets and $\Omega_f$  is a modulus of equiboundedness on bounded sets:

\begin{align*}
&\forall x_1^{X_{s_1}},\ldots,x_n^{X_{s_n}},k^0 \,\left(\bigwedge_{i\le n} (\Tilde{d}_{s_i}(x_i,a_i) <_\RR (k)_\RR) \to \Tilde{d}_{s_0}(\Tilde{f}(x_1,\ldots,x_n),a_i) \le_\RR (\Omega_f(k))_\RR \right).\tag*{$(B)_f$}\\
&\forall x_1^{X_{s_1}},\ldots,x_n^{X_{s_n}},y_1^{X_{s_1}},\ldots,y_n^{X_{s_n}},k^0,r^0\tag*{$(C)_f$} \\
&\Bigg(\bigwedge_{i\le n} (\Tilde{d}_{s_i}(x_i,a_i) <_\RR (r)_\RR\land \Tilde{d}_{s_i}(y_i,a_i) <_\RR (r)_\RR\land \Tilde{d}_{s_i}(x_i,y_i) <_\RR 2^{-\Delta_f(k,r)}) \to\\
&\Tilde{d}_{s_0}(\Tilde{f}(x_1,\ldots,x_n),\Tilde{f}(y_1,\ldots,y_n)) \le_\RR 2^{-k} \Bigg).
\end{align*}
Next, we suitably interpret the terms of $L$ in $\mathcal{A}^\omega[L]$. In the above, and throughout this article, we write $\Tilde{a}_{s_\RR} :\equiv 0_\RR$ and $\Tilde{d}_{s_\RR}(x,y) :\equiv |x-y|_\RR$ for $x^1,y^1$.
\begin{definition}
    Associate each variable, $x$, of sort $s$ in positive bounded logic with a variable, which we still call $x$, of type $X_s$ in $\mathcal{A}^\omega[L]$. For each term $t$ of positive bounded logic, define the term, $\Tilde{t}$, of $\mathcal{A}^\omega[L]$, as follows:
       \begin{itemize}
    \item If $t\equiv x$ a variable of sort $s$, $\Tilde{t}:= x$. 
    \item If $t\equiv f(t_1,\ldots,t_n)$ with $f:s_1 \times \ldots s_n \to s_0 $ a function symbol in $L$ and $t_1,\ldots,t_n$ terms of positive bounded logic with $\Tilde{t}_1,\ldots,\Tilde{t}_1$ defined. Then $\Tilde{t}:\equiv \Tilde{f}(\Tilde{t}_1,\ldots,\Tilde{t}_n)$. 
    \end{itemize}
\end{definition}
It is clear that $t \mapsto \Tilde{t}$ provides an embedding of terms of positive bounded logic into terms of $\mathcal{A}^\omega[L]$. When it is clear, we shall write $\Tilde{a}_s$ and $\Tilde{d}_s$ simply as $a_s$ and $d_s$ respectively.

Next, we define an embedding of the formulas of positive bounded logic in our system $\mathcal{A}^\omega[L]$. For a formula $\varphi$ of $\mathcal{A}^\omega[L]$, $s \in \mathbf{S}$ , and $r^0$ In the following, we introduce the abbreviations:
\[ 
\forall_rx^{X_s}\varphi:\equiv \forall x^{X_s}(\Tilde{d_s}(x,a_s)<_\RR P(r) \to \varphi ) \mbox{ and } \exists_rx^{X_s}\varphi:\equiv \exists x^{X_s}(\Tilde{d_s}(x,a_s)\le_\RR P(r) \land \varphi)
\]
Formally, $P(r)$ is a type 0 object which codes for a non-negative number. So when we write $\Tilde{d_s}(x,a_s)<_\RR P(r)$, for example, we actually mean $\Tilde{d_s}(x,a_s)<_\RR (\lambda k. P(r))$, where $(\lambda k. P(r))$ is the embedding of the nonnegative rational number encoded by $P(r)$  into the real numbers via the constant sequence. We make similar simplifications throughout.
\begin{definition}
    For each positive bounded formula $\varphi$, define the formula, $\Tilde{\varphi}$, of $\mathcal{A}^\omega[L]$, as follows:
       \begin{itemize}
    \item If $\varphi \equiv t\le r$ (respectively $r \le t$) for a term $t$ and rational $r$, then  $\Tilde{\varphi}:\equiv \Tilde{t}\le_\RR r$ (respectively $r \le_\RR \Tilde{t}$). 
    \item If $\varphi \equiv \varphi_1 \land \varphi_2$ (respectively $\varphi_1 \lor \varphi_2$) with $\Tilde{\varphi}_1$ and $\Tilde{\varphi}_2$ defined then  $\Tilde{\varphi}:\equiv \Tilde{\varphi}_1 \land\Tilde{\varphi}_2$ (respectively $\Tilde{\varphi}_1 \lor\Tilde{\varphi}_2$).
    \item If $\varphi \equiv \forall_r x \varphi$ (respectively $\exists_r x \varphi$) with $r$ positive rational and $\Tilde{\varphi}$  defined then  $\Tilde{\varphi}:\equiv \forall_r x \Tilde{\varphi}$ (respectively $\exists_r x \Tilde{\varphi}$). 
    \end{itemize}
\end{definition}
Again, it is clear that $\varphi \mapsto \Tilde{\varphi}$ embeds positive bounded formulas into $\mathcal{A}^\omega[L]$.

\begin{definition}
   We add the following axioms to the theory  $\mathcal{A}^\omega[L]$:
   \begin{itemize}
       \item $(C)_f$ and $(B)_f$ for each $f:s_1 \times \ldots s_n \to s_0 $ in $L$.
       \item (Pseudo-) metric axioms for $d_s$. That is, for each $s \in \mathbf{S}\setminus \{s_\RR\}$ we have the axioms:
       \begin{enumerate}
           \item $\forall x^{X_s}\,(d_s(x,x)=_\RR 0)$
            \item $\forall x^{X_s},y^{X_s}\,(d_s(x,y)=_\RR d_s(y,x))$
            \item $\forall x^{X_s},y^{X_s},z^{X_s}\,(d_s(x,y)\le_\RR d_s(x,z)+d_s(z,y)))$
       \end{enumerate}
   \end{itemize}
   As with $\mathcal{A}^\omega$ equality at type 0 is still the only a primitive predicate. For $s \in \mathbf{S}$, we define $x^{X_s} =_{X_s} y^{X_s}$ as $d_s(x,y)=_\RR 0$. Equality for complex types is defined as before as extensional equality using $=_0$ and $=_{X_s}$ for the base cases.
\end{definition}

\begin{definition}[cf.\ Definition 3.1 of \cite{kohlenbach2005some}]
    The full set-theoretic type structure $\mathcal{S}^{\omega,\mathscr{M}}= \langle S_\rho\rangle_{\rho \in T^L}$ over $\NN$ and the family of spaces $\mathscr{M}:=(M_s\,|\, s \in \mathbf{S})$ is defined via $\mathcal{S}_0:=\mathbb{N}$, $S_{X_s}:= M_s$, for each $ s \in \mathbf{S}\setminus \{s_\RR\}$  and
\[
\mathcal{S}_{\tau(\xi)}:=\mathcal{S}_{\tau}^{\mathcal{S}_{\xi}}.
\]
\end{definition}
 For $x \in S_1=\NN^\NN$, define $[x]$ to be the real number corresponding to the fast converging Cauchy sequence $\hat{x}$. For $x \in \mathcal{M}_{X_{s}}=M_s$ for $s \in \mathbf{S}\setminus \{s_\RR\}$ set $[x]:=x$. 
\begin{proposition}[cf.\ Definition 3.2 of \cite{kohlenbach2005some}]
\label{prop:model:set}
 Let $\mathscr{M}$ be a Henson $L$-structure with universe $(M_s\,|\, s \in \mathbf{S})$ and respective moduli of equicontinuity and equiboundedness on
bounded sets $\Delta_f$ and $\Omega_f$ for each function symbol \footnote{Here we actually mean the interpretations of the constants $\Delta_f$ and $\Omega_f$ in $\mathcal{S}^{\omega,\mathscr{M}}$.}, $f$, in $L$. Then $\mathcal{S}^{\omega,\mathscr{M}}$ becomes a model of $\mathcal{A}^\omega[L]$ by letting the variables of type $\rho$ range over $S_\rho$ if the additional constant symbols coming from $L$ interpreted appropriately, namely: 
\[
\tilde{f}(x_1,\ldots,x_n):=\begin{cases}(f^{\mathscr{M}}([x_1],\ldots,[x_n]))_\circ&\text{ if } s_0 =s_\RR,\\ f^{\mathscr{M}}([x_1],\ldots,[x_n]) &\text{otherwise}\end{cases}
\]
 for $f:s_1 \times \ldots s_n \to s_0 $ a function symbol in $L$. 
\end{proposition}

\begin{definition}[cf.\ Definition 3.2 of \cite{kohlenbach2005some}]
     A sentence of the language of $\mathcal{A}^\omega[L]$ holds in an Henson $L$-structure $\mathscr{M}$ if it is true in the models of $\mathcal{A}^\omega[L]$ obtained from $\mathcal{S}^{\omega,\mathscr{M}}$ as specified in Proposition \ref{prop:model:set}.
\end{definition}
We now introduce a collection of formulas in $\mathcal{A}^\omega[L]$ that will represent the class of positive bounded formulas taking parameters. First, we introduce the following abbreviations:
\[ 
\forall_rx^{X_s(0)}\varphi:\equiv\forall x^{X_s(0)}(\forall n^0\,(\Tilde{d_s}(x(n),a_s)<_\RR P(r))\to \varphi)
\]
and 
\[
\exists_rx^{X_s(0)}\varphi:\equiv \exists x^{X_s(0)}(\forall n\, (\Tilde{d_s}(x(n),a_s)\le_\RR P(r))\land \varphi).
\]
Furthermore, as in \cite[p. 201]{Kohlenbach2008}, for $s \in \mathbf{S}$, $k^0$ and $x^{X_s(0)}$, we can define $(\overline{x,l})$ such that   
\[
(\overline{x,l})(i)= \begin{cases} x(i) \text{ if } &i<l \\
a_s&\text{otherwise}\end{cases}
\]
where the equality is $=_1$ for $s = s_\RR$ and $=_{X_s}$ otherwise.
\begin{definition}
Let $\mathcal{PBL}[L]$ denote the class of formulas in $\mathcal{A}^\omega[L]$ of the form:
\begin{align*}
    \Theta_m(T,\underline{r},l,\underline{z}):\equiv &\forall_{r_1(\underline{z})} x_1^{X_{s_1}(0)} \exists_{r_2(\underline{z})} y_1^{X_{s_2}(0)} \ldots \forall_{r_{2m-1}(\underline{z})} x_m^{X_{s_{2m-1}}(0)}\exists_{r_{2m}(\underline{z})} y_m^{X_{s_{2m}}(0)}\\&(T((\overline{\underline{x},l(\underline{z})}),(\overline{\underline{y},l(\underline{z})}),\underline{z})=_\RR 0)
\end{align*}
For, $k$ of type $0$, $\underline{z}$ a sequence of terms with respective arbitrary types $\underline{\sigma}:= \sigma_1,\ldots,\sigma_p$, for each $1\le i \le 2m$ we have $r_i$ a closed term of type $0(\sigma_p)\ldots(\sigma_1)$, $T$ a closed term of type $1(\sigma_p)\ldots(\sigma_1)(X_{s_{2m}}(0))\ldots$ $(X_{s_{1}}(0))$, and $l$ is a closed term of type $0(\sigma_p)\ldots(\sigma_1)$. Here, and throughout, we write $(\overline{\underline{x},l(\underline{z})}):\equiv (\overline{x_1,l(\underline{z})}), \ldots, (\overline{x_m,l(\underline{z})})$ and define $(\overline{\underline{y},l(\underline{z})})$ similarly. Furthermore, for $i >0$, if $s_{2i-1}=s_\RR$ (respectively, $s_{2i}=s_\RR$) we shall assume that $T$ satisfies $\QFER^0_\RR(|T(x_i)|\le_\RR 2^{-k})$ (respectively $\QFER^0_\RR(|T(y_i)|\le_\RR  2^{-k})$) for each $k^0$, where we write $\QFER^0_\RR(|T(x_i)|\le_\RR 2^{-k})$ (respectively $\QFER^0_\RR(|T(y_i)|\le_\RR  2^{-k})$)  for $\QFER^0_\RR(A(x_i))$ (respectively $\QFER^0_\RR(B(y_i))$) with 
\[
A(x_i) :\equiv T(\underline{x},\underline{y}, \underline{z}) \le_\RR 2^{-k} \mbox{ (respectively } B(y_i) :\equiv T(\underline{x},\underline{y}, \underline{z})  \le_\RR 2^{-k} \mbox{)}.
\]
We make similar abbreviations throughout.
\end{definition}

\begin{proposition}
\label{prop:equiv:pbl:prenex}
    If $\varphi$ is a positive bounded formula then $\Tilde{\varphi}$ is provably equivalent over $\mathcal{A}^\omega[L]$ to a formula in $\mathcal{PBL}[L]$:
      \[
  \Theta_m^{\varphi}(T,\underline{r},1):\equiv\forall_{r_1} x_1^{X_{s_1}(0)} \exists_{r_2} y_1^{X_{s_2}(0)} \ldots \forall_{r_{2m-1}} x_m^{X_{s_{2m-1}}(0)}\exists_{r_{2m}} y_m^{X_{s_{2m}}(0)}\,(T((\overline{\underline{x},1}),(\overline{\underline{y},1}))=_\RR 0)
\]
for a term $T$ and (codes for) positive rationals $r_1,\ldots,r_{2m}$.
\end{proposition}
\begin{proof}
It is clear that if $\varphi$ is a positive bounded formula then $\Tilde{\varphi}$ is equivalent to a formula in prenex normal form with bounded quantifiers which is built up
from formulas $r\le_\RR \Tilde{t}$  and $\Tilde{t} \le_\RR r$ viewed as prime formulas, by and $\lor$ and $\land$, for terms $t$ in $L$. Now, $\Tilde{\varphi}$ is equivalent to 
\[
\forall_{r_1} x_1^{X_{s_1}} \exists_{r_2} y_1^{X_{s_2}} \ldots \forall_{r_{2m-1}} x_m^{X_{s_{2m-1}}}\exists_{r_{2m}} y_m^{X_{s_{2m}}}\,(Q(\underline{x},\underline{y})=_\RR 0)
\]
as in Lemma 6.12 of \cite{GuK2016}, $Q$ is constructed by induction on the complexity of $\varphi$:
\begin{enumerate}
    \item $r\le_\RR \Tilde{t}$  is replaced by $\min\{r,\Tilde{t}\}-r=_\RR 0$
    \item $\Tilde{t}\le_\RR r$  is replaced by $\min\{r,\Tilde{t}\}-\Tilde{t}=_\RR 0$
    \item $\phi =_\RR 0 \lor \psi =_\RR 0$ is replaced by $\min\{|\phi|,|\psi|\}=_\RR 0$
    \item $\phi =_\RR 0 \land \psi =_\RR 0$ is replaced by $\max\{|\phi|,|\psi|\}=_\RR 0$.
\end{enumerate}
Now we can define $T$ of type $1(X_{s_{2m}}(0))\ldots(X_{s_{1}}(0))$ via 
\[
T:\equiv \lambda \underline{x},\underline{y}. Q(x_1(0),\ldots, x_m(0),y_1(0),\ldots,y_m(0)).
\]
Now, $\QFER^0_\RR(|T(x_i)| \le_\RR  2^{-k})$ and $\QFER^0_\RR(|T(y_i)|\le_\RR  2^{-k})$ will follow from the fact that $Q$ can be shown to be uniformly continuous by induction on the complexity of $\varphi$ and the axioms $(C)_f$. Thus, we actually have extensionality implicicatively.
\end{proof}

\begin{example}
\label{ex:pbl}
In addition to the class of formulas obtained via our embedding of positive bounded logic into our system, $\mathcal{PBL}[L]$ represents an embedding of parameterised positive bounded formulas. For example, in the context of $L:= L_{\mathcal{G}}$ defined  in Section \ref{subsec:model:APP1}, we had the following positive bounded formula
\[
\mbox{$H$ does not have index at most $n$}:\equiv \forall_2 g_1\ldots\forall_2 g_n\,\left( d_{\mathcal{F}}\left(\bigcup_{i\le n}g_i\cdot H, \hat{\Omega}\right)=1  \right)
\]
where, here $H$ is a variable and $n$ is a natural number parameter. This formula can be expressed as a formula in $\mathcal{PBL}[\mathcal{L}_{G}]$ via
\[
\forall_{2} g^{X_{\Omega}(0)}\,\left( d_{\mathcal{F}}\left(\Tilde{\bigcup}_{i\le n}g(i)\cdot H, \Tilde{\hat{\Omega}}\right)-1=_\RR 0 \right),
\]
where now $n$ is a free type 0 parameter and, 
\[
\Tilde{\bigcup}_{i\le n}g(i)\cdot H:=R_{X_\mathcal{F}}(n,g(0)\tilde{\cdot} H,\lambda B,m.(B\tilde{\cup} (g(m+1)\tilde{\cdot} H)))
\]
where $R_{X_\mathcal{F}}$ is a (single) type $X_\mathcal{F}$ recursor constant (cf.\ \cite{Kohlenbach2008}, p.\ 48). From this example, one sees why it is important to allow $\mathcal{PBL}[L]$ to contain formulas with quantification over bounded finite sequences.
\end{example}
\begin{remark}
    We note that the rules corresponding to $\QFER^0_\RR(|T(x_i)| \le_\RR 2^{-k})$ and $\QFER^0_\RR($ $|T(y_i)|\le_\RR2^{-k})$ in the case where the sort $s$ is not $s_\RR$ follows from $\QFER$. $\QFER^0_\RR(|T(x_i)| \le_\RR r)$ and $\QFER^0_\RR(|T(y_i)|\le_\RR r)$ represent the minimal extensionality assumption on $=_\RR$ we need to prove our metatheorems. However, we note that when treating concrete instances of parametrised positive bounded formulas, as we did in Proposition \ref{prop:equiv:pbl:prenex} and Example \ref{ex:pbl}, one actually has that $T$ is uniformly continuous on bounded sets, which yields the extensionality implicitatively rather than as a rule.
\end{remark}

\begin{remark}
\label{rem:PBL:diff}
Our approach to embedding positive bounded formulas in a system of finite type arithmetic differs to \cite{GuK2016}, where (working with norm structures, as opposed to metric structures with a single sort), the class $\mathcal{PBL}$ is defined to capture the satisfiability of a countable collection of positive bounded formulas indexed by natural numbers $n$ and not containing further parameters. Furthermore, as \cite{GuK2016} is done in the normed setting, bounded universal quantification is taken over closed balls (cf.\ Remark \ref{rem:met:vs:normed}). The assumption that the term $T$ associated with a $\mathcal{PBL}$ is extensional with respect to real variables shall be important later on. The assumption ensures that $T$ \emph{treats} type 1 variables as real numbers and does not perform manipulations on their representations as type 1 objects.

\end{remark}

We now formally write down the approximate satisfiability of a positive bounded formula.

\begin{definition}
Let
\begin{align*}
    \Theta_m(T,\underline{r},l,\underline{z}):\equiv &\forall_{r_1(\underline{z})} x_1^{X_{s_1}(0)} \exists_{r_2(\underline{z})} y_1^{X_{s_2}(0)} \ldots \forall_{r_{2m-1}(\underline{z})} x_m^{X_{s_{2m-1}}(0)}\exists_{r_{2m}(\underline{z})} y_m^{X_{s_{2m}}(0)}\\&(T((\overline{\underline{x},l(\underline{z})}),(\overline{\underline{y},l(\underline{z})}),\underline{z})=_\RR 0)
\end{align*}
be a formula in $\mathcal{PBL}[L]$. The following formula expresses the \emph{$k$th approximate of} $\Theta_m(T,\underline{r},\underline{z})$:
 \begin{align*}
          &\Delta_{m,\mathcal{A}}(T,\underline{r},l,\underline{z},k):\equiv k>_0\max\{\ceil{-\log_2(P(r_{2i-1}))}\,|\, i \in [1;m]\} \to\\ &\forall_{P(r_1(\underline{z}))-2^{-k}} x_1^{X_{s_1}(0)} \exists_{P(r_2(\underline{z}))+2^{-k}} y_1^{X_{s_2}(0)} \ldots \forall_{P(r_{2m-1}(\underline{z}))-2^{-k}} x_m^{X_{s_{2m-1}}(0)}\exists_{P(r_{2m}(\underline{z}))+2^{-k}} y_m^{X_{s_{2m}}(0)}\\&(|T((\overline{\underline{x},l(\underline{z})}),(\overline{\underline{y},l(\underline{z})}),\underline{z})|\le_\RR 2^{-k}).
    \end{align*} 
In addition, the following formula expresses the approximate truth of
\[
\Theta_{m,\mathcal{A}}(T,\underline{r},l,\underline{z}):\equiv \forall k^0\,\Delta_{m,\mathcal{A}}(T,\underline{r},l,\underline{z},k).
\]
\end{definition}
The following proposition shows the connection between $\Theta_{m,\mathcal{A}}$ and approximate satisfiability and is immediate:

\begin{proposition}
\label{prop:approx:equiv}
    Suppose $\varphi$ is a positive bounded formula. If $\mathscr{M}$ is a Henson $L$-structure, then  
    \[
    \mathscr{M} \models_\mathcal{A} \varphi \mbox{ if and only if } \mathcal{S}^{\omega,\mathscr{M}} \models \Theta_{m,\mathcal{A}}^{\varphi}(T,\underline{r}).
    \]
\end{proposition}
We shall need the following immediate proposition later on:

\begin{proposition}
\label{prop:equiv:dual}
    Let
\begin{align*}
    \Theta_m(T,\underline{r},l,\underline{z}):\equiv &\forall_{r_1(\underline{z})} x_1^{X_{s_1}(0)} \exists_{r_2(\underline{z})} y_1^{X_{s_2}(0)} \ldots \forall_{r_{2m-1}(\underline{z})} x_m^{X_{s_{2m-1}}(0)}\exists_{r_{2m}(\underline{z})} y_m^{X_{s_{2m}}(0)}\\&(T((\overline{\underline{x},l(\underline{z})}),(\overline{\underline{y},l(\underline{z})}),\underline{z})=_\RR 0)
\end{align*}
be a formula in $\mathcal{PBL}[L]$. Set 
\begin{align*}
          &\Delta^D_{m,\mathcal{A}}(T,\underline{r},l,\underline{z},k):\equiv k>_0\max\{\ceil{-\log_2(P(r_{2i-1}))}\,|\, i \in [1;m]\} \to\\ &\Bar{\forall}_{P(r_1(\underline{z}))-2^{-k}} x_1^{X_{s_1}(0)} \Bar{\exists}_{P(r_2(\underline{z}))+2^{-k}} y_1^{X_{s_2}(0)} \ldots \Bar{\forall}_{P(r_{2m-1}(\underline{z}))-2^{-k}} x_m^{X_{s_{2m-1}}(0)}\Bar{\exists}_{P(r_{2m}(\underline{z}))+2^{-k}} y_m^{X_{s_{2m}}(0)}\\&(|T((\overline{\underline{x},l(\underline{z})}),(\overline{\underline{y},l(\underline{z})}),\underline{z})|\le_\RR 2^{-k}).
    \end{align*}
    Then $\mathcal{A}^\omega[L]$ proves $\Theta_{m,\mathcal{A}}(T,\underline{r},l,\underline{z})$ and $\forall k^0\,\Delta^D_{m,\mathcal{A}}(T,\underline{r},l,\underline{z},k)$ are equivalent. Where we write:
\[ 
\Bar{\forall}_rx^{X_s(0)}\varphi:\equiv\forall x^{X_s(0)}(\forall n^0\,(\Tilde{d_s}(x(n),a_s)\le_\RR P(r))\to \varphi)
\]
and 
\[
\Bar{\exists}_rx^{X_s(0)}\varphi:\equiv \exists x^{X_s(0)}(\forall n\, (\Tilde{d_s}(x(n),a_s)<_\RR P(r))\land \varphi).
\]
\end{proposition}

We want to add the axioms expressing approximate satisfiability of a given positive bounded formula in $\mathcal{T}$ to $\mathcal{A}^\omega[L]$. To show that such axioms will be admissible with respect to program extraction, we shall show that they are equivalent to universal formulas with respect to a suitable extension of  $\mathcal{A}^\omega[L]$. We first need to introduce a version of the retraction map from \cite{GuK2016} (see Definition 6.15 of that reference) to allow intensional expression of bounded quantification in our metric setting.
\begin{definition}
\label{def:retr}
Denote by $\mathcal{A}^\omega[L,\mathrm{retr}]$ the extension of $\mathcal{A}^\omega[L]$ where, for each sort $s\in \mathbf{S}$, we add a constant, $\mathrm{retr}_s$, of type $X_s(0)(X_s)$ and the following axioms:
\begin{itemize}
    \item $\forall x^{X_s}, r^0(\mathrm{retr}_s(x,r) =_{X_s} x \lor \mathrm{retr}_s(x,r) =_{X_s}  a_{s})$
    \item $\forall x^{X_s}, r^0(d_s(x,a_s)<_\RR P(r) \to \mathrm{retr}_s(x,r) =_{X_s}  x)$
    \item $\forall x^{X_s}, r^0(   d_s(x,a_{s})>_{\RR} P(r) \to \mathrm{retr}_s(x,r) =_{X_s}  a_{s} )$
    \item $\forall x^{X_s}, r^0(\mathrm{retr}_s(\mathrm{retr}_s(x,r),r) =_{X_s} \mathrm{retr}_s(x,r))$
\end{itemize}

Lastly, we define the term, $\Tilde{\mathrm{retr}}_s(x^{X_s(0)},r^0):= \lambda n^0.\mathrm{retr}_s(x(n),r)$.
We shall write $ \Tilde{\mathrm{retr}}_s$ simply as $\mathrm{retr}_s$ when the context is clear.
\end{definition}
\begin{remark}
\label{rem:retr}
The intended semantics of $\mathrm{retr}$ is to fix every point in the closed
ball and to map every point outside the ball to the centre. However, axioms
that fully specify this behaviour are not admissible in our system. We
therefore give an \emph{intensional} treatment of this mapping, which leaves
the behaviour on the boundary of the ball unspecified. By contrast, in the normed setting the existence of intermediate points permits
an \emph{extensional} retraction map whose semantics are fully captured by the
system (cf.\ \cite[Definition 6.15]{GuK2016}): one can define, as a term, the map that fixes every point in the ball
and sends each exterior point to the intersection of the line segment from that
point to the centre with the boundary of the ball.
\end{remark}
\begin{proposition}
\label{prop:retr:int}
Let $\mathscr{M}$ be a Henson $L$-structure with universe $(M_s\,|\, s \in \mathbf{S})$ and respective moduli of equicontinuity and equiboundedness on
bounded sets $\Delta_f$ and $\Omega_f$ for each function symbol, $f$, in $L$. Then $\mathcal{S}^{\omega,\mathscr{M}}$ becomes a model of $\mathcal{A}^\omega[L,\mathrm{retr}]$ by letting the variables of type $\rho$ range over $S_\rho$ if all constant of $\mathcal{A}^\omega[L]$ are interpreted as in Proposition \ref{prop:model:set} and the constants $\mathrm{retr}_s$ are interpreted via:
\[
\mathrm{retr}_s(x^{X_s},r):=\begin{cases} x \text{ if } &d_s(x,a_s) \le_\RR P(r) \\a_s&\text{otherwise}\end{cases}
\]
\end{proposition}
For $s \in \mathbf{S}\setminus \{s_\RR\}$ and 
\[
\mathrm{retr}_{s_\RR}(x^1,r):=\begin{cases} ([x])_\circ \text{ if } &|x|_\RR \le_\RR P(r) \\(0)_\circ&\text{otherwise.}\end{cases}
\]
We say a formula $A(x^{X_s(0)})$ is uniformly continuous if $\mathcal{A}^\omega[L,\mathrm{retr}]$ proves $\exists k^0\forall x^{X_s(0)}\, (A((\overline{x,k})) \leftrightarrow A(x))$.
\begin{proposition}
\label{prop:retr}
Let  $s \in \mathbf{S}$ and $A(x^{X_s(0)})$ be a uniformly contentious formula. If $s=s_\RR$ suppose we have \textup{\textsf{QF}\mbox{-}\textsf{ER}}$^0_\RR(A(x))$. We have the following:
    \begin{itemize}
         \item $\mathcal{A}^\omega[L,\mathrm{retr}]\vdash \forall r^0\,( \forall x^{X_s(0)}A(\mathrm{retr}_s(x,r)) \to \forall_{r} x^{X_s(0)}A(x)).$
         \item $\mathcal{A}^\omega[L,\mathrm{retr}]\vdash  \forall r^0\,(\exists x^{X_s(0)}A(\mathrm{retr}_s(x,r))\to \exists_{r}x^{X_s(0)}A(x)).$
    \end{itemize}
  
\end{proposition}
\begin{proof}
  Take $k^0$ realising that $A(x)$ is uniformly continuous.
  \begin{itemize}
      \item Take $x^{X_s}$. Then the second axiom of $\mathrm{retr}$ (cf.\ Definition \ref{def:retr}) implies, that
      \[
      \forall n <_0 k\, (d_s(x(n),a_s)<_\RR P(r)) \to \forall n^0\,((\overline{\mathrm{retr}_s(x,r),k})(n) =_{X_s}  (\overline{x,k})(n))
      \]
       So $\QFER$ ($\QFER^0_\RR(A(x))$ in the case $s= s_\RR$) implies\footnote{An easy consequence of $\QFER$ is the extensionality rule for existential premises, as one may prenex the existential quantifier as universal and apply $\QFER$ with the universally quantified variable treated as a parameter. More precisely, if $B_0(z^\sigma)$ is quantifier-free and $C(x^\rho)$ is an arbitrary formula, then if $\exists z^\sigma B_0(z) \to u=_{\rho} v$ we have $\forall z^\sigma (B_0(z) \to u=_{\rho} v)$ and so for all $z^\sigma$, applying $\QFER$ to $B_0(z) \to u=_{\rho} v$ yields that $B_0(z) \to (C(u)\leftrightarrow C(v))$ which implies $\exists z^\sigma B_0(z) \to (C(u)\leftrightarrow C(v))$.}  
         \[
      \forall n <_0 k\, (d_s(x(n),a_s)<_\RR P(r)) \to (A(\overline{\mathrm{retr}_s(x,r),k})) \leftrightarrow  A((\overline{x,k})))
      \]
     and the result follows from the uniform continuity of $A$.
      \item Let $r^0$ be given. If we have $x^{X_s(0)}$ satisfying $A(\mathrm{retr}_s(x,r))$ then we are done once we have shown $\forall n^0\,(d_s(\mathrm{retr}_s(x(n),r),a_s)\le_{\RR} P(r))$. The first axiom of $\mathrm{retr}$ (cf.\ Definition \ref{def:retr}) yields $\forall n^0\,(\mathrm{retr}_s(x(n),r) =_{X_s} x(n) \lor \mathrm{retr}_s(x(n),r) =_{X_s}  a_{s})$, so if for some $n$ we have $\mathrm{retr}_s(x(n),r) =_{X_s}  a_{s}$ then $d_s(\mathrm{retr}_s(x(n),r),a_s) =_\RR 0\le_{\RR} P(r)$ and if not then we have $\mathrm{retr}_s(x(n),r) =_{X_s} x(n)$ and by the third axiom, we have $d_s(x(n),a_s) \le P(r)$  which implies $d_s(\mathrm{retr}_s(x(n),r),a_s) =_\RR d_s(x(n),a_s)\le_{\RR} P(r)$.
  \end{itemize}

\end{proof}
\begin{remark}
 A result analogous to Proposition \ref{prop:retr} is established in
Lemma 6.16 of \cite{GuK2016}, in the context of their extensionally defined
retraction map (cf.\ Remark \ref{rem:retr}) and with bounded universal
quantification taken over closed balls (cf.\ Remarks \ref{rem:met:vs:normed}
and \ref{rem:PBL:diff}). In \cite{GuK2016}, the authors establish the converse of both implications in
their setting. By contrast, each direction of Proposition \ref{prop:retr}
yields only a partial converse of the other; however, this will suffice for
our purposes. A further difference is that in~\cite{GuK2016} the formula $A$ is required to
be extensional in the free variable $x^{\Tilde{X}_{s}}$, for sorts $s \neq s_\RR$.
With our intensionally defined retraction map, we only need an extensionailty rule for $s_\RR$.
\end{remark}

Using $\mathrm{retr}$ we define an intensional notion of approximate satisfiability:

\begin{definition}
    Let
\begin{align*}
    \Theta_m(T,\underline{r},l,\underline{z}):\equiv &\forall_{r_1(\underline{z})} x_1^{X_{s_1}(0)} \exists_{r_2(\underline{z})} y_1^{X_{s_2}(0)} \ldots \forall_{r_{2m-1}(\underline{z})} x_m^{X_{s_{2m-1}}(0)}\exists_{r_{2m}(\underline{z})} y_m^{X_{s_{2m}}(0)}\\&(T((\overline{\underline{x},l(\underline{z})}),(\overline{\underline{y},l(\underline{z})}),\underline{z})=_\RR 0)
\end{align*}
be a formula in $\mathcal{PBL}[L]$. We define
 \begin{align*}
          &\Delta^\mathrm{retr}_{m,\mathcal{A}}(T,\underline{r},l,\underline{z},k):\equiv k>_0\max\{\ceil{-\log_2(P(r_{2i-1}))}\,|\, i \in [1;m]\} \to\\ &\forall x_1^{X_{s_1}(0)} \exists y_1^{X_{s_2}(0)} \ldots \forall x_m^{X_{s_{2m-1}}(0)}\exists y_m^{X_{s_{2m}}(0)}(|T((\overline{\mathrm{retr}(\underline{x}),l(\underline{z})}),(\overline{\mathrm{retr}(\underline{y}),l(\underline{z})}),\underline{z})|\le_\RR 2^{-k}).
    \end{align*} 
Where we write $(\overline{\mathrm{retr}(\underline{x}),l(\underline{z})})$ for the tuple $(\overline{\mathrm{retr}_{s_1}(x_1,P(r_1(\underline{z}))-2^{-k}),l(\underline{z})}),\ldots$ and define  $(\overline{\mathrm{retr}(\underline{y}),l(\underline{z})})$ similarly. In addition, we define
\[
\Theta^\mathrm{retr}_{m,\mathcal{A}}(T,\underline{r},l,\underline{z}):\equiv \forall k^0\,\Delta^\mathrm{retr}_{m,\mathcal{A}}(T,\underline{r},l,\underline{z},k).
\]
\end{definition}
\begin{proposition}
\label{prop:ext:approx}
    Let
\begin{align*}
    \Theta_m(T,\underline{r},l,\underline{z}):\equiv &\forall_{r_1(\underline{z})} x_1^{X_{s_1}(0)} \exists_{r_2(\underline{z})} y_1^{X_{s_2}(0)} \ldots \forall_{r_{2m-1}(\underline{z})} x_m^{X_{s_{2m-1}}(0)}\exists_{r_{2m}(\underline{z})} y_m^{X_{s_{2m}}(0)}\\&(T((\overline{\underline{x},l(\underline{z})}),(\overline{\underline{y},l(\underline{z})}),\underline{z})=_\RR 0)
\end{align*}
be a formula in $\mathcal{PBL}[L]$. Then
\[
\mathcal{A}^\omega[L,\mathrm{retr}]\vdash \Theta^\mathrm{retr}_{m,\mathcal{A}}(T,\underline{r},l,\underline{z}) \leftrightarrow \Theta_{m,\mathcal{A}}(T,\underline{r},l,\underline{z}).
\]
\end{proposition}
\begin{proof}
 Proposition $\ref{prop:retr}$ immediately yields that for all $k^0$
\[
\Delta^D_{m,\mathcal{A}}(T,\underline{r},l,\underline{z},k) \to \Delta^\mathrm{retr}_{m,\mathcal{A}}(T,\underline{r},l,\underline{z},k) \to \Delta_{m,\mathcal{A}}(T,\underline{r},l,\underline{z},k)\tag{$\star$}
\]
and the result follows from Proposition \ref{prop:equiv:dual}. Note that $\QFER$ implies that for each $i$ the formulas $|T((\overline{x_i,l(\underline{z})}))|\leq_\RR 2^{-k}$ and $|T((\overline{y_i,l(\underline{z})}))|\leq_\RR 2^{-k}$ are uniformly continuous.
\end{proof}

\begin{definition}[cf.\ Definition 6.26 of \cite{GuK2016}]
\label{def:phi}
Denote by $\mathcal{A}^\omega[L,\mathrm{retr},\phi]$ the extension of $\mathcal{A}^\omega[L,\mathrm{retr}]$ where, for each $s\in \mathbf{S}$, we add a constant, $\phi_s$, of type $X_s(0)(0)(0(X_s(0)))$ and the axiom:
\[
   \forall x^{X_s(0)},r^0,z^{0(X_s(0))}\,(z(\mathrm{retr}_s(x,r))=_0 0 \to z(\mathrm{retr}_s(\phi_s(z,r),r))=_0 0)\tag{$\phi_s$}
\]
\end{definition}
\begin{proposition}[cf.\ Definition 6.27 of \cite{GuK2016}]
\label{prop:int:phi}
Let $\mathscr{M}$ be a Henson $L$-structure with universe $(M_s\,|\, s \in \mathbf{S})$ and respective moduli of equicontinuity and equiboundedness on
bounded sets $\Delta_f$ and $\Omega_f$ for each function symbol, $f$, in $L$. Then $\mathcal{S}^{\omega,\mathscr{M}}$ becomes a model of $\mathcal{A}^\omega[L,\mathrm{retr},\phi]$ by letting the variables of type $\rho$ range over $S_\rho$ if all constant of $\mathcal{A}^\omega[L,\mathrm{retr}]$ are interpreted as in Proposition \ref{prop:retr:int} and the constants $\phi_s$ are interpreted by any function with the
semantics:
\[
\phi_s(z^{0(X_s(0))},r):= \begin{cases}\mathrm{retr}_s(x,r)&\text{for }x^{X_s(0)}\text{ with }z(\mathrm{retr}_s(x,r))=_00\text{ if existent},\\ \lambda n. a_s&\text{otherwise}\end{cases}
\]
\end{proposition}

We now show that $\mathcal{A}^\omega[L,\mathrm{retr},\phi]$ prove that the approximate satisfiability of formulas in $\mathcal{PBL}$ is equivalent to a universal formula. 
\begin{proposition}
\label{prop:approx:univ}
Let
\begin{align*}
    \Theta_m(T,\underline{r},l,\underline{z}):\equiv &\forall_{r_1(\underline{z})} x_1^{X_{s_1}(0)} \exists_{r_2(\underline{z})} y_1^{X_{s_2}(0)} \ldots \forall_{r_{2m-1}(\underline{z})} x_m^{X_{s_{2m-1}}(0)}\exists_{r_{2m}(\underline{z})} y_m^{X_{s_{2m}}(0)}\\&(T((\overline{\underline{x},l(\underline{z})}),(\overline{\underline{y},l(\underline{z})}),\underline{z})=_\RR 0)
\end{align*}
be a formula in $\mathcal{PBL}[L]$. Then there is a quantifier-free formula $\theta_\mathrm{qf}(k^0,\underline{z})$ in the language of  $\mathcal{A}^\omega[L,\mathrm{retr},\phi]$ such that
\[
\mathcal{A}^\omega[L,\mathrm{retr},\phi]\vdash \Theta_{m,\mathcal{A}}(T,\underline{r},l,\underline{z}) \leftrightarrow \forall k^0\theta_\mathrm{qf}(k,\underline{z}).
\]
\end{proposition}
\begin{proof}
We show that
\[
\mathcal{A}^\omega[L,\mathrm{retr},\phi]\vdash \Theta^\mathrm{retr}_{m,\mathcal{A}}(T,\underline{r},l,\underline{z}) \leftrightarrow \forall k^0\theta_\mathrm{qf}(k,\underline{z}).
\]
and the result follows follows from Proposition \ref{prop:ext:approx}. Fix $T,\underline{r},\underline{z}$ and let 
    \begin{align*}
          &\tilde{\Delta}^\mathrm{retr}_{m,\mathcal{A}}(T,\underline{r},l,\underline{z},k):\equiv k>_0\max\{\ceil{-\log_2(P(r_{2i-1}))}\,|\, i \in [1;m]\} \to\\ &\forall x_1^{X_{s_1}(0)} \exists y_1^{X_{s_2}(0)} \ldots \forall x_m^{X_{s_{2m-1}}(0)}\exists y_m^{X_{s_{2m}}(0)}
\,(\widehat{|T((\overline{\mathrm{retr}(\underline{x}),l(\underline{z})}),(\overline{\mathrm{retr}(\underline{y}),l(\underline{z})}),\underline{z})|}(k) -_\QQ 2^{-k}<_\QQ 2^{-k}).
    \end{align*} 
Now, we have, for all $k^0$
\[
\tilde{\Delta}^\mathrm{retr}_{m,\mathcal{A}}(T,\underline{r},l,\underline{z},k+2) \to \Delta^\mathrm{retr}_{m,\mathcal{A}}(T,\underline{r},l,\underline{z},k) \to \tilde{\Delta}^\mathrm{retr}_{m,\mathcal{A}}(T,\underline{r},l,\underline{z},k)\tag{$\dagger$}
\]
since for all $\underline{x},\underline{y}$, \cite[Lemma 4.2]{Kohlenbach2008} yields that 
\[
\widehat{|T(\underline{x},\underline{y},\underline{z})|}(k+2) -_\QQ 2^{-(k+2)}<_\QQ 2^{-(k+2)}
\]
implies  $|T(\underline{x},\underline{y},l,\underline{z})|<_\RR  3 \cdot 2^{-(k+2)}\le _\RR 2^{-k}$ and the first implication follows. The second implication is immediate from \cite[Lemma 4.2]{Kohlenbach2008}. These implications imply that $\Theta^{\mathrm{retr}}_{m,\mathcal{A}}(T,\underline{r},l,\underline{z})$ is equivalent to $\forall k^0\,\tilde{\Delta}^\mathrm{retr}_{m,\mathcal{A}}(T,\underline{r},l,\underline{z},k)$. Thus, our goal, will be to show $\tilde{\Delta}^\mathrm{retr}_{m,\mathcal{A}}(T,\underline{r},\underline{z},l,k)$ is equivalent to a quantifier-free formula, from which the result will follow. First observe that 
\[
\widehat{|T((\overline{\underline{x},l}),(\overline{\underline{y},l}),\underline{z})|}(k) -_\QQ 2^{-k}<_\QQ 2^{-k}.
\]
is quantifier-free. Call this formula $\theta^1_{\mathrm{qf}}(k,\underline{x},\underline{y},\underline{z})$.  Then there is a closed term $t_\theta$ that satisfies
\[
t_{\theta^1}(k,\underline{x},\underline{y},\underline{z})=_0 0 \leftrightarrow \theta^1_{\mathrm{qf}}(k,\underline{x},\underline{y},\underline{z}). 
\]
Now we apply $(\phi_{s_{2m}})$ to $z := \lambda y_m^{X_{s_{2m}}(0)}. t_{\theta^1}(k,\underline{x},\underline{y},y_m,\underline{z})$ and $P(r_{2m}(\underline{z}))+2^{-k}$, yields
\[
 \exists y_m^{X_{s_{2m}}(0)}\,\theta^1_{\mathrm{qf}}(k,\underline{x},\underline{y},\mathrm{retr}_{s_{2m}}(y_{m},P(r_{2m}(\underline{z}))+2^{-k}),\underline{z})
\]
is equivalent to the quantifier-free formula
\[
\theta^2_\mathrm{qf}(k,\underline{x},\underline{y},\underline{z}):=\theta^1_{\mathrm{qf}}(k,\underline{x},\underline{y},\mathrm{rect}_{s_{2m}}(\phi_{s_{2m}}(z,P(r_{2m}(\underline{z}))+2^{-k})),P(r_{2m}(\underline{z}))+2^{-k}),\underline{z}).
\]
Though a similar argument, considering 
\[
\neg(\exists x_m^{X_{s_{2m-1}}} \neg \theta^2_\mathrm{qf}(k,\underline{x},\mathrm{retr}(x_m,P(r_{2m-1}(\underline{z}))-2^{-k}),\underline{y},\underline{z})) 
\]
we can find a quantifier-free $\theta^3_\mathrm{qf}(k,\underline{x},\underline{y},\underline{z})$ equivalent to
\[
 \forall x_m^{X_{s_{2m-1}}}\exists y_m^{X_{s_{2m}}}\,\theta^1_{\mathrm{qf}}(k,\underline{x},\mathrm{retr}(x_m,P(r_{2m-1}(\underline{z}))-2^{-k}),\underline{y},\mathrm{retr}_{s_{2m}}(y_{m},P(r_{2m}(\underline{z}))+2^{-k}),\underline{z}).
\]
Continuing, yields a quantifier-free formula equivalent to $\tilde{\Delta}^\mathrm{retr}_{m,\mathcal{A}}(T,\underline{r},l,\underline{z},k)$. The result follows. 
\end{proof}

\begin{definition}
    We shall write $\mathcal{A}^\omega[L,\mathcal{T}]$ for the extension of $\mathcal{A}^\omega[L,\mathrm{retr},\phi]$ with axioms $\Theta_{m,\mathcal{A}}^\varphi(T,\underline{r},1)$ for each $\varphi$ in the theory $\mathcal{T}$.
\end{definition}

\begin{proposition}[cf.\ Definition 3.2 of \cite{kohlenbach2005some}]
\label{prop:model:set:T}
 Let $\mathscr{M}\models_\mathcal{A} \mathcal{T}$ be a Henson $L$-structure with universe $(M_s\,|\, s \in \mathbf{S})$ modelling $\mathcal{T}$ and respective moduli of equicontinuity and equiboundedness on
bounded sets $\Delta_f$ and $\Omega_f$ for each function symbol, $f$, in $L$. Then $\mathcal{S}^{\omega,\mathscr{M}}$ becomes a model of $\mathcal{A}^\omega[L,\mathcal{T}]$ by letting the variables of type $\rho$ range over $S_\rho$ if all constant of $\mathcal{A}^\omega[L,\mathrm{retr},\phi]$ are interpreted as in Proposition \ref{prop:int:phi}.
\end{proposition}

\begin{definition}[cf.\ Definition 3.2 of \cite{kohlenbach2005some}]
     A sentence of the language of $\mathcal{A}^\omega[L]$ holds in a model of $\mathcal{T}$ if it holds in the models of $\mathcal{A}^\omega[L,\mathcal{T}]$ obtained from $\mathcal{S}^{\omega,\mathscr{M}}$ (cf.\ Proposition \ref{prop:model:set:T}).
\end{definition}

\subsection{A bound extraction theorem for $\mathcal{A}^\omega[L,\mathcal{T}]$}
The main result of this section will be to obtain a bound extraction theorem, similar to Proposition 6.35 of \cite{GuK2016}, from which our proof-theoretic version of the uniformity principle, Corollary \ref{cor:main:model:unif:sat}, will be a consequence. The main tool for the bound extraction theorem G\"odel's Dialectica interpretation \cite{Goe1958} combined with a negative translation by Kuroda \cite{Kur1951}.

\begin{definition}[\cite{Goe1958,Tro1973}]
The Dialectica interpretation $A^D=\exists\underline{x}\forall\underline{y} A_D(\underline{x},\underline{y})$ of a formula $A$ in the language of $\mathcal{A}^\omega[L,\mathcal{T}]$  is defined via the following recursion on the structure of the formula:
\begin{enumerate}
\item $A^D:=A_D:=A$ for $A$ being a prime formula.
\end{enumerate}
If $A^D=\exists\underline{x}\forall\underline{y} A_D(\underline{x},\underline{y})$ and $B^D=\exists\underline{u}\forall\underline{v} B_D(\underline{u},\underline{v})$, we set
\begin{enumerate}
\setcounter{enumi}{1}
\item $(A\land B)^D:=\exists\underline{x},\underline{u}\forall\underline{y},\underline{v}(A\land B)_D$\\ where $(A\land B)_D(\underline{x},\underline{u},\underline{y},\underline{v}):=A_D(\underline{x},\underline{y})\land B_D(\underline{u},\underline{v})$,
\item $(A\lor B)^D:=\exists z^0,\underline{x},\underline{u}\forall\underline{y},\underline{v}(A\lor B)_D$\\ where $(A\lor B)_D(z^0,\underline{x},\underline{u},\underline{y},\underline{v}):=(z=0\rightarrow A_D(\underline{x},\underline{y}))\land (z\neq 0\rightarrow B_D(\underline{u},\underline{v}))$,
\item $(A\rightarrow B)^D:=\exists\underline{U},\underline{Y}\forall\underline{x},\underline{v}(A\rightarrow B)_D$\\ where $(A\rightarrow B)_D(\underline{U},\underline{Y},\underline{x},\underline{v}):=A_D(\underline{x},\underline{Y}\underline{x}\underline{v})\to B_D(\underline{U}\underline{x},\underline{v})$,
\item $(\exists z^\tau A(z))^D:=\exists z,\underline{x}\forall\underline{y}(\exists z^\tau A(z))_D$\\ where $(\exists z^\tau A(z))_D(z,\underline{x},\underline{y}):=A_D(\underline{x},\underline{y},z)$,
\item $(\forall z^\tau A(z))^D:=\exists\underline{X}\forall z,\underline{y}(\forall z^\tau A(z))_D$\\ where $(\forall z^\tau A(z))_D(\underline{X},z,\underline{y}):=A_D(\underline{X}z,\underline{y},z)$.
\end{enumerate}
\end{definition}

\begin{definition}[\cite{Kur1951}]
The negative translation of $A$ is defined by $A':=\neg\neg A^*$ where $A^*$ is defined by the following recursion on the structure of $A$:
\begin{enumerate}
\item $A^*:= A$ for prime $A$;
\item $(A\circ B)^*:= A^*\circ B^*$ for $\circ\in\{\land,\lor,\rightarrow\}$;
\item $(\exists x^\tau A)^*:=\exists x^\tau A^*$;
\item $(\forall x^\tau A)^*:=\forall x^\tau \neg\neg A^*$.
\end{enumerate}
\end{definition}
We now have the following soundness result that allows us to extract quantitative bounds for suitable reformulations of theorems proven in $\mathcal{A}^\omega[L,\mathcal{T}]$. Let us write $\mathcal{A}^\omega[L,\mathcal{T}]^-$ for $\mathcal{A}^\omega[L,\mathcal{T}]$ without $\QFAC$ and $\DC$.

\begin{proposition}
[essentially \cite{kohlenbach2005some}]\label{lem:ndinterpretation}
Let $A(\underline{a})$ be an arbitrary formula (with only the variables $\underline{a}$ free) in the language of $\mathcal{A}^\omega[L,\mathcal{T}]$. Then the rule
\[
\begin{cases}\mathcal{A}^\omega[L,\mathcal{T}]\vdash A(\underline{a})\Rightarrow\\
\mathcal{A}^\omega[L,\mathcal{T}]^-+(\mathrm{BR})\vdash\forall\underline{a},\underline{y}(A')_D(\underline{t}\underline{a},\underline{y},\underline{a})\end{cases}
\]
holds where $\underline{t}$ is a tuple of closed terms of $\mathcal{A}^\omega[L,\mathcal{T}]^-+(\mathrm{BR})$ which can be extracted from the respective proof and $(\mathrm{BR})$ is the schema of \emph{simultaneous bar-recursion} of Spector \cite{Spe1962}, here extended to all types from $T^{L}$ (similar as in e.g.\ \cite{Kohlenbach2008}). Furthermore, this result extends to any extension of the language of $\mathcal{A}^\omega[L,\mathcal{T}]$ by new types and constants together with any number of additional universal axioms in that language (noting that the additional axioms added to $\mathcal{A}^\omega$ to get $\mathcal{A}^\omega[L,\mathcal{T}]$ are (equivalent to) universal axioms.
\end{proposition}
The proof of this result is exactly as the analogous soundness result in \cite{kohlenbach2005some}, so we omit it.

\begin{remark}
\label{rem:soundness}
If $A:\equiv \forall x\,\exists y\, A_0(x,y)$ for $A_0$ quantifier-free then $(A')_D$ will be equivalent to $\exists f\, \forall x\, A(x,f(x))$ and \ref{lem:ndinterpretation} will guarantee that we can extract a term $t$ witnessing $\exists f$ from the proof of $A$. However, it is well known $(\mathrm{BR})$ is not a set-theoretically valid principle (cf.\ p.214 of \cite{Kohlenbach2008}) and so the validity of $t$ in models of $\mathcal{T}$ is in question. 
\end{remark}
\begin{definition}[essentially \cite{GeK2008}]
Define $\widehat{\tau}\in T$, given $\tau\in T^{\Omega,S}$, by recursion on the structure via
\[
\widehat{0}:=0,\;\widehat{\Omega}:=0,\;\widehat{S}:=0,\;\widehat{\tau(\xi)}:=\widehat{\tau}(\widehat{\xi}).
\]
\end{definition}

The majorizability relation $\gtrsim$ is then defined in tandem with the structure of all strongly majorizable functionals.

\begin{definition}[essentially Definition 9.1 of \cite{GeK2008}]
Let $\mathscr{M}$ be a Henson $L$-structure with universe $(M_s\,|\, s \in \mathbf{S})$, metrics $(d_s\,|\, s \in \mathbf{S})$ and reference points $(a_s\,|\, s \in \mathbf{S})$ that models $\mathcal{T}$. The structure $\mathcal{M}^{\omega,\mathscr{M}}$ and the majorizability relation $\gtrsim_\rho$ are defined,
\[
\begin{cases}
\mathcal{M}_0:=\mathbb{N}, n\gtrsim_0 m:=n\geq m\land n,m\in\mathbb{N},\\
\mathcal{M}_{X_s}:= M_s, n\gtrsim_s x:= n\geq d_s(x,a_s)\land n\in \mathcal{M}_0,x\in \mathcal{M}_s,\\
f\gtrsim_{\tau(\xi)}x:=f\in \mathcal{M}_{\widehat{\tau}}^{\mathcal{M}_{\widehat{\xi}}}\land x\in \mathcal{M}_\tau^{\mathcal{M}_\xi}\\
\qquad\qquad\qquad\land\forall g\in \mathcal{M}_{\widehat{\xi}},y\in \mathcal{M}_\xi(g\gtrsim_\xi y\rightarrow fg\gtrsim_\tau xy)\\
\qquad\qquad\qquad\land\forall g,y\in \mathcal{M}_{\widehat{\xi}}(g\gtrsim_{\widehat{\xi}}y\rightarrow fg\gtrsim_{\widehat{\tau}}fy),\\
\mathcal{M}_{\tau(\xi)}:=\left\{x\in \mathcal{M}_\tau^{\mathcal{M}_\xi}\mid \exists f\in \mathcal{M}^{\mathcal{M}_{\widehat{\xi}}}_{\widehat{\tau}}:f\gtrsim_{\tau(\xi)}x\right\}.
\end{cases}
\]
where $s$ ranges over $\mathbf{S}\setminus\{s_\RR\}$ and the additional constant symbols coming from $L$ interpreted appropriately.
\end{definition}
\begin{definition}
    We call a type $\xi$ \emph{of degree $n$} if $\xi\in T$ and it has degree $\leq n$ in the usual sense (see e.g.\ \cite{Kohlenbach2008}). Further we call $\xi$ \emph{small} if it is of the form $\xi=\xi_0(0)\dots(0)$ for $\xi_0\in\{0\}\cup\{X_s\,|\, s \in \mathbf{S}\}$ (including $\{0\}\cup\{X_s\,|\, s \in \mathbf{S}\}$) and call it \emph{admissible} if it is of the form $\xi=\xi_0(\tau_k)\dots(\tau_1)$ where each $\tau_i$ is small and $\xi_0\in\{0,\Omega,S\}$ (also including $\{0\}\cup\{X_s\,|\, s \in \mathbf{S}\}$).
\end{definition}
By essentially the same argument presented in Lemma 5.7 of \cite{GuK2016}, for small types $\rho$ we have $M_\rho=S_\rho$, and for admissible types $\rho$, we have $M_\rho\subseteq S_\rho$ (for which it is important that admissible types take arguments of small types). 

 We shall show that $\mathcal{M}^{\omega,\mathscr{M}}$ is a model of $\mathcal{A}^\omega[L,\mathcal{T}]^-+(\mathrm{BR})$. This fact along with the relationship between $\mathcal{M}^{\omega,\mathscr{M}}$ and $\mathcal{S}^{\omega,\mathscr{M}}$ mentioned above will address the problem raised in Remark \ref{rem:soundness}. In fact we may consider further set-theoretically invalid principles: 
\begin{definition}[cf.\ Page 526 of \cite{GuK2016}]
Define 
\[
F^X:=\bigcup_{s \in \mathbf{S}}  F^{X_s}
\]
where,
    \[
     F^{X_s}:\equiv 
            \forall \Phi^{0(X_s(0))}\forall r^0 \exists_{r} y^{X_s(0)}\forall x^{X_s(0)}\,(\Phi(\mathrm{retr}_s(x,r))\le_0 \Phi(y))
    \]
\end{definition}
 Clearly $F^X$ is not a set-theoretic valid principle as it entails that all functions $\Phi: B_r(a_s) \to \NN$ are bounded. We shall see in the next section that  $F^X$ can be used to entail a large class of saturation arguments, covering many applications of saturation used in ultraproduct arguments. 
\begin{definition}
 A formula $F$ of $\mathcal{A}^\omega[L,\mathcal{T}]$ is called a $\forall$-formula (resp. $\exists$-formula) if it
has the form $F \equiv \forall \underline{a}^{\underline{\sigma}}
F_{\mathrm{qf}}(a, b)$ (resp. $F \equiv\exists \underline{a}^{\underline{\sigma}}
F_{\mathrm{qf}}(a, b)$) where $F_{\mathrm{qf}}$ does not contain any quantifiers and the types in $\underline{\sigma}$ are admissible and $b$ are parameters of arbitrary finite type.
\end{definition}

\begin{proposition}
\label{lem:majmainresult}
Let $\mathscr{M}$ be a Henson $L$-structure with universe $(M_s\,|\, s \in \mathbf{S})$, metrics $(d_s\,|\, s \in \mathbf{S})$ and reference points $(a_s\,|\, s \in \mathbf{S})$ that models $\mathcal{T}$. Then $\mathcal{M}^{\omega,\mathscr{M}}$ is a model of $\mathcal{A}^\omega[L,\mathcal{T},\mathscr{Y}]^-+\Tilde{F}^X+(\mathrm{BR})$, where $\mathcal{A}^\omega[L,\mathcal{T},\mathscr{Y}]^-$ is the extension of $\mathcal{A}^\omega[L,\mathcal{T},\mathscr{Y}]$ with an additional constant, $\mathscr{Y}$, of type $X_s(0)(0)(0(X_s(0)))$ and $\Tilde{F}^X$ is the collection of $\Tilde{F}^{X_s}$ for $s \in \mathbf{S}$ with
\[
\Tilde{F}^{X_s}:\equiv \forall \Phi^{0(X_s(0))}\,\forall r^0\, \forall x^{X_s(0)}\,(\forall n^0(d_s(\mathscr{Y}(\Phi,r)(n),a_s) \le_\RR r)\land\Phi(\mathrm{retr}_s(x,r))\le_0 \Phi(\mathscr{Y}(\Phi,r)))
\]
Moreover, for any closed term $t$ of $\mathcal{A}^\omega[L,\mathcal{T},\mathscr{Y}]^-+\Tilde{F}^X+(\mathrm{BR})$, one can construct a closed term $t^*$ of $\mathcal{A}^\omega+(\mathrm{BR})$ such that
\[
\mathcal{M}^{\omega,\mathscr{M}}\models\left( t^*\gtrsim t\right).
\]
\end{proposition}
\begin{proof}
    The proof of this result is standard and similar to those in the literature (see e.g.\ Lemma 4.7 of \cite{kohlenbach2005some} and \cite{Kohlenbach2008}). In particular, by the exact arguments in \cite{Kohlenbach2008} we have $\mathcal{M}^{\omega,\mathscr{M}}$ models $\mathcal{A}^\omega + (\mathrm{BR})$, with the axioms and constants augmented via the inclusion of the abstract types $X_s$, and we may choose suitable interpretations and majorants for terms of this system. Furthermore majorizing a composition of terms is also argued as in \cite{Kohlenbach2008}. Since the additional axioms are universal, it only remains to show that suitable interpretations of the newly introduced constants, namely $\mathscr{Y}, \phi_s, \mathrm{retr}_s$ and $\tilde{f}$, are majorizable (with $f$ a function symbol from $L$). As the types of $\mathrm{retr}_s$ and $\tilde{f}$ are small, we can interpret them as we did in $\mathcal{S}^{\omega,\mathscr{M}}$ (cf.\ Propositions \ref{prop:model:set} and \ref{prop:retr:int}) and we may construct majorants as follows: 
   
   Let $f:s_1 \times \ldots s_n \to s_0 $ a function symbol in $L$. Now, in the following, if $m$ is of type $1$, write $\Tilde{m}:=m(0)+1$ and if $m$ is of type $0$ then $\Tilde{m}:= m$. For $s_0\neq s_\RR$ we have 
\[
\lambda m_1,\ldots,m_n.\Omega^M_f(\max(\Tilde{m}_1,\ldots,\Tilde{m}_n))\gtrsim \tilde{f}
\]
where $\Omega^M_f:\equiv \lambda k. \max_{n\le k}\Omega_f(n)$ and 
\[
\lambda m_1,\ldots,m_n.(\Omega^M_f(\max(\Tilde{m}_1,\ldots,\Tilde{m}_n)))_\circ\gtrsim \tilde{f}
\]
when $s_0=s_\RR$, since, if we have $m_1,\ldots,m_n$ and $x_1,\ldots,x_n$ with $m_i \gtrsim x_i$ then if $s_i \neq s_\RR$ we have $m_i \ge d_{s_i}(x_i,a_{s_i})$, otherwise we will have $m_i(n) \ge_0 x_i(n)$ for all $n^0$, which implies $m_i(n) \ge_0 \hat{x}_i(n)$ for all $n$ (as $m_i$ is non-decreasing). So, $\langle m_i(0)+1\rangle \ge_\QQ \langle \hat{x}_i(0) +1\rangle  \ge_\QQ |\hat{x}_i(0)|_\QQ +_\QQ 1 \ge _\QQ |\hat{x}_i(n)|_\QQ$ for all $n^0$. Thus, we have $\Tilde{m}_i=m_i(0)+1 \ge_\RR |x_i|_\RR=_\RR d_{s_\RR}(x_i,a_{s_\RR})$. Thus we have $\max(\Tilde{m}_1,\ldots,\Tilde{m}_n)=:k \ge_\RR d_{s_i}(x_i,a_{s_i})$ for all $i=1,\ldots,n$, thus, $(B)_f$ implies $\Omega^M_f(k)\ge \Omega_f(k) \ge f(x_1,\ldots,x_n)$ and the result follows (using Lemma \ref{lem:circprop} for the case $s=s_\RR$). 

For $\mathrm{retr}$, we have 
\[
\lambda m^0,n^0.\max\{n,1\} \gtrsim \mathrm{retr}_s
\]
for $s \in \mathbf{S}\setminus \{s_\RR\}$ as, for all $x \in M_s$ and $n,r \in \NN$ with $n \ge r$, suppose $r$ codes for $q \in \QQ$. If $q>0$ then $d_s(\mathrm{retr}_s(x,r),a_s)\le q \le r \le n$ (where we use the fact that all codes are at least the rational number they code for) and if $q\le 0$ then $d_s(\mathrm{retr}_s(x,r),a_s)\le 1  \le \max\{n,1\} $. Furthermore, 
\[
\lambda m^1,n^0.(\max\{n,1\} )_\circ \gtrsim \mathrm{retr}_{s_\RR}
\]
as for $x \in \NN^\NN$, and $n,r \in \NN$ with $n \ge r$, suppose that $r$ codes for $q \in \QQ$. By definition, $\mathrm{retr}_{s_\RR}(x,r)= ([x])_\circ$ or $(0)_\circ$. If $\mathrm{retr}_{s_\RR}(x,r)= ([x])_\circ$ and $q>0$ then $|[x]|\le q \le n$ and if $q\le 0$ then $|[x]|\le 1 \le \max\{n,1\}$, so $([x])_\circ \le_1 (\max\{n,1\} )_\circ$, by Lemma \ref{lem:circprop}. Thus, $\mathrm{retr}_{s_\RR}(x,r)\le_1 (n)_\circ$.

Now, as the axioms that come from $\mathcal{T}$ and those for $\mathrm{retr}$ (cf.\ Definition \ref{def:retr}) are (equivalent to) $\forall$-formulas, we have that they also hold in $\mathcal{M}^{\omega,\mathscr{M}}$ (using that the variable bounded to the universal quantification of a $\forall$-formulas is of admissible type).

For $\phi_s$, we argue in the exact same way as in the proof of Theorem 6.36 of \cite{GuK2016}. We take the interpretation of $\phi_s$ in $\mathcal{M}^{\omega,\mathscr{M}}$ to be the restriction of it's interpretation on $\mathcal{S}^{\omega,\mathscr{M}}$ (cf.\ Proposition \ref{prop:int:phi}) to arguments of $\mathcal{M}^{\omega,\mathscr{M}}$. Furthermore, with this interpretation, we have $\phi_s$ is majorized by $\lambda z^{2},r^0.(\lambda n^0.  n)$. Furthermore, as the arguments of $\phi_s$ are admissible and its defining axioms $(\phi_s)$ (cf.\ Definition \ref{def:phi}) are $\forall$-formulas, we have that these axioms also hold in $\mathcal{M}^{\omega,\mathscr{M}}$.

All that remains is to deal with $\Tilde{F}^X$. The argument is as in that of Theorem 17.101 of \cite{Kohlenbach2008}. Let $\Phi^{0(X_s(0))}, r^0$ be in $\mathcal{M}^{\omega,\mathscr{M}}$ and $\Phi^*, r^*$ their respective majorants. Then, we have
\[
\forall x^{X_s(0)}\,(\Phi(\mathrm{retr}_s(x,r))\le_0 \Phi^*(\lambda n.\max\{r^*,1\}) 
\]
for $s \in \mathbf{S}\setminus \{s_\RR\}$
and 
\[
\forall x^{1(0)}\,(\Phi(\mathrm{retr}_{s_\RR}(x,r))\le_0 \Phi^*(\lambda n.(\max\{r^*,1\} )_\circ). 
\]
Thus, we have $\max\{\Phi(\mathrm{retr}_s(x,r)): x \in X_s(0)\}$ exists, and so we have 
\[
\forall \Phi\in \mathcal{M}_{0(X_s(0))}\,\forall r\in \NN\, \exists y \in \mathcal{M}_{X_s(0)} \,\forall x\in \mathcal{M}_{X_s(0)}\,(\forall n \in \NN\,(d_s(y(n),a_s) \le _\RR r)\land\Phi(\mathrm{retr}_s(x,r))\le_0 \Phi(y)).
\]
So, by the axiom of choice, we obtain a functional satisfying $\Tilde{F}^X$ and we take this for the interpretation of $\mathscr{Y}$. Thus, it remains to construct a majorant for this choice of interpretation to show that it is in $\mathcal{M}^{\omega,\mathscr{M}}$. For this, we have
\[
 \lambda \Phi^*,r^*.(\lambda n.\max\{r^*,1\} ) \gtrsim \mathscr{Y}
\]
for $s \in \mathbf{S} \setminus \{s_\RR\}$
and 
\[
 \lambda \Phi^*,r^*.(\lambda n.(\max\{r^*,1\})_\circ) \gtrsim \mathscr{Y}
\]
for $s = s_\RR$.

\end{proof}

\begin{theorem}\label{thm:metatheorem}
 Let $\rho$ be admissible and let $B_\forall(x,u)$/$C_\exists(x,v)$ be $\forall$-/$\exists$-formulas of $\mathcal{A}^\omega[L,\mathcal{T}]$ with only $x,u$/$x,v$ free. If
\[
\mathcal{A}^\omega[L,\mathcal{T}]+F^X\vdash\forall x^\rho\left(\forall u^0 B_\forall(x,u)\rightarrow\exists v^0 C_\exists(x,v)\right),
\]
then one can extract a partial functional $\Phi:\mathcal{S}_{\widehat{\rho}}\rightharpoonup\mathbb{N}$ which is total and (bar-recursively) computable on $\mathcal{M}_{\widehat{\rho}}$ and the following holds in all non-trivial models  $\mathscr{M}$ of $\mathcal{T}$: for all $x\in \mathcal{S}_\rho$, $x^*\in \mathcal{S}_{\widehat{\rho}}$, if $x^*\gtrsim x$, then
\[
\forall u\leq_0\Phi(x^*) B_\forall(x,u)\rightarrow\exists v\leq_0\Phi(x^*)C_\exists(x,v).
\]
\end{theorem}

\begin{proof}
\[
\mathcal{A}^\omega[L,\mathcal{T}]+F^X\vdash\forall x^\rho\left(\forall u^0 B_\forall(x,u)\rightarrow\exists v^0 C_\exists(x,v)\right),
\]
implies 
\[
\mathcal{A}^\omega[L,\mathcal{T},\mathscr{Y}]+\Tilde{F}^X\vdash\forall x^\rho\left(\forall u^0 B_\forall(x,u)\rightarrow\exists v^0 C_\exists(x,v)\right).
\]
Now, we have $B_\forall(x,u)=\forall\underline{a} B_{qf}(x,u,\underline{a})$ and $C_\exists(x,v)=\exists\underline{b} C_{qf}(x,v,\underline{b})$ for quantifier-free $B_{qf}$ and $C_{qf}$. So prenexiation yields
\[
\mathcal{A}^\omega[L,\mathcal{T},\mathscr{Y}]+\Tilde{F}^X\vdash\forall x^\rho\exists u,v,\underline{a},\underline{b}(B_{qf}(x,u,\underline{a})\rightarrow C_{qf}(x,v,\underline{b})).
\]
Now applying Proposition \ref{lem:ndinterpretation} (noting that $\mathcal{A}^\omega[L,\mathcal{T},\mathscr{Y}]+\Tilde{F}^X$ extends $\mathcal{A}^\omega[L,\mathcal{T}]$ only by a new constant and purely universal axioms), we get closed terms $t_u,t_v$ of $\mathcal{A}^\omega[L,\mathcal{T},\mathscr{Y}]+\Tilde{F}^X+(\mathrm{BR})$ such that
\[
\mathcal{A}^\omega[L,\mathcal{T},\mathscr{Y}]+\Tilde{F}^X+(\mathrm{BR})\vdash\forall x^\rho(B_\forall (x,t_u(x))\rightarrow C_\exists(x,t_v(x))).
\]
By Proposition \ref{lem:majmainresult} there are closed terms $t^*_u,t^*_v$ of $\mathcal{A}^\omega+(\mathrm{BR})$ such that
\[
\mathcal{M}^{\omega,\mathscr{M}}\models t^*_u\gtrsim t_u\land t^*_v\gtrsim t_v\land\forall x^\rho(B_\forall(x,t_u(x))\rightarrow C_\exists(x,t_v(x)))
\]
for all non-trivial models $\mathscr{M}$ of $\mathcal{T}$. Define $\Phi(x^*):=\max\{t^*_u(x^*),t^*_v(x^*)\}$. Then
\[
\mathcal{M}^{\omega,\mathscr{M}}\models\forall u\leq_0\Phi(x^*) B_\forall(x,u)\rightarrow\exists v\leq_0\Phi(x^*) C_\exists(x,v)
\]
holds for all $x\in \mathcal{M}_\tau$ and $x^*\in \mathcal{M}_{\widehat{\tau}}$ with $x^*\gtrsim x$. That $\mathcal{S}^{\omega,\mathscr{M}}$ satisfies the same sentence, one uses the restriction on the types and argues as in the proof of \cite[Theorem 17.52]{Kohlenbach2008}.
\end{proof}
\begin{remark}
    In Theorem \ref{thm:metatheorem}, we further have the following standard additional remarks:
    \begin{enumerate}
\item If $\widehat{\tau}$ is of degree $1$, then $\Phi$ is a total computable functional. 
\item We may have tuples instead of single variables $x,y,z,u,v$ and a finite conjunction instead of a single premise $\forall u^0 B_\forall(x,y,z,u)$.
\item If the claim is proved without $\DC$, then $\tau$ may be arbitrary and $\Phi$ will be a total functional on $ \mathcal{S}_{\widehat{\tau}}$ which is primitive recursive in the sense of G\"odel \cite{Goe1958} and Hilbert \cite{Hil1926}. In that case, also plain majorization can be used instead of strong majorization \emph{(}see e.g.\ \cite{Kohlenbach2008}\emph{)}.
\end{enumerate}
\end{remark}
\begin{remark}
\label{rem:how:to:use}
    The algorithm for extracting $\Phi$ in Theorem \ref{thm:metatheorem} is hidden in the proofs of Propositions \ref{lem:ndinterpretation} and \ref{lem:majmainresult}. In particular, this process is modular: one can construct a majorant for the realisers of the Dialectica interpretation of the negative translation of a sentence by constructing majorants  for the realisers of the Dialectica interpretation of the negative translation of the lemmas used to prove the sentence. When metatheorems such as Theorem \ref{thm:metatheorem} are applied in proof mining, one very often exploits this observation. Instead of applying the proofs of Propositions \ref{lem:ndinterpretation} and \ref{lem:majmainresult} in one go to construct bounds for the theorem in question, one can break the theorem into smaller lemmas where one can use the proofs of Propositions \ref{lem:ndinterpretation} and \ref{lem:majmainresult} or mathematical intuition to extract marjorants for the realisers of the Dialectica interpretation of the negative translation of these lemmas. Then, one can obtain bounds for the main theorem in question.
\end{remark}

\subsection{Extracting bounds from the model-theoretic uniformity principle}
\label{sec:quant:proof}
We may formalise saturation in our system as follows:
\begin{definition}
 Let
 \begin{align*}
\Theta_m(T,\underline{r},l,u,p,n,\underline{z}):\equiv &\forall_{r_1(n)} x_1^{X_{s_1}(0)} \exists_{r_2(n)} y_1^{X_{s_2}(0)} \ldots \forall_{r_{2m-1}(n)} x_m^{X_{s_{2m-1}}(0)}\exists_{r_{2m}(n)} y_m^{X_{s_{2m}}(0)}\\&(T((\overline{\underline{x},l(n)}),(\overline{\underline{y},l(n)}),(\overline{u,p}),n,\underline{z})=_\RR 0)
\end{align*}
be a formula in $\mathcal{PBL}[L]$. We define the following saturation principle, 
\[
 \mathrm{Sat}(\Theta_m):\equiv
\begin{cases}
      \forall \underline{z},p\,(\exists r^0\forall n^*\exists_{P(r)+2^{-n^*}}u^{X_s(0)}\, \forall n\le_0 n^*\, \Delta_{m,\mathcal{A}}(T,\underline{r},l,u,p,n,\underline{z},n^*)\\\to\exists u^{X_s(0)}\, \forall n^0\, \Theta_{m,\mathcal{A}}(T,\underline{r},l,u,p,n,\underline{z})) 
\end{cases}
\] 
for $\Theta_m$ satisfying $\QFER^0_\RR(|T(u)|\le_\RR  2^{-k})$ for all $k^0$ in the case $s=s_\RR$. When we refer to the above principle by $\mathrm{Sat}$, we allow all instances of such $\Theta_m$.
\end{definition}

\begin{remark}
\label{rem:F:t:sat}
$\mathrm{Sat}$ formalises the restriction of $\kappa$-saturation (cf.\ Definition \ref{def:sat}) to a countable collection of positive bounded formulas of bounded quantifier rank (number of quantifier alternations). Concretely, in the case $\kappa = \aleph_1$, suppose we are given a countable subset $C$ of the metric structure and a set of positive bounded formulas
\[
\Gamma(u_1,\ldots,u_p) := \{\varphi_n(u_1,\ldots,u_p,\underline{z}) | n \in \NN,\ \underline{z} \subseteq C\}.
\]
We may formalise $C$ in the system as an object $z^{X_s(0)}$. If, moreover, each $\varphi_n(u_1,\ldots,u_p,\underline z)$ has at most $m$ alternating quantifiers, then each such formula can be formalised as a formula of $\mathcal{PBL}[L]$ of the form $\Theta_m(T,\underline r, l, u, p, n, z)$, for appropriate closed terms $T,\underline r, l$. (We note that it is not clear how the quantifier rank of a positive bounded formula could be encoded directly in the higher-type systems we work with; bounding the number of alternations, as above, is the workaround we adopt.) With this in hand, $\mathrm{Sat}(\Theta_m)$ formalises exactly the instance of saturation corresponding to $\Gamma(u_1,\ldots,u_p)$. The case $\kappa > \aleph_1$ can be handled analogously, by encoding $C$ via higher type objects. In practice, however, and in particular in the context of the model-theoretic uniformity principle (cf.\ Corollary \ref{cor:main:model:unif:sat}), only countable saturation is used, since the principle is typically applied in the context of ultraproducts (as described in the introduction), and the formulas considered typically have bounded quantifier complexity.

It was observed in \cite{GuK2016} that a formula analogous to $F^X$ (there defined in the context of normed structures) suffices to derive a principle $\sigUB$, which in turn yields many of the properties of ultrapowers of normed structures found in the literature, properties which, on inspection, are all instances of countable saturation\footnote{The similarities between the logical form of $\sigUB$ and countable saturation was communicated to the third author by Fernando Ferreira.}. In the following section (cf.\ Proposition \ref{prop:F:to:sat}) we extend this observation by showing that our formalisation of saturation, $\mathrm{Sat}$, which, as argued above, captures the typical instances of saturation occurring in the literature, is itself provable from $F^X$. We work throughout in the metric setting, with the restricted class of positive bounded formulas formalised in our system (cf.\ Remarks \ref{rem:met:vs:normed} and \ref{rem:PBL:diff}). We note, however, that an analogous saturation principle can also be formulated in the normed setting of \cite{GuK2016}, and more generally in any setting admitting \emph{intermediate points}, such as geodesic spaces; in that setting too, the principle can be shown to follow from an appropriately defined $F^X$, together with the additional extensionality assumptions used in \cite{GuK2016}, assumptions which, in any case, hold automatically for typical parametrised positive bounded formulas. We do not pursue this direction further, as it would take us too far afield.
\end{remark}

We show that $\mathcal{A}^\omega[L,\mathrm{retr},\phi] + \mathrm{Sat}$ proves the equivalence of satisfiability and approximate satisfiability (cf.\ Theorem \ref{thrm:equiv:approxsat:sat}). 
\begin{theorem}
\label{thrm:equiv:approxsat:sat}
    Let
\begin{align*}
    \Theta_m(T,\underline{r},l,\underline{z}):\equiv &\forall_{r_1(\underline{z})} x_1^{X_{s_1}(0)} \exists_{r_2(\underline{z})} y_1^{X_{s_2}(0)} \ldots \forall_{r_{2m-1}(\underline{z})} x_m^{X_{s_{2m-1}}(0)}\exists_{r_{2m}(\underline{z})} y_m^{X_{s_{2m}}(0)}\\&(T((\overline{\underline{x},l(\underline{z})}),(\overline{\underline{y},l(\underline{z})}),\underline{z})=_\RR 0)
\end{align*}
be a formula in $\mathcal{PBL}[L]$.
\[
\mathcal{A}^\omega[L,\mathrm{retr},\phi] + \mathrm{Sat}\vdash \Theta_m(T,\underline{r},l,\underline{z}) \leftrightarrow \Theta_{m,\mathcal{A}}(T,\underline{r},l,\underline{z}).
\]
\end{theorem}
\begin{proof}
The proof will proceed by induction on $m$, with the case $m=0$ (the atomic case) trivially true. Now, for $m>0$, as the forward direction is trivial, it suffices to show $ \Theta_{m,\mathcal{A}}(T,\underline{r},l,\underline{z}) \to \Theta_m(T,\underline{r},l,\underline{z})$. Now $\Theta_{m,\mathcal{A}}(T,\underline{r},l,\underline{z})$ will be equivalent to 
\begin{align*}
    \forall k^0 &( k>_0\max\{\ceil{-\log_2(P(r_{2i-1}))}\,|\, i \in [1;m]\} \to\\ &\forall_{P(r_1(\underline{z}))-2^{-k}} x_1^{X_{s_1}(0)} \exists_{P(r_2(\underline{z}))+2^{-k}} y_1^{X_{s_2}(0)} \ldots \forall_{P(r_{2m-1}(\underline{z}))-2^{-k}} x_m^{X_{s_{2m-1}}(0)}\exists_{P(r_{2m}(\underline{z}))+2^{-k}} y_m^{X_{s_{2m}}(0)}\\&(|T((\overline{\underline{x},l(\underline{z})}),(\overline{\underline{y},l(\underline{z})}),\underline{z})|\le_\RR 2^{-k}))
\end{align*}
which is equivalent to 
\begin{align*}
    &\forall k^0\forall_{P(r_1(\underline{z}))-2^{-k}} x_1^{X_{s_1}(0)} \exists_{P(r_2(\underline{z}))+2^{-k}} y_1^{X_{s_2}(0)}( k>_0\max\{\ceil{-\log_2(P(r_{2i-1}))}\,|\, i \in [1;m]\} \to\\ & \ldots \forall_{P(r_{2m-1}(\underline{z}))-2^{-k}} x_m^{X_{s_{2m-1}}(0)}\exists_{P(r_{2m}(\underline{z}))+2^{-k}} y_m^{X_{s_{2m}}(0)}\\&(|T((\overline{\underline{x},l(\underline{z})}),(\overline{\underline{y},l(\underline{z})}),\underline{z})|\le_\RR 2^{-k}\land \forall n<_0l(\underline{z})\,(d(a_{s_2},y_1(n))\le P(r_2(\underline{z}))+2^{-k})))
\end{align*}
and will imply (crucially, we use that bounded universal quantification refers to an open ball)
\begin{align*}
     \forall_{P(r_1(\underline{z}))} x_1^{X_{s_1}(0)} \forall k^0\exists_{P(r_2(\underline{z}))+2^{-k}} y_1^{X_{s_2}(0)} \Delta_{m-1,\mathcal{A}}(\tilde{T},\underline{\tilde{r}},\tilde{l},y_1,l,\underline{z},x_1,k)
\end{align*}
for appropriately defined $\tilde{T},\underline{\tilde{r}}, \tilde{l}$ (cf.\ the proof of Proposition \ref{prop:equiv:pbl:prenex}). Applying  $\mathrm{Sat}(\Theta_{m-1})$ yields 
\[
\forall_{P(r_1(\underline{z}))} x_1^{X_{s_1}(0)} \exists y_1^{X_{s_2}(0)} \, \Theta_{m-1,\mathcal{A}}(\tilde{T},\underline{\tilde{r}},\tilde{l},y_1,l,\underline{z},x_1))
\]
which will be equivalent to 
\begin{align*}
    &\forall_{P(r_1(\underline{z}))} x_1^{X_{s_1}(0)} \exists y_1^{X_{s_2}(0)}\forall k^0( k>_0\max\{\ceil{-\log_2(P(r_{2i-1}))}\,|\, i \in [1;m]\} \to\\ & \ldots \forall_{P(r_{2m-1}(\underline{z}))-2^{-k}} x_m^{X_{s_{2m-1}}(0)}\exists_{P(r_{2m}(\underline{z}))+2^{-k}} y_m^{X_{s_{2m}}(0)}\\&(|T((\overline{\underline{x},l(\underline{z})}),(\overline{\underline{y},l(\underline{z})}),\underline{z})|\le_\RR 2^{-k}\land  \forall n<_0l(\underline{z})\,( d(a_{s_2},y_1(n))\le P(r_2(\underline{z}))+2^{-k})))
\end{align*}
which is equivalent to 
\begin{align*}
    &\forall_{P(r_1(\underline{z}))} x_1^{X_{s_1}(0)} \exists_{P(r_2(\underline{z}))} y_1^{X_{s_2}(0)}\forall k^0( k>_0\max\{\ceil{-\log_2(P(r_{2i-1}))}\,|\, i \in [1;m]\} \to\\ & \ldots \forall_{P(r_{2m-1}(\underline{z}))-2^{-k}} x_m^{X_{s_{2m-1}}(0)}\exists_{P(r_{2m}(\underline{z}))+2^{-k}} y_m^{X_{s_{2m}}(0)}(|T((\overline{\underline{x},l(\underline{z})}),(\overline{\underline{y},l(\underline{z})}),\underline{z})|\le_\RR 2^{-k})).
\end{align*}
The induction hypothesis yields that the above implies 
\begin{align*}
    &\forall_{P(r_1(\underline{z}))} x_1^{X_{s_1}(0)} \exists_{P(r_2(\underline{z}))} y_1^{X_{s_2}(0)}\ldots \forall_{r_{2m-1}(\underline{z})} x_m^{X_{s_{2m-1}}(0)}\exists_{r_{2m}(\underline{z})} y_m^{X_{s_{2m}}(0)}\,(T((\overline{\underline{x},l(\underline{z})}),(\overline{\underline{y},l(\underline{z})}),\underline{z})=_\RR 0).
\end{align*}
and the result follows.
\end{proof}
\begin{remark}
    A similar result was obtained in \cite[Theorem 6.33]{GuK2016}, in the context of normed structures, with the additional requirement that the term $T$ was extensional. The requirement of extensionality is crucial and is due to the broader class of positive bound formulas considered in \cite{GuK2016} (cf.\ Remarks \ref{rem:met:vs:normed} and \ref{rem:PBL:diff}).
\end{remark}

 We shall show that the system we have introduced, with $F^X$, proves $\mathrm{Sat}$ and as a consequence we shall obtain a proof-theoretic analogue of the model-theoretic uniformity principle, Corollary \ref{cor:main:model:unif:sat}. It will be convenient to define the following technical intensional saturation principle:
\begin{definition}
 Let
 \begin{align*}
\Theta_m(T,\underline{r},l,u,p,n,\underline{z}):\equiv &\forall_{r_1(n)} x_1^{X_{s_1}(0)} \exists_{r_2(n)} y_1^{X_{s_2}(0)} \ldots \forall_{r_{2m-1}(n)} x_m^{X_{s_{2m-1}}(0)}\exists_{r_{2m}(n)} y_m^{X_{s_{2m}}(0)}\\&(T((\overline{\underline{x},l(n)}),(\overline{\underline{y},l(n)}),(\overline{u,p}),\underline{z},n)=_\RR 0)
\end{align*}
be a formula in $\mathcal{PBL}[L]$. We define the following
intensional saturation principle
\[
\mathrm{Sat}_{-}(\Theta_m):\equiv
\begin{cases}
    \forall \underline{z},p\,(\exists r^0\forall n^*\exists u^{X_s(0)}\, \forall n\le_0 n^*\, \Delta^\mathrm{retr}_{m,\mathcal{A}}(T,\underline{r},l,\mathrm{retr}_s(u,P(r)+2^{-n^*}),p,\underline{z},n,n^*)\\
    \to \exists u^{X_s(0)}\forall n^0\, \Theta^\mathrm{retr}_{m,\mathcal{A}}(T,\underline{r},l,u,p,\underline{z},n))
\end{cases}
\] 
for $\Theta_m$ satisfying $\QFER^0_\RR(|T(u)|\le_\RR  2^{-k})$ for all $k^0$ in the case $s=s_\RR$. When we refer to the above principle by $\mathrm{Sat}_{-}$ we allow all instances of $\Theta_m$.
\end{definition}
\begin{proposition}
\label{prop:sat:int:to}
    $\mathcal{A}^\omega[L,\mathrm{retr},\phi]+\mathrm{Sat}_{-}(\Theta_m) \vdash \mathrm{Sat}(\Theta_m) $ 
\end{proposition}
\begin{proof}
    Let $\underline{z},t$ be given. Suppose we have $r^0$ such that
   \[
   \forall n^*\exists_{P(r)+2^{-n^*}}u^{X_s(0)}\, \forall n\le_0 n^*\, \Delta_{m,\mathcal{A}}(T,\underline{r},l,u,p,n,\underline{z},n^*)
   \]
   then, this will be equivalent to
    \[
   \forall n^*\exists_{P(r)+2^{-n^*}}u^{X_s(0)}\, \forall n\le_0 n^*\, \Delta^D_{m,\mathcal{A}}(T,\underline{r},l,u,p,n,\underline{z},n^*)
   \]
   which implies, by $(\star)$,
      \[
   \forall n^*\Bar{\exists}_{P(r)+2^{-n^*}}u^{X_s(0)}\, \forall n\le_0 n^*\, \Delta^\mathrm{retr}_{m,\mathcal{A}}(T,\underline{r},l,u,p,n,\underline{z},n^*).
   \]
    Now, Proposition \ref{prop:retr} yields that $\forall n^*\Bar{\exists}_{P(r)+2^{-n^*}}u^{X_s(0)}\, \forall n\le_0 n^*\, \Delta^\mathrm{retr}_{m,\mathcal{A}}(T,\underline{r},l,u,p,n,\underline{z},n^*)$ implies $\forall n^*\exists u^{X_s(0)}\, \forall n\le_0 n^*\, \Delta^\mathrm{retr}_{m,\mathcal{A}}(T,\underline{r},l,\mathrm{retr}_s(u,P(r)+2^{-n^*}),p,n,\underline{z},n^*)$. Thus $\mathrm{Sat}_{-}(\Theta_m)$ yields $\exists u^{X_s(0)}\, \forall n^0$  $\Theta^\mathrm{retr}_{m,\mathcal{A}}(T,\underline{r},l,u,p,n,\underline{z})$ and the result follows from Proposition \ref{prop:ext:approx}.
\end{proof}
\begin{proposition}
\label{prop:F:to:sat}
 $\mathcal{A}^\omega[L,\mathrm{retr},\phi]+F^X \vdash\mathrm{Sat}$.    
\end{proposition}
\begin{proof}
We show that $\mathcal{A}^\omega[L,\mathrm{retr},\phi]+F^X \vdash\mathrm{Sat}_{-}$ and the result follows from Proposition \ref{prop:sat:int:to}. Let $\Theta_m(T,\underline{r},l,u,p,\underline{z},n)$ be a formula in $\mathcal{PBL}[L]$ and $\underline{z},l,p$ be given. Suppose we have
    \[
  \forall u^{X_s(0)}\, \exists n^0\, \neg\Theta^\mathrm{retr}_{m,\mathcal{A}}(T,\underline{r},l,u,p,n,\underline{z}).
  \]
  then 
  \[
  \forall r^0 \forall u^{X_s(0)}\, \exists n^0\, \neg\Theta^\mathrm{retr}_{m,\mathcal{A}}(T,\underline{r},l,\mathrm{retr}_s(u,P(r)+_\QQ1),p,n,\underline{z}).  \]
 Let $r^0$ be given. Then by Proposition \ref{prop:approx:univ}, there exists quantifier-free $\theta_{qf}(k,\mathrm{retr}_s(u,P(r)+_\QQ1),p,n,\underline{z})$ such that 
   \[
  \forall u^{X_s(0)}\, \exists n^0\, \exists k^0\,\neg\theta_{qf}(k,\mathrm{retr}_s(u,P(r)+_\QQ1),p,n,\underline{z}).
  \]
  Now, by $\QFAC$ we have 
   \[
 \exists \Phi_1^{0(X_s(0))}\,\exists \Phi_2^{0(X_s(0))}\,\forall u^{X_s(0)}\,\neg\theta_{qf}(\Phi_1(u),\mathrm{retr}_s(u,P(r)+_\QQ1),p,\Phi_2(u),\underline{z})
  \]
  which implies, from the proof of Proposition \ref{prop:approx:univ}, that
     \[
 \forall u^{X_s}\neg \Tilde{\Delta}^\mathrm{retr}_{m,\mathcal{A}}(T,\underline{r},l,\mathrm{retr}_s(u,P(r)+_\QQ1),p,\Phi_2(u),\underline{z},\Phi_1(u)).
  \]
Now, we have, from the axioms of $\mathrm{retr}$ that 
\[
\forall n^0\, (\mathrm{retr}_s(\mathrm{retr}_s(u(n),P(r)+_\QQ1),P(r)+_\QQ1) =_{X_s} \mathrm{retr}_s(u(n),P(r)+_\QQ1).
\]
Thus, $\QFER$ and $\QFER^0_\RR$ yield
   \[
 \forall u^{X_s}\neg \Tilde{\Delta}^\mathrm{retr}_{m,\mathcal{A}}(T,\underline{r},l,\mathrm{retr}_s(u,P(r)+_\QQ1),p,\Phi_2(\mathrm{retr}_s(u,P(r)+_\QQ1)),\underline{z},\Phi_1(\mathrm{retr}_s(u,P(r)+_\QQ1)))
  \]
and from $F^X$, we have $\exists n^* \forall u^{X_s}\,(\Phi_1(\mathrm{retr}_s(u,P(r)+_\QQ1)) \le_0 n^*\land \Phi_2(\mathrm{retr}_s(u,P(r)+_\QQ1)) \le_0 n^*)$. This implies
\[
\exists n^* \forall u^{X_s} \exists n,m \le_0 n^*\,\neg \Tilde{\Delta}^\mathrm{retr}_{m,\mathcal{A}}(T,\underline{r},l,\mathrm{retr}_s(u,P(r)+_\QQ1),p,n,\underline{z},m)
\]
and $(\dagger)$ (and the monotonicity of $\Delta^\mathrm{retr}_{m,\mathcal{A}}$ in its final argument) implies 
\[
\exists n^* \forall u^{X_s} \exists n \le_0 n^*\,\neg \Delta^\mathrm{retr}_{m,\mathcal{A}}(T,\underline{r},\mathrm{retr}_s(u,P(r)+_\QQ1),p,n,\underline{z},n^*)
\]
and the result follows by $\QFER$ and $\QFER^0_\RR$ since we have, provably, $\forall n^0\,(\mathrm{retr}_s(\mathrm{retr}_s(u,P(r)+2^{-n^*}),P(r)+_\QQ1)(n) =_{X_s} \mathrm{retr}_s(u,P(r)+2^{-n^*})(n))$.               
\end{proof}

\begin{theorem}\label{thm:metatheorem2}
 Let
 \begin{align*}
       \Theta_m(T,\underline{r},l,z^\rho,n^0):\equiv&\forall_{r_1(z,n)} x_1^{X_{s_1}(0)} \exists_{r_2(z,n)} y_1^{X_{s_2}(0)} \ldots \forall_{r_{2m-1}(z,n)} x_m^{X_{s_{2m-1}}(0)}\exists_{r_{2m}(z,n)} y_m^{X_{s_{2m}}(0)}\,\\&(T((\overline{\underline{x},l(z,n)}),(\overline{\underline{y},l(z,n)}),z,n)=_\RR 0)
 \end{align*}
be a formula in $\mathcal{PBL}[L]$ with $\rho$ admissible. If
\[
\mathcal{A}^\omega[L,\mathcal{T}]+F^X\vdash \forall z^\rho\, \exists n^0\,\neg\Theta_{m}(T,\underline{r},l,z,n)
\]
then one can extract a partial functional $\Phi:\mathcal{S}_{\widehat{\rho}}\rightharpoonup\mathbb{N}$ which is total and (bar-recursively) computable on $\mathcal{M}_{\widehat{\rho}}$ and the following holds in all non-trivial models $\mathscr{M}$ of $\mathcal{T}$: for all $z\in \mathcal{S}_\rho$ and $z^*\in \mathcal{S}_{\widehat{\rho}}$, if $z^*\gtrsim z$, then
\[
 \exists n\leq_0\Phi(z^*)\,\neg\Delta_{m,\mathcal{A}}(T,\underline{r},l,z,n,\Phi(z^*)).
\]
In particular,
\[
 \exists n\leq_0\Phi(z^*)\,\neg\Theta_{m,\mathcal{A}}(T,\underline{r},l,z,n).
\]
\end{theorem}
\begin{proof}
Theorem \ref{thrm:equiv:approxsat:sat} and Proposition \ref{prop:F:to:sat} imply
\[
\mathcal{A}^\omega[L,\mathcal{T}]+F^X\vdash \forall z^\rho\, \exists n^0\, \neg\Theta_{m,\mathcal{A}}(T,\underline{r},l,z,n).
\]
    Proposition \ref{prop:approx:univ} implies that there exists a quantifier -free formula $\theta_{qf}$ such that $\Theta_{m,\mathcal{A}}(T,\underline{r},l,z,n)$ is equivalent to $\forall k^0\theta_{qf}(k,z,n)$. Thus, Theorem \ref{thm:metatheorem} implies we can extract the appropriate functional $\Phi:\mathcal{S}_{\widehat{\rho}}\rightharpoonup\mathbb{N}$ such that for all non-trivial models $\mathscr{M}$ of $\mathcal{T}$ we have for all $z\in \mathcal{S}_\rho$ and $z^*\in \mathcal{S}_{\widehat{\rho}}$, if $z^*\gtrsim z$, then
\[
 \exists n,k \leq_0\Phi(z^*)\,\neg\theta_{qf}(k,z,n).
\]
But, from (the proof of) Proposition \ref{prop:approx:univ}, we have $\theta_{qf}(k,z,n)$ is equivalent to $\Tilde{\Delta}^\mathrm{retr}_{m,\mathcal{A}}(T,\underline{r},l,z,n,k)$. This implies,
\[
 \exists n,k \leq_0\Phi(z^*)\,\neg\Tilde{\Delta}^\mathrm{retr}_{m,\mathcal{A}}(T,\underline{r},l,z,n,k)
\]
Now, from the proof of Proposition \ref{prop:approx:univ}, we have
\[
\tilde{\Delta}^\mathrm{retr}_{m,\mathcal{A}}(T,\underline{r},l,z,n,k+2) \to \Delta^\mathrm{retr}_{m,\mathcal{A}}(T,\underline{r},l,z,n,k) \to \tilde{\Delta}^\mathrm{retr}_{m,\mathcal{A}}(T,\underline{r},l,z,n,k).
\]
and so
\[
 \exists n,k \leq_0\Phi(z^*)\,\neg\Delta^\mathrm{retr}_{m,\mathcal{A}}(T,\underline{r},l,z,n,k).
\]
Finally, from $(\star)$ in  Proposition \ref{prop:ext:approx} (and the monotonicity of $\Delta^D_{m,\mathcal{A}}(T,\underline{r},l,z,n,k)$ in the argument $k$) we obtain the result, as we have 
\[
 \exists n \leq_0\Phi(z^*)\,\neg\Delta^D_{m,\mathcal{A}}(T,\underline{r},l,z,n,\Phi(z^*)).
\]

\end{proof}
Our main result, a proof-theoretic analogue of Corollary \ref{cor:main:model:unif:sat}, immediately follows from Theorem \ref{thm:metatheorem2} and Proposition \ref{prop:F:to:sat}.
\begin{theorem}
\label{thm:meta:main}
     Let
 \begin{align*}
 \Theta_m(T,\underline{r},l,z^\rho,n^0):\equiv&\forall_{r_1(z,n)} x_1^{X_{s_1}(0)} \exists_{r_2(z,n)} y_1^{X_{s_2}(0)} \ldots \forall_{r_{2m-1}(z,n)} x_m^{X_{s_{2m-1}}(0)}\exists_{r_{2m}(z,n)} y_m^{X_{s_{2m}}(0)}\,\\&(T((\overline{\underline{x},l(z,n)}),(\overline{\underline{y},l(z,n)}),z,n)=_\RR 0)
 \end{align*}
be a formula in $\mathcal{PBL}[L]$ with  $\rho$ admissible. If
\[
\mathcal{A}^\omega[L,\mathcal{T}]+\mathrm{Sat}\vdash \forall z^\rho\, \exists n^0\,\neg\Theta_{m}(T,\underline{r},l,z,n)
\]
then one can extract a partial functional $\Phi:\mathcal{S}_{\widehat{\rho}}\rightharpoonup\mathbb{N}$ which is total and (bar-recursively) computable on $\mathcal{M}_{\widehat{\rho}}$ and the following holds in all non-trivial models $\mathscr{M}$ of $\mathcal{T}$: for all $z\in \mathcal{S}_\rho$ and $z^*\in \mathcal{S}_{\widehat{\rho}}$, if $z^*\gtrsim z$, then
\[
 \exists n\leq_0\Phi(z^*)\,\neg\Theta_{m,\mathcal{A}}(T,\underline{r},l,z,n).
\]
\end{theorem}

\begin{remark}
    We may have tuples instead of a single variable $z$, in Theorems \ref{thm:metatheorem2} and \ref{thm:meta:main}. As in Remark \ref{rem:F:t:sat}, similar results can be obtained in the context of normed structures (cf.\ Remarks \ref{rem:met:vs:normed} and \ref{rem:PBL:diff}) with an appropriate formalization of parametrised positive bounded formulas and additional extensionality requirements.
\end{remark}

\section{Application of metatheorem: A quantitative version of Theorem \ref{thrm:meta:main:strong}}
\label{sec:quant:main}
We shall now  provide an application of the metatheorem we gave in the previous section. We do this by extracting explicit bounds for Theorem \ref{thrm:meta:main:strong}, through analysing the model-theoretic proof we gave in Section \ref{subsec:model:APP2}, guided by the formalism of the previous section and Theorem \ref{thm:meta:main}, recalling that, as mentioned in Remark \ref{rem:how:to:use}, the processes of extracting the bounds from the proof are contained in the proofs of Propositions \ref{lem:ndinterpretation} and \ref{lem:majmainresult}. Concretely, we extract majorants for the Dialectica interpretation of the negative translation of the lemmas used in the model-theoretic proof of  Theorem \ref{thrm:meta:main:strong}. The key observation that allows us to perform this process smoothly is that, by following the proof of Proposition \ref{prop:approx:univ}, one can show that formulas containing only bounded quantifiers are equivalent to a quantifier-free statement. For example, in the context of embedding $\mathcal{L}_G$ in Section \ref{subsec:model:APP1} into the system, if we have a term $B^{X_{\mathcal{F}}}$ and $n^0$, the formula expressing that $B$ is $n$-generic is
\[
\exists g^{X_{\Omega}(0)}\left(d_{\mathcal{F}}\left(\bigcup_{i\le n}g(i)\cdot B,\hat{\Omega}\right) =0\right).
\]
This formula is equivalent to a quantifier-free formula (using the terms $\phi_s$, as in the proof of Proposition \ref{prop:approx:univ}) after noting that the axioms expressing that $d_{\mathcal{F}}$ is a discrete metric allow one to show that the matrix of the formula is equivalent to a quantifier-free formula. 

Throughout this section, fix $G,A,k, \mathcal{F}, \mu$ as in the statement of Theorem \ref{thrm:meta:main:strong}.
\begin{lemma}[Quantitative version of (the forward direction of) Lemma \ref{lem:pos:gen}]
\label{lem:pos:gen:quant}
    For all $B\in \mathcal{F}$ and $m,p \in \NN$, if $B$ is $p$-stable then, $\mu(B) > 2^{-m}$ implies $B$ is $((p-1)2^{2m+5})$-(left and right ) generic.
\end{lemma}
\begin{proof}
    The result is immediate from the proof of Theorem 5.1 in \cite{conant2021quantitative}.
\end{proof}
Lemma \ref{lem:pos:gen:quant} provides  majorants for the Dialectica interpretation of the negative translation of Lemma \ref{lem:pos:gen} (after recognising formulas with bounded quantification as quantifier-free, as discussed above, and expanding hidden quantifiers in the relationships between the real numbers). Theorem \ref{thrm:packing} is already quantitative. The next results we need are quantitative versions of Lemmas \ref{lem:bound:symm:nonquant} and \ref{lem:fin:orbit:quant}, whose forms are obtain in a similar manner to Lemma \ref{lem:pos:gen:quant}. 
\begin{lemma}[Quantitative version of Lemma \ref{lem:bound:symm:nonquant}]
\label{lem:bound:symm:quant}
    for every $f :\NN \to \NN$ there exists $m \le \Phi(f)$ such that for all $g, h \in G$
        \[
    \mu(Ag \triangle Ah) > 2^{-f(m)} \to \mu(Ag \triangle Ah)>2^{-m}
    \]
    where 
    \[
    \Phi(f):= \Tilde{f}^{(s)}(0)
    \]
    with $\Tilde{f}(n):= n+5 +2f_M(n+1)+\lceil \log_2(s-1)\rceil$, $\Tilde{f}^{(k)}$ denotes the $kth$ iterate of $\Tilde{f}$, $f_M(n):= \max_{i\le n}f(n)$, and $s:= R(R(k, k + 1), R(k, k + 1))+ 1$.
\end{lemma}
\begin{proof}
    The result and proof are essentially given in Lemma 3.4 of \cite{conant2021quantitative}. We present a full proof here to show how it one obtains this result from an analysis of Lemma \ref{lem:bound:symm:nonquant}.

As in Lemma \ref{lem:bound:symm:nonquant}, write $\phi(x,(y,z))$ for the relation $x \in Ay \triangle Az$, which is $s$-stable by Lemma \ref{lem:stable:symm}. We again argue by contradiction. Suppose for each $m \le \Phi(f)$, there exists $g,h \in G$ such that $2^{-m}\ge \mu(Ag \triangle Ah) > 2^{-f(m)}$. 
By induction on $1 \le n \le s$,  for $1 \le i \le n$ we will construct $g_i,h_i \in G$, such that for all $1 \le t \le n$
\[
\mu\left(\bigcap_{i=1}^t Ag_i \triangle Ah_i \cap \bigcap_{i=t+1}^n(Ag_i \triangle Ah_i)^c\right)>2^{-\Tilde{f}^{(n)}(0)}.
\]
and as in the proof of  Lemma \ref{lem:bound:symm:nonquant}, this will  contradict the stability of $\phi$. For the case $n =1$, by our assumption, with $m=0$, we can take $g,h \in G$ such that $\mu(Ag \triangle Ah) >2^{-f(0)}$ and the result follows since $\Tilde{f}(0) \ge f(0)$. Now fix $1\le n < s$ and suppose we have constructed $g_i, h_i$ with $1 \le i \le n$ satisfying the desired condition, so for all $1 \le t \le n$
\[
\mu\left(\bigcap_{i=1}^t Ag_i \triangle Ah_i \cap \bigcap_{i=t+1}^n(Ag_i \triangle Ah_i)^c\right)> 2^{-\Tilde{f}^{(n)}(0)} .
\]
Since $n < s$, we have $\Tilde{f}^{(n)}(0) < \Phi(f)$ and so, by assumption, we can take $g,h \in G$ satisfying $2^{-(\Tilde{f}^{(n)}(0)+1)}\ge \mu(Ag \triangle Ah) >2^{-f(\Tilde{f}^{(n)}(0)+1)}$. Thus, Lemma \ref{lem:pos:gen:quant} implies that $Ag \triangle Ah$ is $q$-generic, where $q:=2^{2f(\Tilde{f}^{(n)}(0)+1)+5 +\lceil\log_2(s-1)\rceil}$. Setting 
\[
B:= \bigcap_{i=1}^n Ag_i \triangle Ah_i
\]
which will have measure greater than $2^{-\Tilde{f}^{(n)}(0)}$, we will have, for some $v_1, \ldots, v_q \in G$ that 
\[
B = \bigcup_{i=1}^q (B \cap (Ag \triangle Ah)v_i)
\]
 Thus we can take some $v \in G$ such that 
 \[
 \mu(B \cap Agv \triangle Ahv)=\mu(B \cap (Ag \triangle Ah)v)\ge 2^{-(\Tilde{f}^{(n)}(0)+ 2f(\Tilde{f}^{(n)}(0)+1)+5 +\lceil\log_2(s-1)\rceil)}\ge  2^{-\Tilde{f}^{(n+1)}(0)}.
 \]
 Setting $g_{n+1} = gv$ and $h_{n+1} = hv$ yields that 
\[
\mu\left(\bigcap_{i=1}^{n+1} Ag_i \triangle Ah_i \right)> 2^{-\Tilde{f}^{(n+1)}(0)}
\]
and for $1 \le t \le n$ we have 
\begin{equation*}
    \begin{aligned}
        \mu\left(\bigcap_{i=1}^t Ag_i \triangle Ah_i \cap \bigcap_{i=t+1}^{n+1}(Ag_i \triangle Ah_i)^c\right)&\ge \mu\left(\bigcap_{i=1}^t Ag_i \triangle Ah_i \cap \bigcap_{i=t+1}^n(Ag_i \triangle Ah_i)^c\right) - \mu\left(Ag_{n+1}\triangle Ag_{n+1}\right)\\
        & \ge 2^{-\Tilde{f}^{(n)}(0)} - \mu\left((Ag\triangle Ah)v\right) \ge 2^{-(\Tilde{f}^{(n)}(0)+1)}> 2^{-\Tilde{f}^{(n+1)}(0)} .
    \end{aligned}
\end{equation*}    
\end{proof}
\begin{lemma}[Quantitative version of Lemma \ref{lem:fin:orbit:nonquant}]
\label{lem:fin:orbit:quant}
    For all $e: \NN \to \NN$ there exists $n \le \Psi(e)$ such that for all $(g_i) \subseteq G$ there exists $i < j \le n$ satisfying 
    \[
    \mu(Ag_i \triangle Ag_j) \le 2^{-e(n)}
    \]
    where 
    \[
    \Psi(e):=\omega_{k-1}(\Phi(f_e))
    \]
    with $f_e(n):= e(\omega_{k-1}(n))$ and $\omega$ defined as in Theorem \ref{thrm:packing} and $\Phi$ as in Lemma \ref{lem:bound:symm:quant}.
\end{lemma}
\begin{proof}
    As in the proof of Lemma \ref{lem:fin:orbit:nonquant}, we argue by contradiction. Suppose we have $e: \NN \to \NN$ such that for all $n \le \Psi(e)$ there exists $(g_i) \subseteq G$ such that for all $i<j\le n$ we have $\mu(Ag_i \triangle Ag_j) > 2^{-e(n)}$. By Lemma \ref{lem:bound:symm:quant}, we can take $m \le \Phi(f_e)$ such that for all $g, h \in G$
        \[
    \mu(Ag \triangle Ah) > 2^{-f_e(m)} \to \mu(Ag \triangle Ah)>2^{-m}.
    \]
    So we can set $n:= \omega_{k-1}(m) \le \omega_{k-1}(\Phi(f_e)) = \Psi(e)$ and so for all $i<j\le n$ we have 
    \[
    \mu(Ag_i \triangle Ag_j) > 2^{-e(n)} = 2^{-e(\omega_{k-1}(m))} = 2^{-f_e(m)}
    \]
    so, $\mu(Ag_i \triangle Ag_j) >  2^{-m}$. Theorem \ref{thrm:packing} and Theorem \ref{thrm:stab:vc} implies $n < \omega_{k-1}(m)$, a contradiction.
\end{proof}
We can now  provide explicit bounds for Theorem  \ref{thrm:meta:main:strong}.
\begin{theorem}
    For $k \ge 1$ and $f:\NN \times \NN \to \NN$, there exists  $n \le \Psi_1(f) $ and $m \le \Psi_2(f)$ such that  $\mathrm{Stab}_m(A) \in \mathcal{F}$, $\mathrm{Stab}_m(A) \le G$, and $\mathrm{Stab}_m(A)$ has index at most $n$. In addition,  there exists a subset $Y \subseteq G$, which is a union of cosets of $\mathrm{Stab}_m(A)$, such that $\mu(A \triangle Y ) \le 2^{-f(n,m)}$. Where,
    \[
    \Psi_2(f) = \Phi(e_f) \mbox{ and } \Psi_1(f):=\Psi\left(C_{\Psi_2(f)}\right) 
    \]
    with  
    \[
    e_f(m) =
    \max\left\{ 3\max_{n \le \Psi(C_m)}(f(n,m))+3\lceil\log_2 (\Psi(C_m))\rceil+\lceil\log_2(k-1)\rceil+5,\;m+1\right\}.
   \]

    and  for all $m \in \NN$, $C_m : \NN \to \NN$ is the constant function taking value $m$. Here, $\Phi$ and $\Psi$ are as in Lemmas \ref{lem:bound:symm:quant} and \ref{lem:fin:orbit:quant} respectively. 
\end{theorem}

\begin{proof}
    Lemma \ref{lem:bound:symm:quant} implies that for all $f: \NN \to \NN$, there exists $m \le \Phi(f)$ such that $\mathrm{Stab}_m(A) \subseteq \mathrm{Stab}_{f(m)}(A)$. Furthermore, if $f(m) \ge m+1$ we will have 
    \[
    \mathrm{Stab}_m(A)\mathrm{Stab}_m(A)\subseteq \mathrm{Stab}_{f(m)}(A)\mathrm{Stab}_{f(m)}(A) \subseteq \mathrm{Stab}_{f(m)-1}(A)\subseteq \mathrm{Stab}_{m}(A),
    \]
    here we use the fact that for all $n \in \NN$, $\mathrm{Stab}_{n}(A)\mathrm{Stab}_{n}(A) \subseteq \mathrm{Stab}_{n-1}(A)$. Thus, for such an $f$ and $m$, we have $\mathrm{Stab}_m(A)$ is closed under the group operation. Clearly $\mathrm{Stab}_m(A)$ contains the identity and is closed under inverses, thus $\mathrm{Stab}_m(A) \le G$. Furthermore, as in the proof of Theorem \ref{thrm:main:nonquant:strong}, Lemma \ref{lem:Stab:in:F} implies $\mathrm{Stab}_m(A) \in \mathcal{F}$.

    Now, let $f:\NN \times \NN \to \NN$ be given. By the above, we can take $m \le \Phi(e_f)= \Psi_2(f)$ such that $H:= \mathrm{Stab}_m(A)$ is a subgroup of $G$. Now, by Lemma \ref{lem:fin:orbit:quant}, we have $n \le \Psi(C_m) \le \Psi_1(f)$ such that, for all $(g_i) \subseteq G$ there exists $i<j\le n$ with $\mu(Ag_i \triangle Ag_j) \le 2^{-m}$. Suppose we have distinct cosets $g_0H,\ldots, g_nH$, then for each $i<j\le n$, we have $g_ig_j^{-1} \notin H$, which implies (by the right invariance of $\mu$) that $\mu(Ag_i \triangle Ag_j) > 2^{-m}$, a contradiction. Thus, $H$ has index at most $n \le \Psi_1(f)$.

Now, we have that for all $g \in G$, either $\mu(gH \cap A) \le 2^{-f(n,m)}/n$ or  $\mu(gH \cap A^c) \le 2^{-f(n,m)}/n$. If not, setting $B:=H \cap g^{-1}A$ and $C:=H \cap g^{-1}A^c$, would imply $\mu(B), \mu(C) >  2^{-(f(n,m)+\lceil\log_2(n)\rceil)}$. Thus, Lemma \ref{lem:pos:gen:quant} implies $B$ is $((k-1)2^{2(f(n,m)+\lceil\log_2(n)\rceil)+5})$-generic and so since $\mu(C)>2^{-(f(n,m)+\lceil\log_2(n)\rceil)}$ there exists $h \in G$ such that
\[
\mu(Bh \cap C)> \frac{2^{-(3(f(n,m)+\lceil\log_2(n)\rceil) + 5)}}{k-1}.
\]

As in the proof of Theorem \ref{thrm:main:nonquant:strong}, for such an $h$, we have $Bh \cap C = H \cap Hh \cap g^{-1}(Ah \cap A^c)$ which implies  $h \in H$ and 
\[
\mu(Ah \cap A^c)=\mu(g^{-1}(Ah \cap A^c))> \frac{2^{-(3(f(n,m)+\lceil\log_2(n)\rceil) + 5)}}{k-1}.
\]
But, $h \in H  = \mathrm{Stab}_{e_f(m)}(A)$ implies
\[
\mu(Ah \cap A^c) \le \mu(Ah \triangle A)\le 2^{-e_f(m)} \le\frac{2^{-(3(f(n,m)+\lceil\log_2(n)\rceil) + 5)}}{k-1}
\]
a contradiction.

    Now let $C_1,\ldots C_r$ be the cosets of $H$, so $r \le n$. Let $I$ be the set of $1 \le i \le r$ such that $\mu(C_i \cap A) > 2^{-f(n,m)}/n$, so if $i \in I$ we have $\mu(C_i \cap A^c) \le 2^{-f(n,m)}/n$. Set 
    \[
    Y:= \bigcup_{i\in I}C_i.
    \]
    So, the result follows since
    \begin{equation*}
        \begin{aligned}
            \mu(A \triangle Y) = \mu(A \cap Y^c) + \mu(A^c \cap Y)
            = \mu\left(\bigcup_{i \notin I} C_i \cap A\right)+ \mu\left(\bigcup_{i \in I} C_i \cap A^c\right)\le  2^{-f(n,m)}
        \end{aligned}
    \end{equation*}

\end{proof}
\noindent
{\bf Acknowledgments:} The authors are grateful to Nicholas Pischke for his helpful comments.

\end{document}